\documentclass[12pt,a4paper]{amsart}
    
\usepackage{fullpage}
\makeatletter

\usepackage[english]{babel}
\usepackage[utf8]{inputenc}
 \usepackage{graphicx}
\usepackage[square,sort,comma,numbers]{natbib}
\usepackage{stmaryrd}

\usepackage{amsfonts}
\usepackage{amsthm}
\usepackage{amsmath,amssymb}
\usepackage{xparse}
\usepackage{tikz-cd}
\usepackage{verbatim}
\usepackage{mathtools}

\title{DGAs with polynomial homology}
\author{Haldun Özgür Bayındır}

\newcommand{\cof}{\rightarrowtail}

\newcommand{\bwe}{\smash{\rlap{\kern 8.5pt\raise 4pt\hbox{\footnotesize $\sim$}}}\longleftarrow}

\def\longfib{\DOTSB\relbar\joinrel\twoheadrightarrow}

\newcommand{\trfib}{\smash{\rlap{\kern 7pt\raise 4pt\hbox{\footnotesize $\sim$}}}\longfib}
\newcommand{\hmod}{\mhyphen\textup{modules}}
\newcommand{\trcof}{\smash{\rlap{\kern 5.5pt\raise 4pt\hbox{\footnotesize $\sim$}}}\cof}

\newcommand{\wdg}{\wedge}
\newcommand{\ot}{\otimes}
\newcommand{\wdgz}{\wedge_{H\mathbb{Z}}}
\newcommand{\wdgfp}{\wedge_{H\mathbb{F}_p}}
\newcommand{\wdgft}{\wedge_{H\mathbb{F}_2}}
\NewDocumentCommand{\tens}{t_}
 {%
  \IfBooleanTF{#1}
   {\tensop}
   {\otimes}%
 }
\NewDocumentCommand{\tensop}{m}
 {%
  \mathbin{\mathop{\otimes}\displaylimits_{#1}}%
 }
\DeclareMathOperator{\Hom}{Hom}

\DeclareMathOperator{\hh}{\ensuremath{\textup{HH}}}
\DeclareMathOperator{\ho}{\ensuremath{\textup{Ho}}}
\DeclareMathOperator{\modu}{\ensuremath{\textup{-}\mathsf{Mod}}}

\newcommand{\hhfps}{\textup{HH}^{\mathbb{F}_p}_*}
\DeclareMathOperator{\hhfts}{\ensuremath{\textup{HH}^{\mathbb{F}_2}_*}}

\DeclareMathOperator{\hhfpn}{\ensuremath{\textup{HH}^{\mathbb{F}_p}_n}}
\DeclareMathOperator{\thh}{\ensuremath{\textup{THH}}}
\DeclareMathOperator{\aut}{\ensuremath{\textup{Aut}}}
\newcommand*{\rom}[1]{\expandafter\@slowromancap\romannumeral #1@}
\newcommand{\we}{\smash{\rlap{\kern 3.8pt\raise 4pt\hbox{\footnotesize $\sim$}}}\rightarrow}

\newcommand{\overbar}[1]{\mkern 1.5mu\overline{\mkern-1.5mu#1\mkern-1.5mu}\mkern 1.5mu}

\newcommand{\arrowcat}{\mathcal{E}_{H\F_p}^{\cdot \to \cdot}}
\newcommand{\dsa}{\mathcal{A}_*}
\newcommand{\whz}{\wedge_{H\Z}}
\newcommand{\Z}{\mathbb{Z}}

\newcommand{\F}{\mathbb{F}}

\newcommand{\lv}{\lvert}
\newcommand{\hz}{H\mathbb{Z}}

\newcommand{\hfp}{H\mathbb{F}_p}
\newcommand{\hft}{H\mathbb{F}_2}

\newcommand{\obhfp}{\overbar{H\mathbb{F}}_p}
\newcommand{\rv}{\rvert}
\DeclareMathOperator{\tor}{\ensuremath{\textup{Tor}}}

\newcommand{\fp}{\mathbb{F}_p}
\newcommand{\ft}{\mathbb{F}_2}
\newcommand{\torup}{\textup{Tor}}
\newcommand{\lambdatau}{\Lambda_{\F_p}(\tau_0)}
\newcommand{\lambdafp}{\Lambda_{\F_p}}
\newcommand{\lambdaft}{\Lambda_{\F_2}}
\newcommand{\nc}{\newcommand}
\nc{\C}{\mathcal{C}}

\nc{\z}{\mathbb{Z}}
\nc{\PP}{\mathbb{P}}
\nc{\R}{\mathbb{R}}
\nc{\f}{\mathbb{F}}
\nc{\pis}{\pi_*}
\newcommand{\modmod}[3][{}]{{#2}\sslash_{#1}{#3}}
\nc{\sph}{\mathbb{S}}

\nc{\etw}{E_2}

\newcommand{\Sp}{\mathbb{S}}
\mathchardef\mhyphen="2D

\newtheorem{theorem}{Theorem}[section]
\newtheorem{lemma}[theorem]{Lemma}

\newtheorem{corollary}[theorem]{Corollary}
\newtheorem{proposition}[theorem]{Proposition}

\theoremstyle{definition}

\newtheorem{definition}[theorem]{Definition}
\newtheorem{example}[theorem]{Example}
\newtheorem{construction}[theorem]{Construction}
\newtheorem{remark}[theorem]{Remark}

 \newtheoremstyle{TheoremNum}
        {\topsep}{\topsep}              
        {\itshape}                      
        {}                              
        {\bfseries}                     
        {.}                             
        { }                             
        {\thmname{#1}\thmnote{ \bfseries #3}}
    \theoremstyle{TheoremNum}
    \newtheorem{thmn}{Theorem}

\def\co{\colon\thinspace}

\date{}

\usepackage{natbib}
\begin{document}

\maketitle

\begin{abstract}
In this work, we study the classification of differential graded algebras over $\mathbb{Z}$ (DGAs) whose homology is $\mathbb{F}_p[x]$, i.e.\ the polynomial algebra over $\mathbb{F}_p$ on a single generator. This  classification problem was left open in work of Dwyer, Greenlees and Iyengar.

For $\lvert y_{2p-2} \rvert = 2p-2$, we show that there is a unique non-formal DGA with homology $\mathbb{F}_p[y_{2p-2}]$ and a non-formal $2p-2$ Postnikov section. Among a classification result, this provides the first example of a non-formal DGA with homology $\mathbb{F}_p[x]$. By duality, this also shows that there is  a non-formal DGA whose homology is an exterior algebra over $\mathbb{F}_p$ with a generator in degree $-(2p-1)$.

Considering the classification of the ring spectra corresponding to these DGAs, we show that every $E_2$ DGA with homology $\mathbb{F}_p[x]$ (with no restrictions on $\lvert x \rvert$) is topologically equivalent to the formal DGA with  homology $\mathbb{F}_p[x]$, i.e.\ they are topologically formal. This follows by a theorem of Hopkins and Mahowald.

\end{abstract}

\section{Introduction}

This paper concerns the classification of DGAs. An interesting classification result is given by Dwyer, Greenlees and Iyengar in \cite{dwyer2013dg}. Namely, they provide a complete classification of DGAs with homology $\Lambda_{\F_p}(x_{-1})$, i.e.\ the exterior algebra over $\F_p$ with a single generator in degree $-1$. This classification result states that there is a bijection between quasi-isomorphism classes of DGAs with homology $\Lambda_{\F_p}(x_{-1})$ and isomorphism classes of complete discrete valuation rings with
residue field $\F_p$. Indeed, there is also a one to one correspondence between  topological equivalence classes of DGAs with homology $\Lambda_{\F_p}(x_{-1})$ and isomorphism classes of complete discrete valuation rings with residue field $\F_p$, see Theorem \ref{thm topeqCDVR}.

Classifying the Postnikov extensions of $\F_p$ in DGAs, Dugger and Shipley obtain a classification of quasi-isomorphism classes of  DGAs with homology $\Lambda_{\F_p}(x_n)$ for $\lv x_n \rv = n >0$ \cite{dugger2007topological}. They show that there is a unique DGA with homology $\Lambda_{\F_p}(x_n)$ for odd $n$ and there are precisely two distinct DGAs with homology $\Lambda_{\F_p}(x_n)$ for even $n$. Also, it is known that there is a unique DGA with homology $\Lambda_{\F_p}(x_0)$. This result together with the result of Dwyer, Greenlees and Iyengar provides a complete classification of DGAs with homology $\Lambda_{\F_p}(x_n)$ for $n \geq -1$.

Dwyer, Greenlees and Iyengar also work on the classification of DGAs with homology $\Lambda_{\F_p}(x_n)$ for $n<-1$. Using a Moore-Koszul duality argument, they show that for $n<-1$ there is a bijection between quasi-isomorphism classes of DGAs with homology $\Lambda_{\F_p}(x_n)$  and quasi-isomorphism classes of DGAs with homology $\F_p[x_{-n-1}]$, i.e.\ the polynomial algebra over $\F_p$ with a single generator in degree $-n-1$. The classification of DGAs with homology $\F_p[x]$ for $\lv x \rv >0$ is left as an open question in \cite{dwyer2013dg}. In this work, we study this problem. 

Note that when we say DGAs, we mean differential graded algebras over $\Z$. For $\F_p$-DGAs, the answer to this problem is obvious. There is a unique $\F_p$-DGA with homology $\F_p[x]$. This is the free DGA over $\Sigma^{\lv x \rv} \F_p$. For a proof of this fact, see Section \ref{sec subsection on proof of the e2 theorems}. 

Our methods heavily rely on the topological equivalences of DGAs. The idea of topological equivalences is to use ring spectra to obtain new equivalences between DGAs. This provides us new tools for our calculations. Namely, this allows us to use  our knowledge of the B\"okstedt spectral sequence calculating $\text{THH}(H\F_p)$ to obtain the  Hochschild cohomology groups that classify the Postnikov extensions we need.

The homotopy category of DGAs is equivalent to the homotopy category of $H\Z$-algebras \cite{stanley1997dissertation}. To move from DGAs to ring spectra, we use the zig-zag of Quillen equivalences between DGAs and $H\Z$-algebras \cite{shipley2007hz}.  For a DGA $X$, we denote the corresponding $H\Z$-algebra by $HX$. Furthermore, we omit the forgetful functor from $H\Z$-algebras to $\Sp$-algebras and denote the underlying $\Sp$-algebra of $HX$  by $HX$. Note that we denote the sphere spectrum by $\sph$.
\begin{definition}
 Two DGAs $X$ and $Y$ are said to be \textbf{topologically equivalent} if $HX$ and $HY$ are weakly equivalent as $\Sp$-algebras. 
\end{definition}
The definition of topological equivalences is due to Dugger and Shipley \cite{dugger2007topological}. Note that because $H\Z$-algebras is Quillen equivalent to DGAs, two DGAs $X$ and $Y$ are quasi-isomorphic  precisely when $HX$ and $HY$ are equivalent as $H\Z$-algebras. Furthermore, the forgetful functor from $H\Z$-algebras to $\Sp$-algebras preserves weak equivalences. This shows that quasi-isomorphic DGAs are always topologically equivalent. However, the main result of \cite{dugger2007topological} shows that the opposite direction of this statement is not true. Namely, Dugger and Shipley show that there are examples of DGAs that are topologically equivalent but not quasi-isomorphic. This is explained in the following example. 

A DGA is said to be \textbf{formal} if it is quasi-isomorphic to a DGA with trivial differentials. Similarly, we say a DGA is $\textbf{topologically formal}$ if it is topologically equivalent to a formal DGA. Note that given a graded ring $A$, there is a unique formal DGA with homology $A$.

\begin{example}\cite[Section 5]{dugger2007topological} \label{ex duggershipley} 
Up to quasi-isomorphism, there are only two DGAs with homology  $\Lambda_{\F_p}(x_{2p-2})$ where $\lv x_{2p-2} \rv = 2p-2$. One of these DGAs is the formal one. For the rest of this work, we denote the  non-formal DGA with homology $\Lambda_{\F_p}(x_{2p-2})$ by $Y$. For $p=2$, $Y$ is given by
\[\Z[e_1 \lv de_1 =2]/(e_1^4)\]
where $\lv e_1 \rv = 1$. 

 The interesting point of this example is that $Y$ and the formal DGA with homology $\Lambda_{\F_p}(x_{2p-2})$ are topologically equivalent although they are not quasi-isomorphic. In other words, $Y$ is topologically formal but it is not formal. 
\end{example}

We show that the following result is a  consequence of a theorem of Hopkins and Mahowald. Note that when we say an $\sph$-algebra is an $E_2$ $\sph$-algebra, we mean that the the corresponding object in the $\infty$-category of $\sph$-algebras is an $E_2$ $\sph$-algebra.

\begin{theorem} \label{thm e2dgassinglegenerator}
Let $X$ be a DGA with homology $\F_p[x]$. If the $\Sp$-algebra corresponding to $X$ is an $E_2$ $\Sp$-algebra, then $X$ is topologically formal. 
\end{theorem}

For homology with multiple polynomial generators, we have the following result.

\begin{theorem}\label{thm e2dgas}
Let $X$ be a DGA whose homology is a polynomial algebra over $\F_p$ with three or less generators which all lie in even degrees. If the $\Sp$-algebra corresponding to $X$ is an $E_2$ $\Sp$-algebra, then $X$ is topologically formal.
\end{theorem}


\begin{remark}\label{rmk e2dgas result in e2dgas in the infty category}
The second hypothesis in the theorems above is satisfied if $X$ is quasi-isomorphic to an $E_2$ DGA (in the model categorical setting or in the $\infty$-categorical setting).  Indeed, the underlying $\infty$-category of the model category of $E_n$ DGAs is equivalent to the $\infty$-category of   $E_n$-algebras in the  $\infty$-category of chain complexes \cite[4.1.1]{hinich2015rectificationofalgebrasandmodules}. Furthermore, there is an equivalence between the $\infty$-categories of $E_n$ DGAs and $E_n$ $\hz$-algebras induced by Shipley's zig-zag of symmetric monoidal Quillen equivalences between the model categories of chain complexes and $\hz$-modules \cite[4.4]{peroux2020coalgebras}. This equivalence ensures that the object in the $\infty$-category of $\sph$-algebras  corresponding to an $E_2$ DGA is an $E_2$ $\sph$-algebra. See the discussion at the beginning of Section \ref{sec hopkins mahowald} for more details. 
 \end{remark}

 We prove the following classification result for coconnective DGAs with exterior homology. 
 
 \begin{theorem}\label{thm exterior e2 DGAs}
For $\lv x \rv <0$, every $E_2$ DGA with homology $\Lambda_{\F_p}[x]$ is formal as a DGA.
\end{theorem}
Furthermore, we apply this theorem to the examples of non-formal DGAs with homology $\Lambda_{\F_p}(x_{-1})$ constructed in \cite{dwyer2013dg}. 
 
  \begin{corollary}\label{cor not e2}
  Let $X$ be a DGA with homology $\Lambda_{\fp}(x_{-1})$. If  $X$  is not formal then $X$ is not an $E_2$ DGA.
 Equivalently, for a complete discrete valuation ring  $A$ with residue field $\fp$ and characteristic different than $p$, the derived endomorphism DGA $\text{End}_{A\text{-mod}}(\F_p)$ is not an $E_2$ DGA. 
 \end{corollary}

To obtain classification results without the $E_2$ assumption, we attack this  problem by calculating the quasi-isomorphism types of  Postnikov extensions which are classified by Hochschild cohomology groups. The main ingredient for these calculations is our knowledge of the topological equivalence type of $Y$ which allows us to carry out the Hochschild cohomology calculations. This is interesting because we obtain a purely algebraic result on DGAs by using ring spectra and topological equivalences of DGAs. This result is stated in the following theorem. Note that in the theorem below when we say unique, we mean unique up to quasi-isomorphisms.

\begin{theorem}\label{thm nontrivial}
 There is a unique DGA whose homology is  $\F_p[x_{2p-2} ]$ (with $\lv x_{2p-2} \rv = 2p-2$) and whose $2p-2$ Postnikov section is non-formal. Furthermore for  every $m>1$, there is a unique DGA whose  homology is $\F_p[x_{2p-2} ]/(x_{2p-2}^m)$ and whose  $2p-2$ Postnikov section is non-formal. 
\end{theorem}

\begin{theorem} \label{thm hochschild}
Let $X$ be the non-formal DGA with homology $\F_p[x_{2p-2} ]$ given in Theorem \ref{thm nontrivial}. We have  \[\textup{HH}^{\Z}_*(X,\F_p) = \F_p[\mu]/(\mu^p) \ \text{with} \ \lvert \mu \rvert = 2.\]
\end{theorem}
It is interesting that we obtain these simple Hochschild homology groups. For example, for the formal DGA $Z$ with homology $\F_p[x_{2p-2}]$, we have  $\text{HH}^{\Z}_*(Z, \F_p) \cong \Gamma(y) \otimes \Lambda(z)$ where  $\lvert y \rvert =2$ and $\lvert z \rvert = 2p-1$, see Remark \ref{rmk hh of formal}. 

We use the B\"okstedt spectral sequence to calculate $\text{HH}^{\Z}_*(X,\F_p)$ where $X$ is as in Theorem \ref{thm hochschild}. Using topological equivalences, we show that the non-trivial differentials in the B\"okstedt spectral sequence calculating $\text{THH}_*(\F_p)$ carry into the spectral sequence calculating $\text{HH}^{\Z}_*(X,\F_p)$ and cause the cancellations that provides this simple result. Indeed, this is the way we carry out our Hochschild cohomology calculations to classify the Postnikov extensions at each level of the Postnikov tower of $X$.

\begin{example} \label{ex nontrivial}
Among other things, Theorem \ref{thm nontrivial} provides a  non-formal DGA with   homology $\F_p[x_{2p-2}]$. This DGA is not formal because its $2p-2$ Postnikov section is $Y$ which is not formal.
 Proposition 6.1 of \cite{dwyer2013dg} states that there is a bijection between quasi-isomorphism classes of DGAs with homology $\F_p[x_{2p-2}]$ and quasi-isomorphism classes of DGAs with homology $\Lambda_{\F_p}(z_{-(2p-1)})$ for $\lv z \rv = -(2p-1)$. In particular, we obtain that there is a non-formal DGA with homology $\Lambda_{\F_p}(z_{-(2p-1)})$.

\end{example}

\textbf{Organization} In Section \ref{sec hopkins mahowald} we prove Corollary \ref{cor not e2}, Theorems  \ref{thm e2dgassinglegenerator}, \ref{thm e2dgas}, \ref{thm exterior e2 DGAs} and \ref{thm topeqCDVR}. The rest of this work is independent from Section \ref{sec hopkins mahowald} and it is devoted to the proof of Theorems \ref{thm nontrivial} and \ref{thm hochschild}.    Section \ref{sec Postnikov} contains a review of  $k$-invariants for Postnikov extensions of ring spectra. In Section \ref{sec pstnkv and hochschild}, we use standard results to show that the $k$-invariants that we are interested in lie in certain Hochschild homology groups and discuss a few well known calculational tools for Hochschild homology. In Section  \ref{sec differentials of the bokstedt ss}, we compute the relevant differentials in the B\"okstedt spectral sequences of interest.  Section \ref{sec proof of the main theorems} is devoted to the proof of   
Theorems \ref{thm nontrivial} and  \ref{thm hochschild}. In Section \ref{sec homology lemmas}, we prove technical lemmas that provide the ring structures on the $\hfp$-homology of the $\hz$-algebras of interest. These technical lemmas are used in Sections \ref{sec differentials of the bokstedt ss} and \ref{sec proof of the main theorems}.

\textbf{Notation} In our notation, we don not distinguish between a DGA and the corresponding $H\Z$-algebra in most situations, i.e.\ we omit the functor $H$.

In Section \ref{sec Postnikov} and after, we work in the setting of  EKMM spectra \cite{elmendorf2007rings} but our final results do not depend on the model of spectra we use.  Most of the smash products we mention represent the corresponding derived smash products due to our q-cofibrancy assumptions. In the situations where the given smash product do not necessarily represent the derived smash product, this
will be explicitly indicated.

\textbf{Acknowledgments} I would like to thank Irakli Patchkoria for suggesting the nilpotence theorem of Hopkins and Mahowald for our classification problem. Also, I  thank Jeremy Hahn for helpful conversations. 
\section{Proof of the classification  results for $E_2$-algebras}\label{sec hopkins mahowald}
This section is devoted to the proof of our results that involve classification of DGAs with $E_2$-structure, i.e.\ we prove  Theorems  \ref{thm e2dgassinglegenerator}, \ref{thm e2dgas} and \ref{thm exterior e2 DGAs} as well as Corollary \ref{cor not e2}. At the end of this section, we prove Theorem \ref{thm topeqCDVR} which improves the main result of \cite{dwyer2013dg} to a classification of the topological equivalence classes of the DGAs considered in \cite{dwyer2013dg}.

In this section, we work in the $\infty$-categorical setting unless we state otherwise. On the other hand,  when we talk about spectra in a model categorical setting, we use the stable model structure on symmetric spectra in simplicial sets \cite{Hovey-Shipley-Smith}. It follows by \cite[4.1.8.4]{lurie2012higher} and \cite[3.4.2]{pavlov2019symmetricopSsp}  that we do not need to distinguish between the model categorical setting and the $\infty$-categorical setting when we are studying   equivalence classes of $\sph$-algebras or when we are studying equivalence classes of $Hk$-algebras where $k$ denotes a discrete commutative ring.

There are several reasons why we work in  $\infty$-categories. For instance, there is a zig-zag of symmetric monoidal Quillen equivalences between the model categories of chain complexes and $H\Z$-modules which provides a zig-zag of Quillen equivalences between the model categories of DGAs and $H\Z$-algebras \cite{shipley2007hz}.  However, we need a zig-zag of Quillen equivalences between the model categories of $E_n$ DGAs and $E_n$ $H\Z$-algebras for $n=2$ but this is not available in the literature; $n=\infty$ case is due to Richter and Shipley \cite{richter2014algebraic}.  

On the other hand, Shipley's zig-zag of symmetric monoidal Quillen equivalences induce an equivalence of symmetric monoidal $\infty$-categories between the $\infty$-categories of chain complexes and $H\Z$-modules \cite[4.4]{peroux2020coalgebras}. This induces the desired equivalence between the $\infty$-categories of $E_n$ $k$-DGAs and $E_n$ $Hk$-Algebras. This equivalence for $n=1$, together with  \cite[4.1.8.4]{lurie2012higher} implies that we do not need to distinguish between the $\infty$-categorical setting and the model categorical setting when we classify quasi-isomorphism classes or topological equivalence classes of $k$-DGAs. The model category of chain complexes in $k$-modules satisfy the hypothesis of \cite[4.1.8.4]{lurie2012higher} due to \cite[7.1.4.3]{lurie2012higher}, \cite[2.3.11]{Hovey1999book} and the discussion after Convention 2.2 in \cite{chuang2018derivedlocalisations}.

Furthermore,  Theorem 4.1.1 of \cite{hinich2015rectificationofalgebrasandmodules} implies that the underlying $\infty$-category of the model category of $E_n$ $k$-DGAs is equivalent to the  $\infty$-category of $E_n$-algebras in the $\infty$-category of  $k$-chain complexes. This allows us to obtain an $E_2$-algebra in the $\infty$-category of $Hk$-modules from a given object in the model category of $E_2$ DGAs.

Another reason we work in the $\infty$-categorical setting is the following. Given an $\Sp$-algebra $F$ and a commutative $\Sp$-algebra $E$ in  the model category of spectra, an $E$-algebra structure on $F$ is given by a central map $E \to F$. However, center of an $\Sp$-algebra in the model categorical setting  may not be homotopy invariant. On the other hand, in the $\infty$-category of spectra,  one may use an $E_2$-map $E\to F$ to construct an $E$-algebra structure on $F$. This is made precise in the following lemma.

\begin{lemma}\label{lem e2 map makes e1 algebra}
Let $R$ be a commutative ring spectrum and let $E \to F$ be a map of $E_{n+1}$ $R$-algebras for some  $n\geq 1$. In this situation, the underlying $E_n$ $R$-algebra of $F$ carries the structure of an $E_n$ $E$-algebra.
\end{lemma}
\begin{proof}

 This follows as in \cite[3.7]{camarena2019simpleuniversal}. Since $E$ and $F$ are $E_{n+1}$ $R$-algebras, the $\infty$-categories of $E$-modules and $F$-modules are $E_n$-monoidal  $\infty$-categories \cite[5.1.2.6]{lurie2012higher}. Again due to  \cite[5.1.2.6]{lurie2012higher}, the left adjoint of the forgetful  functor from $F$-modules to $E$-modules is $E_n$-monoidal. Therefore, it follows by \cite[7.3.2.7]{lurie2012higher} that this forgetful functor is lax $E_n$-monoidal. 
 
 Furthermore, $R$ is the initial object in $E_{n+1}$ $R$-algebras \cite[3.2.1.8]{lurie2012higher}. In particular, we have a sequence of maps 
 \[ R \to E \to F\]
 of $E_{n+1}$ $R$-algebras. As before, the $\infty$-category of $R$-modules is a symmetric monoidal and therefore an $E_n$-monoidal $\infty$-category and the forgetful functors induced by the maps $R \to E$ and $R\to F$ above are lax $E_n$-monoidal. Note that a lax $E_n$-monoidal functor carries $E_n$-algebras to $E_n$-algebras.

 We obtain a commuting diagram of lax $E_n$-monoidal forgetful functors: 
 \begin{equation*}
     \begin{tikzcd}
     F\mhyphen \textup{modules} \ar[r] \ar[rd]& E\hmod \ar[d]\\
     & R\hmod.
     \end{tikzcd}
 \end{equation*}
 Since $F$ is the unit of the $E_n$-monoidal $\infty$-category $F$-modules, $F$ is an $E_n$-algebra in $F$-modules. Furthermore, it forgets to the underlying $E_n$-algebra of $F$  in $R\hmod$ (i.e.\ the underlying $E_n$ $R$-algebra of $F$ in the hypothesis of the lemma) through the forgetful functor induced by the map $R\to F$. 

 Carrying $F$ in $F$-modules to $E$-modules through the top horizontal arrow, we obtain an $E_n$-algebra in $E$-modules that also forgets to the underlying $E_n$-algebra of $F$ in $R$-modules because the diagram above commutes. In other words, the underlying $E_n$ $R$-algebra of $F$ carries the structure of an $E_n$ $E$-algebra as desired.
\end{proof}
\subsection{Proof of Theorems  \ref{thm e2dgassinglegenerator} and  \ref{thm e2dgas}}\label{sec subsection on proof of the e2 theorems}
 We start by introducing a version of the nilpotence theorem of Hopkins and Mahowald provided in  \cite[5.1]{camarena2019simpleuniversal},  this version of the nilpotence theorem is also proved in  \cite[4.16]{mathew2015nilpotenceconfecture}.

Let $\modmod[E_2]{\sph}{p}$ denote the pushout of the diagram 
\[\sph \xleftarrow{\cdot p} F_{E_2} \sph \xrightarrow{0}\sph\]
in $E_2$ $\sph$-algebras where $F_{E_2}-$ denotes the free functor from $\sph$-modules to $E_2$ $\sph$-algebras. Furthermore, the arrows denoted by $\cdot p$ and $0$ are the  adjoints of the maps $\sph \to \sph$ given by multiplication by $p$ and $0$ respectively.  Theorem 5.1 of \cite{camarena2019simpleuniversal} provides the following  characterization of the nilpotence theorem which states that $\hfp$ is the free $E_2$ $\sph$-algebra with $p=0$. 
\begin{theorem}
There is an equivalence of $E_2$ $\sph$-algebras 
\[\hfp \simeq \modmod[E_2]{\sph}{p}.\]
\end{theorem}

  Considering the pushout diagram defining $\modmod[E_2]{\sph}{p}$, one obtains the following corollary of the theorem above.
\begin{corollary}
Let $E$ be an $E_2$ $\sph$-algebra with $p=0$ in $\pi_*(E)$, then there is a  map
\[\hfp \to E\]
of $E_2$ $\sph$-algebras.
\end{corollary}

Combining this corollary with Lemma \ref{lem e2 map makes e1 algebra}, we obtain the following corollary.

\begin{corollary}\label{cor F is hfp algebra}
Let $F$ be an $E_2$ $\sph$-algebra with $p=0$ in $\pi_*(F)$, then the underlying $\sph$-algebra of $F$ carries the structure of an  $\hfp$-algebra.
\end{corollary}

To prove Theorems \ref{thm e2dgassinglegenerator} and \ref{thm e2dgas}, we first prove the stronger result stated in the following theorem.
\begin{theorem} \label{thm strongere2result}
Let $X$ be a DGA whose homology ring is an $\F_p$-algebra $A$ and assume that the $\Sp$-algebra corresponding to $X$ is an $E_2$ $\Sp$-algebra. If every $\F_p$-DGA with homology ring $A$ is formal, then $X$ is topologically formal.
\end{theorem}

\begin{proof}[Proof of Theorem \ref{thm strongere2result}]

Let $U(X)$ denote the $\sph$-algebra corresponding to $X$.
By hypothesis,  $U(X)$ is an $E_2$ $\Sp$-algebra. Since $ \pi_*U(X) = H_*X =  A$ is an $\fp$-algebra, we have $p=0$ in $\pi_0U(X)$. It follows from Corollary \ref{cor F is hfp algebra} that $U(X)$ carries an  $\hfp$-algebra structure. Let $Y$ denote the $\fp$-DGA corresponding to a chosen $\hfp$-algebra structure on $U(X)$. As an $\fp$-DGA, the homology of $Y$ is given by: 
\[H_*Y = \pis U(X) = H_*X = A.\]
Our hypothesis states that every $\fp$-DGA with homology $A$ is formal; we deduce that $Y$ is formal as an $\fp$-DGA. By construction, the underlying $\sph$-algebra of the $\hfp$-algebra corresponding to $Y$ is $U(X)$. Therefore, $X$ is topologically equivalent to $Y$ by construction. Since $Y$ is formal, we deduce that $X$ is topologically formal.
\end{proof}

To deduce Theorem \ref{thm e2dgassinglegenerator} and Theorem \ref{thm e2dgas} from Theorem \ref{thm strongere2result} we need to show that if $A$ is the homology of a DGA satisfying the hypothesis of Theorem \ref{thm e2dgassinglegenerator} or Theorem \ref{thm e2dgas}, then every $\F_p$-DGA with homology $
A$ is formal. 

For Theorem \ref{thm e2dgassinglegenerator}, this follows by a standard observation. In this case, $A= \F_p[x]$. Let $Z$ be an $\F_p$-DGA with homology $\F_p[x]$. One can show that there is a map $x \co \Sigma^{\lv x \rv}\F_p \to Z$ of $\F_p$-chain complexes  realizing $x$ in homology. Here,   $\Sigma^{\lv x \rv} \F_p$ denotes the chain complex that is trivial in all degrees except degree $n$ where it is given by $\fp$. By adjunction, there is a map  $\psi\co T(\Sigma^{\lv x \rv} \F_p) \to Z$ of $\F_p$-DGAs where $T\co \F_p\text{-chains} \to \F_p\text{-DGAs}$ denotes the free functor. We have 
\[T(\Sigma^{\lv x \rv} \F_p) \cong \bigoplus_{i \geq 0}(\Sigma^{\lv x \rv} \F_p)^{\otimes_{\F_p} i} \cong  \bigoplus_{i \geq 0}\Sigma^{\lv x \rv i}\F_p.\]
Therefore, the homology of $T(\Sigma^{\lv x \rv} \F_p)$ is also $\F_p[x]$ and $\psi$ induces an isomorphism in homology. Since $T(\Sigma^{\lv x \rv} \F_p)$ carries trivial differentials, we deduce that $Z$ is formal as it is quasi-isomorphic to $T(\Sigma^{\lv x \rv} \F_p)$. This proves Theorem \ref{thm e2dgassinglegenerator}.

For Theorem \ref{thm e2dgas}, we refer to Proposition 5.5 of \cite{johnson2014lifting}. This proposition states that 
every $\F_p$-DGA whose homology  is  a polynomial algebra with up to three generators in even degrees is quasi-isomorphic to the formal $\F_p$-DGA with the given homology. This is precisely the statement one needs in order to obtain Theorem \ref{thm e2dgas} from Theorem \ref{thm strongere2result}.

\subsection{Classification results for coconnective DGAs}\label{sec coconnective dgas}
Here, we prove Theorem \ref{thm exterior e2 DGAs}, Theorem \ref{thm topeqCDVR} and Corollary \ref{cor not e2}. We start by giving a short description of the main result of \cite{dwyer2013dg}. 

As mentioned before, the main result of \cite{dwyer2013dg} provides a bijection between quasi-isomorphism classes of DGAs with homology $\Lambda_{\F_p}(x_{-1})$ and  isomorphism classes of complete discrete valuation rings (CDVRs) with residue field $\F_p$. For a DGA $Z$ with homology $\Lambda_{\F_p}(x_{-1})$, the corresponding CDVR is given by the degree zero homology of $\text{End}_{Z\textup{-mod}}(\F_p)$ where $\text{End}_{Z\textup{-mod}}(-)$ denotes the derived endomorphism DGA in the model category of $Z$-modules. Similarly for a given CDVR $A$ with residue field $\F_p$, the corresponding DGA is given by the derived endomorphism DGA  $\text{End}_{A\textup{-mod}}(\F_p)$ of $\F_p$ in the model category of differential graded $A$-modules. 
\begin{example}(\cite[Remark 3.3]{dwyer2013dg})
For a CDVR $A$ with residue field $\fp$ and maximal ideal $(\pi)$, there is a projective resolution: 
\[0 \to A \xrightarrow{\cdot \pi} A \to 0, \]
of $\fp$ in $A$-modules. This provides an explicit description of the DGA $\text{End}_{A\textup{-mod}}(\F_p)$  given by the chain complex:
\[0\to \langle L \rangle \xrightarrow{\partial L = \pi D_1 + \pi D_2} \langle D_1, D_2 \rangle \xrightarrow{\partial D_i = (-1)^{i}\pi U}\langle U \rangle\to 0, \]
concentrated in degrees $-1,0$ and $1$. Here, $\langle - \rangle$ denotes the free $A$-module with the given generators and the multiplication on this chain complex is given by the multiplication of the following matrices.  
\begin{equation*}
L = 
\begin{pmatrix}
0 & 0\\
1 & 0
\end{pmatrix}\ 
D_1 = \begin{pmatrix}
1 & 0\\
0 & 0 
\end{pmatrix}\ 
D_2 = \begin{pmatrix}
0 & 0 \\
0 & 1 
\end{pmatrix}\ 
U = \begin{pmatrix}
0 & 1 \\
0 & 0 
\end{pmatrix}
\end{equation*}
\end{example}
\begin{example}
It is also interesting to consider the case $A= \Z$ although $\Z$ is not a CDVR. We compare the DGA corresponding to $\z$ with the DGA corresponding to the ring of $p$-adic integers $\z_p$. Note that $\z_p$ is a CDVR with residue field $\fp$. In this case, an elementary Ext computation shows that the left Quillen functor $\Z_p \otimes_{\z}-$ from chain complexes of $\Z$-modules to chain complexes of $\Z_p$-modules results in  a quasi-isomorphism 
\[\text{End}_{\Z\textup{-mod}}(\F_p) \we \text{End}_{\Z_p\textup{-mod}}(\F_p)\]
of DGAs. In other words, the DGAs corresponding to $\z$ and  $\Z_p$  are quasi-isomorphic. This shows that the correspondence described above is no longer a bijection if one considers rings that are not CDVRs. 
\end{example}

For the convenience of the reader, we provide a restatement of  Theorem \ref{thm exterior e2 DGAs} below. 
\vspace{0.3cm}
\begin{thmn}[\ref{thm exterior e2 DGAs}]
For $\lv x \rv <0$, every $E_2$ DGA with homology $\Lambda_{\F_p}[x]$ is formal as a DGA.
\end{thmn}
\vspace{0.1 cm}
\begin{proof}
Let $X$ be the $H\Z$-algebra corresponding to an $E_2$ DGA with homology $\Lambda_{\fp}(x)$ where $\lv x \rv<0$. As mentioned before, \cite[4.1.1]{hinich2015rectificationofalgebrasandmodules} and \cite[4.4]{peroux2020coalgebras} implies that $X$ is an $E_2$ $H\Z$-algebra. Since $\pi_*X = \Lambda_{\F_p}(x)$ with $\lv x \rv <0$, $X$ is coconnective. Therefore, $X$ admits a connective cover given by a map \[H\F_p \to X\] of $E_2$ $H\Z$-algebras \cite[7.1.3.11]{lurie2012higher}.  It follows by Lemma \ref{lem e2 map makes e1 algebra} that  $X$ is weakly equivalent as an $H\Z$-algebra to an $H\F_p$-algebra. Therefore, what remains is to show that every $\F_p$-DGA with homology $\Lambda_{\F_p}(x)$ is formal as a DGA. 

For $\lv x \rv <-1$, we use Proposition 6.1 of \cite{dwyer2013dg}. This states that there is a bijection between the  quasi-isomorphism classes of DGAs with homology $\Lambda_{\F_p}(x)$ and the  quasi-isomorphism classes of DGAs with homology $\F_p[u]$ where $\lv u \rv = -\lv x \rv -1$. The DGA with homology $\F_p[u]$ corresponding to a DGA $Z$ with homology $\Lambda_{\F_p}(x)$ is given by $\text{End}_{Z\textup{-mod}}(\F_p)$ where $\text{End}$ denotes the derived endomorphism DGA as before. If $Z$ is an $\F_p$-DGA  then the corresponding DGA $\text{End}_{Z\textup{-mod}}(\F_p)$ is also an $\F_p$-DGA. In the discussion after the proof of Theorem \ref{thm strongere2result}, we show that there is a unique $\F_p$-DGA with homology $\F_p[u]$. Therefore up to quasi-isomorphisms of DGAs, there is a unique $\F_p$-DGA with homology $\Lambda_{\F_p}(x)$. In other words, every $\fp$-DGA with homology  $\Lambda_{\F_p}(x)$ is formal as a DGA.

For the case $\lv x \rv = -1$, we use the bijection between CDVRs with residue field $\F_p$ and DGAs with homology $\Lambda_{\F_p}(x)$ discussed above. If $Z$ is an $\F_p$-DGA, then the corresponding CDVR $\text{End}_{Z\textup{-mod}}(\F_p)$ is an $\F_p$-algebra. It is known that there is a unique CDVR with residue field $\F_p$ and characteristic $p$, this is the formal power series ring  $\F_p[[ u ]]$ \cite[Chapter 2, Theorem 2]{serre1979localfields}. We conclude that up to quasi-isomorphisms of DGAs, there is a unique $\F_p$-DGA with homology $\Lambda_{\F_p}(x)$. 
\end{proof}
\begin{proof}[Proof of Corollary \ref{cor not e2}]
The first statement is a direct consequence of Theorem \ref{thm exterior e2 DGAs}. We prove the second statement.

 As discussed at the end of the proof of Theorem \ref{thm exterior e2 DGAs}, the formal DGA with homology $\Lambda_{\fp}(x_{-1})$ corresponds to the CDVR $\fp[[u]]$ because this formal DGA is an $\fp$-DGA. For a given CDVR $A$ with residue field $\fp$ and characteristic different than $p$, $\text{End}_{A\text{-mod}}(\F_p)$ is the corresponding DGA with homology $\Lambda_{\F_p}(x_{-1})$. Since $\fp[[u]]$ has characteristic $p$, $A \neq \fp[[u]]$. From this, we deduce that $\text{End}_{A\text{-mod}}(\F_p)$ is not formal. Therefore $\text{End}_{A\text{-mod}}(\F_p)$ is not an $E_2$ DGA by Theorem \ref{thm exterior e2 DGAs}.  
\end{proof}
We also improve the main result of \cite{dwyer2013dg} to obtain a classification of the topological equivalence classes of DGAs with homology $\Lambda_{\F_p}(x_{-1})$. In particular, we show that if two DGAs with homology $\Lambda_{\F_p}(x_{-1})$ are not equivalent through quasi-isomorphisms then they are not equivalent through topological equivalences.
\begin{theorem} \label{thm topeqCDVR}
There is a bijection between topological equivalence classes of DGAs with homology $\Lambda_{\F_p}(x_{-1})$ and isomorphism classes of CDVRs with residue field $\fp$.
\end{theorem}
\begin{proof}
Due to the main result of \cite{dwyer2013dg} described at the beginning of Section \ref{sec coconnective dgas}, we only need to show that topological equivalences and quasi-isomorphisms agree on DGAs with homology $\Lambda_{\fp}(x_{-1})$. Since quasi-isomorphic DGAs are topologically equivalent in general, we only need to show that topologically equivalent DGAs with homology $\Lambda_{\fp}(x_{-1})$ are quasi-isomorphic. In particular, it is sufficient to show that CDVRs corresponding to topologically equivalent DGAs with homology $\Lambda_{\fp}(x_{-1})$ are isomorphic. 

Let $Z$ be a DGA with homology $\Lambda_{\fp}(x_{-1})$. The corresponding CDVR is the degree zero homology 
\[H_0(\text{End}_{Z\textup{-mod}}(\F_p)).\]
By Propositions 1.5 and 1.7 of \cite{dugger2007additiveendomorphism}, the ring spectrum corresponding to $\text{End}_{Z\textup{-mod}}(\F_p)$ is the homotopy endomorphism ring spectrum in the sense of \cite{dugger2006spectral}. The homotopy endomorphism ring spectrum is invariant under Quillen equivalences by \cite[1.4]{dugger2006spectral}. For two topologically equivalent DGAs $Z$ and $Z'$ with homology $\Lambda_{\fp}(x_{-1})$, $Z$-modules is Quillen equivalent to $Z'$-modules \cite[7.2]{dugger2007topological}. Furthermore, it is clear that these Quillen equivalences preserve $\F_p$ which is uniquely determined by its homology \cite[Section 2]{dwyer2013dg}. This shows that the ring spectra corresponding to $\text{End}_{Z\textup{-mod}}(\F_p)$ and $\text{End}_{Z'\textup{-mod}}(\F_p)$ are weakly equivalent. Since the homology ring of a DGA is isomorphic to the homotopy ring of the corresponding ring spectrum, this shows that 
\[H_0(\text{End}_{Z\textup{-mod}}(\F_p))\cong H_0(\text{End}_{Z'\textup{-mod}}(\F_p)).\]
In other words, the CDVRs corresponding to $Z$ and $Z'$ are isomorphic. 
\end{proof}

\section{Postnikov extensions of ring spectra}\label{sec Postnikov}

The proof of the classification result in  Theorem \ref{thm nontrivial} relies on the Postnikov theory for ring spectra. This allows us to break the classification problem into smaller pieces by doing induction over a  Postnikov tower. In this section, we describe Postnikov sections for ring spectra and discuss the classification of Postnikov extensions of ring spectra by Andr{\'e}--Quillen cohomology  \cite{dugger2006postnikov}.

Let $R$ denote a connective q-cofibrant commutative $\Sp$-algebra throughout this section. By  connective, we mean $\pi_iR= 0$ for every $i<0$. Furthermore, let $X$ denote a connective q-cofibrant $R$-algebra.

 What we call an \textbf{$m$ Postnikov section} of $X$ is a map $X \to X[m]$ of $R$-algebras
which satisfies the following properties.

\begin{enumerate}
    \item The homomorphism $\pi_{n}X \to \pi_{n} X[m]$ is an isomorphism for every $n \leq m$.
    \item We have $\pi_nX[m] = 0$ for every $n > m$.
\end{enumerate}

Indeed, there is a Postnikov tower  
\[\cdots \to X[2] \to X[1] \to X[0]\]
with compatible Postnikov section maps $X \to X[m]$ for every $m \geq 0$. Note that we reserve this notation $X[m]$ to denote the target of an $m$ Postnikov section for a given $R$-algebra $X$. Furthermore, taking Postnikov sections can be made functorial. In particular  given $X \to Z$, we have a map $X[m]\to Z[m]$ for every $m\geq0$. 

Using this, we define Postnikov extensions. Let $X$ satisfy $X \simeq X[n-1]$ for some $n \geq 1$. Given a $\pi_0 X$-bimodule $M$, a \textbf{type $(M,n)$ Postnikov extension} of $X$ is an $R$-algebra $Z$ with a map $Z \to X$ satisfying the following properties.
\begin{enumerate}
    \item The induced map $Z[n-1] \to X[n-1] \simeq X$ is a weak equivalence.
    \item The map $Z \to Z[n]$ is a weak equivalence.
    \item There is an isomorphism of $\pi_0X$-bimodules $\pi_nZ \cong M$ where $\pi_nZ$ becomes a $\pi_0X$-bimodule through the isomorphism $\pi_0Z \cong \pi_0X$.
\end{enumerate}
Note that for every connective $R$-algebra $X$, the Postnikov section map $X[n] \to X[n-1]$ is a type $(\pi_nX[n],n)$ Postnikov extension of $X[n-1]$.

Let $\mathcal{M}_R(X+(M,n))$ denote the category whose objects are maps $Z \to X$ that are Postnikov extensions of $X$ of type $(M,n)$ and whose morphisms are weak equivalences over $X$. In other words, given two type $(M,n)$ Postnikov extensions of $X$ denoted by $Z_1 \to X$ and $Z_2 \to X$, a morphism from $Z_1 \to X$ to $Z_2 \to X$ in $\mathcal{M}_R(X+(M,n))$ is weak equivalence $Z_1 \we Z_2$ that makes the following diagram commute. 
 \begin{equation*}
 \begin{tikzcd}
 Z_1 \arrow [rr,"\simeq"]
 \arrow[dr,swap,]
 &
 & 
 Z_2
 \arrow[dl,swap,]
 \\
 &
 X
 \end{tikzcd}
 \end{equation*}
 We denote the weak equivalence classes of  type $(M,n)$ Postnikov extensions of $X$ by $\pi_0 \mathcal{M}_R(X+(M,n))$. 
 
 André--Quillen cohomology is a cohomology theory for ring spectra \cite{lazarev2004cohomology}. This is a generalization of the classical  Andr\'e--Quillen cohomology.
 
To define Andr\'e--Quillen cohomology, we consider $HM$ as an $X[0]$-bimodule via the equivalence $X[0]\simeq H\pi_0X$ and let $X[0] \vee \Sigma^{n+1} HM$ denote the square-zero extension of $X[0]$ by $\Sigma^{n+1}HM$. 

We use the following proposition as our definition of Andr\'e-Quillen cohomology where 
\[\mathrm{Der}_R^{n+1} (X,HM)\] denotes the $n+1$'st André--Quillen cohomology of $X$ with $HM$ coefficients in $R$-algebras. 

 \begin{proposition}\cite[8.6]{dugger2006postnikov} \label{prop defof AQcohomology}
Let $X$ and $M$ be as above. There is a bijection
  \[\mathrm{Der}_R^{n+1} (X,HM) \cong \pi_0(R\textup{-alg}_{/X[0]}(X, X[0] \vee \Sigma^{n+1}HM))\]
where $R\textup{-alg}_{/X[0]}(X, X[0] \vee \Sigma^{n+1} HM)$ denotes the derived mapping space of $R$-algebras over $X[0]$.
\end{proposition}

  In \cite{dugger2006postnikov}, Dugger and Shipley show that the weak equivalence classes of type $(M,n)$ Postnikov extensions of $X$ are classified by the following Andr\'e--Quillen cohomology groups.

 \begin{theorem} \label{thm pstnk dugger shipley}
\textup{\cite{dugger2006postnikov}}There is a bijection
\begin{equation*}
\pi_0 \mathcal{M}_R(X+(M,n)) \cong \mathrm{Der}_R^{n+1} (X,HM) / \mathrm{Aut}(M)
\end{equation*}
 where $\text{Aut}(M)$ denotes the $\pi_0X$-bimodule automorphisms  of $M$.
\end{theorem}

Note that in \cite{dugger2006postnikov}, this result is proved in the setting of symmetric spectra. However, there is a Quillen equivalence between $R$-algebras of EKMM spectra and $R$-algebras of symmetric spectra (for the corresponding $R$) \cite{Mandell01Model}. Therefore this result is also valid for EKMM spectra. Furthermore, we note that a similar classification result for   commutative $R$-algebras is given in \cite{basterra1999andre} and for $E_n$ $R$-algebras is given in \cite{basterra2013BP}. 

Given $k \in \mathrm{Der}_R^{n+1} (X,HM)$, $k$ is represented by a derived map $k \co X \to X[0] \vee \Sigma^{n+1}HM$ over $X[0]$. If $X \simeq X[n-1]$, one obtains a type $(M,n)$ Postnikov extension $Z\to X$ given by the following homotopy pullback square in $R$-algebras over $X[0]$.
\begin{equation*}
 \begin{tikzcd}
 Z \arrow [r]
 \arrow[d]
 & X[0]
 \arrow[d,"i"]
 \\
 X \arrow[r,"k"]
 &
 X[0] \vee \Sigma^{n+1}HM
 \end{tikzcd}
 \end{equation*}
 Note that $i$ represents the trivial derivation given by the canonical map. Furthermore by Theorem \ref{thm pstnk dugger shipley} above, every Postnikov extension is obtained from a derivation through a pullback square as above. The derivation corresponding to a Postnikov extension is called the \textbf{$k$-invariant} of the Postnikov extension. Theorem \ref{thm pstnk dugger shipley} also shows that if two derivations differ by an automorphism of $M$, then they result in weakly equivalent Postnikov extensions.

\section{Postnikov extensions and Hochschild homology} \label{sec pstnkv and hochschild}
In the previous section, we explained the way André--Quillen cohomology groups classify Postnikov extensions. Here, we show that the André--Quillen cohomology groups that we are interested in can be calculated using topological Hochschild homology. Later, we discuss the particular form of the B\"okstedt  spectral sequence we use for our calculations and provide standard calculational results. 

    Let $R$ be a q-cofibrant connective commutative $\Sp$-algebra throughout this section.  We use the notation $\wedge_R^L$ to denote the derived smash product. If the smash product is already giving the derived smash product, we often omit the superscript $L$. As usual, if the given smash product do not necessarily represent the derived smash product, we warn the reader.
\subsection{André--Quillen cohomology to Hochschild homology}
Let $X$ denote a connective q-cofibrant  $R$-algebra and let $M$ be a $\pi_0X$-bimodule. By letting  $X[0] = H\pi_0X$, $HM$ becomes an  $X[0]$-bimodule in a canonical way. Furthermore, $HM$ becomes an $X$-bimodule by forgetting through the $R$-algebra map $X\to X[0]$.

First, we define topological Hochschild cohomology as in \cite[\rom{9}]{elmendorf2007rings}.
\begin{definition} \label{def aq coh}
 Let $X$ and $HM$ be as above, then the degree $n$ topological Hochschild cohomology group of $X$ in $R$-algebras with coefficients in $HM$ is given by 
 \[\textup{THH}_{R}^{n}(X,HM) = \pi_{-n} F_{X^e}(X,HM)\]
 where $F_{X^e}(-,-)$ denotes the derived mapping spectrum in $X^e$-modules and $X^e$ denotes the derived smash product $X \wedge_R X^{op}$. Note that $X^{op}$ denotes the $R$-algebra whose underlying $R$-module is $X$ and whose multiplication is given by switching the order of multiplication on $X$.
\end{definition}
 The following result is proved in Section 8.5 of \cite{dugger2007topological}.
 \begin{theorem} \label{thm aq coh to thhcoh}
Let $X$ be a connective $R$-algebra and $M$ be a $\pi_0X$-bimodule. There is an isomorphism 
\[\textup{Der}^n_R(X,HM) \cong \textup{THH}_{R}^{n+1}(X,HM) \]
that respects the action of $\text{Aut}(M)$ for every $n>0$. Here, $HM$ is considered as an $X$-bimodule as described above. 
\end{theorem}
\begin{remark} \label{rmk functorial THHcohomology}
We use the theorem above in the case  $HM=X[0]$ and where we have another $R$-algebra $Z$ with a map $Z \to X$ that induces an isomorphism $\pi_0Z \to \pi_0X$. We define $Z[0]$ to be $X[0]$ together with the composite map $Z \to X \to X[0]$. This makes the map $Z\to X$ a map of $R$-algebras over $X[0]$.

Since we use   Proposition \ref{prop defof AQcohomology} as our definition of André--Quillen cohomology, we obtain a map 
$\textup{Der}^n_R(X,X[0]) \to \textup{Der}^n_R(Z,X[0])$ for every $n$. Furthermore, we have
\[F_{X^e}(X,X[0]) \to F_{Z^e}(X,X[0]) \to F_{Z^e}(Z,X[0]))\]
where the first map comes from the forgetful functor and the second map is the map induced by the map $Z \to X$. This gives us a map 
\[\textup{THH}_{R}^{n+1}(X,X[0]) \to \textup{THH}_{R}^{n+1}(Z,X[0]).\]

In \cite{francis2013thetangentcomplex}, Francis defines an $E_n$ André--Quillen cohomology theory. For $n=1$, this is the André--Quillen cohomology we consider here, see \cite[2.27]{francis2013thetangentcomplex}. Furthermore, Francis shows that the correspondence between André--Quillen cohomology and topological Hochschild cohomology we give in Theorem \ref{thm aq coh to thhcoh} is natural in the setting we describe here. 
\end{remark}

The following is the definition of topological Hochschild homology.
\begin{definition}
 Let $X$ be a q-cofibrant (commutative) $R$-algebra and let $HM$ be as in Definition \ref{def aq coh}. Degree $n$ topological Hochschild homology group of $X$ with $HM$ coefficients in $R$-algebras is defined by the following: 
 \[\text{THH}_{n}^R(X,HM) =\pi_n( X \wedge_{X^e}^L HM),\]
 where $\wedge_{X^e}^L$  denotes the derived smash product. As before, $X^e = X \wedge_{R} X^{op}$ and this represents the derived smash product because $X$ is q-cofibrant. 
\end{definition}

Let  $X$ be a q-cofibrant (commutative) $R$-algebra with an $R$-algebra map $X \to H\F_p$. Furthermore, we assume that $\hfp$ is q-cofibrant as a commutative $R$-algebra. 
An alternative definition of topological Hochschild homology is given in \cite[\rom{9}.2.1]{elmendorf2007rings} using simplicial resolutions. 

To obtain this simplicial resolution, one starts with the bar construction of $X$. This is the simplicial $X^e$-module given by $B_n^R(X) = X^{\wedge_Rn+2}$ in simplicial degree $n$. Note that the $X \wedge X^{op}$-module structure on $B_n^R(X)$ is given by the first and the last $X$ factors in the canonical way. The degeneracy maps are induced by the unit map of $X$ and the face maps are induced by the multiplication map of $X$, see \cite[\rom{4}.7.2]{elmendorf2007rings}. The geometric realization of this simplicial $R$-module is $X$. Furthermore, $B_{\bullet}^R(X)$ is proper by \cite[\rom{7}.6.8]{elmendorf2007rings}. 

We obtain a new simplicial $R$-module $CB^R_\bullet(X,H\F_p)$ by 
\begin{equation}\label{eq definition of the cyclic bar construction}
CB^R_\bullet(X,H\F_p) = B_\bullet^R(X) \wedge_{X^e} H\F_p
\end{equation}
Note that we have $CB^R_m(X,H\F_p) = X^{\wedge_R m} \wedge_R H\F_p$. A description of the face and degeneracy maps of $CB^R_\bullet(X,H\F_p)$ is given in \cite[\rom{9}.2.1]{elmendorf2007rings}. Since $H\F_p$ is q-cofibrant as a commutative $R$-algebra, $CB^R_\bullet(X,H\F_p)$ is also proper by \cite[\rom{7}.6.8]{elmendorf2007rings}.  Since the geometric realization commutes with smash products, we obtain 
\[\lv CB^R_\bullet(X,H\F_p) \rv \simeq X \wedge_{X^e}^L H\F_p\]
where $\lv - \rv$ denotes the geometric realization. Note that  we did not need to replace $H\F_p$ by a cell $X^e$-module because $CB^R_\bullet(X,H\F_p)$ is proper and the smash products in each degree are derived by \cite[\rom{7}.6.7]{elmendorf2007rings}. We obtain that 
\[\pi_*(\lv CB^R_\bullet(X,H\F_p) \rv) \cong \text{THH}_{*}^R(X,H\F_p).\]

Let $f\co X_1\to X_2$ be a map of  $R$-algebras over $H\F_p$ where each $X_i$ is q-cofibrant as a commutative or associative $R$-algebra.  The corresponding map 
\[\thh^R_*(X_1,\hfp) \to \thh^R_*(X_2,\hfp) \]
is described as follows. We start with the map
$B_\bullet^R(X_1) \to B_\bullet^R(X_2)$ of $X_1^e$-modules given by $f^{\wdg_R n+2}$ in simplicial degree $n$. This induces the following map. 
\begin{equation*}
    \begin{split}
CB_\bullet^R(X_1,H\F_p) = B_\bullet^R(X_1) \wedge_{X_1^e} H\F_p &\to B_\bullet^R(X_2) \wedge_{X_1^e} H\F_p \\
&\to B_\bullet^R(X_2) \wedge_{X_2^e} H\F_p = CB_\bullet^R(X_2,H\F_p)   
    \end{split}
\end{equation*}
Note that the middle term above do not necessarily contain derived smash products in its simplicial degrees.
The composite map above is given by 
\[f^{\wedge_R m} \wedge_R H\F_p \co CB^R_m(X_1,H\F_p) \cong X_1^{\wedge_R m} \wedge_R H\F_p \to CB^R_m(X_2,H\F_p) \cong X_2^{\wedge_R m} \wedge_R H\F_p.\]
Taking geometric realizations on this composite map and using the fact that geometric realizations commute with smash products, we obtain another composite: 
\[X_1 \wdg_{X_1^e}^L \hfp \to X_2 \wdg_{X_1^e} \hfp \to X_2 \wdg_{X_2^e}^L \hfp,\]
providing the desired map on topological Hochschild homology. As before, the middle term above may not represent the derived smash product.

Together with Theorems \ref{thm pstnk dugger shipley} and \ref{thm aq coh to thhcoh}, the following proposition allows us to study Postnikov extensions of $R$-algebras using topological Hochschild homology. 
\begin{proposition} \label{prop thh coh to thh ho}
Let $X$ be a q-cofibrant $R$-algebra and let $X\to \hfp$ be a map of $R$-algebras. We consider $\hfp$ as an $X$-bimodule by using this map.  There is an isomorphism 
\[\textup{THH}_{R}^{n+1}(X,\hfp) \cong \textup{Hom}_{\fp}(\textup{THH}^{R}_{n+1}(X,\hfp),\fp)\]
that is natural with respect to maps of $R$-algebras over $\hfp$.
\end{proposition}
\begin{proof}
Since $\hfp$ is a commutative $R$-algebra, $\hfp^{op}= \hfp$ and its multiplication map $\hfp \wedge_R \hfp\to \hfp$ is an $R$-algebra map. Therefore we have the following composite $R$-algebra map
\[X^e \cong X \wedge_R X^{op} \to \hfp \wedge_R \hfp \to \hfp.\]

This induces a Quillen adjunction between $X^e$-modules and $\hfp$-modules where the left adjoint is 
\[- \wedge_{X^e} \hfp \co X^e\text{-modules} \to \hfp\text{-modules}\]
and the right adjoint is the forgetful functor. At the level of derived function spectra, we obtain the following derived adjunction.
\begin{equation}\label{eq an adjunction equality of mapping spectra}
    F_{X^e} (X,\hfp) \simeq F_{\hfp}(X \wedge_{X^e}^L \hfp, \hfp)
\end{equation} 

We identify the homotopy groups of the right hand using the Ext spectral sequence in \cite[\rom{4}.4.1]{elmendorf2007rings}. For an $\sph$-algebra $S$ and $S$-modules $K$ and $L$, this spectral sequence is given by 
\begin{equation}\label{thm ext spectral sequence}
E_2^{s,t} = \textup{Ext}_{S^*}^{s,t}(K^*,L^*) \Longrightarrow \pi_{-(s+t)}(F_S(K ,L))
\end{equation}
where $F_S(-,-)$ denotes the derived function spectrum in $S$-modules and $N^*$ denotes $\pi_{-*}(N)$ for a given spectrum $N$. 

Note that we are in the case $S=\hfp$. Since $\fp$ is a field, the homotopy groups of the mapping spectrum in $\hfp$-modules is given by the graded module of homomorphisms between the homotopy groups of the given $\hfp$-modules. Therefore, taking homotopy  of the  equivalence in \eqref{eq an adjunction equality of mapping spectra} gives the desired result. 

The naturality of the equivalence in the proposition follows by the naturality of adjoint functors considered in \eqref{eq an adjunction equality of mapping spectra}. 
\end{proof}

\subsection{Calculational tools for Hochschild homology}
We proceed with elementary results for topological Hochschild homology computations and a discussion on the B\"okstedt spectral sequence.

\begin{remark} \label{rmk cofsmash}
Let $R$ be a commutative $\Sp$-algebra with a map $R \to H\F_p$ of commutative $\Sp$-algebras. Given a q-cofibrant (commutative) $R$-algebra $X$, $H\F_p \wedge_{R} X$ is a q-cofibrant (commutative) $H\F_p$-algebra. This follows by the fact that $H\F_p \wedge_{R} -$ is a left Quillen functor from (commutative) $R$-algebras to (commutative) $H\F_p$-algebras. 

\end{remark}

\begin{proposition} \label{prop thhchangeofbase}
Let $H\F_p$ be  q-cofibrant as a commutative $R$-algebra and let $X$ be a q-cofibrant (commutative) $R$-algebra over $H\F_p$. There is an equivalence
\[\textup{THH}^R_n(X,H\F_p)\cong \textup{THH}^{H\F_p}_n(H\F_p \wedge_R X,H\F_p).\]
This equivalence is natural with respect to maps of (commutative) $R$-algebras over $H\F_p$. Note that under our q-cofibrancy assumptions, the smash product above is the derived smash product.
\end{proposition}

\begin{proof}
There is a natural isomorphism 
\begin{equation}\label{eq hfp smashsplit}
(H\F_p \wedge_R X^{\wedge_R m})\cong (H\F_p \wedge_R X)^{\wedge_{H\F_p} m}.
\end{equation}
Using this, the following shows that the simplicial $R$-modules calculating  $\textup{THH}^R_*(X,H\F_p)$ and $\textup{THH}^{H\F_p}_*(H\F_p \wedge_R X,H\F_p)$ are isomorphic. 
\[CB_m^R(X,H\F_p) = X^{\wedge_{R}m} \wedge_{R} H\F_p \cong (H\F_p \wedge_R X)^{\wedge_{H\F_p} m} \wedge_{H\F_p} H\F_p \cong CB_m^{H\F_p}(H\F_p \wedge_{R} X,H\F_p)\]
Indeed, this isomorphism preserves the face and degeneracy maps by standard arguments.

Furthermore, note that $H\F_p \wedge_{R} X$ represents the derived product by the q-cofibrancy assumptions on $H\F_p$ and $X$. The smash products on the right hand side are also derived because $H\F_p \wedge_{R} X$ is q-cofibrant as an (commutative) $H\F_p$-algebra, see Remark \ref{rmk cofsmash}. Properness also follows as before. By the naturality of the isomorphism in \eqref{eq hfp smashsplit}, this isomorphism of simplicial $R$-modules is also natural with respect to maps of $R$-algebras over $H\F_p$.
\end{proof}
For our calculations, we use the B\"okstedt spectral sequence in $H\F_p$-algebras. The second page of this spectral sequence is described using  Hochschild homology for graded $\fp$-algebras. These groups are defined as follows. Let $A\to B$ be a map of graded commutative $\fp$-algebras. Through this map, $B$ admits the structure of an $A$-bimodule. Mimicking the construction of the simplicial object in \eqref{eq definition of the cyclic bar construction}, one obtains a simplicial graded commutative  $\fp$-algebra $CB^{\fp}_\bullet(A,B) = A^{\otimes \bullet} \otimes B$ where we denote $\otimes_{\fp}$ by $\otimes$ throughout this work. The degree $n$ homotopy, i.e.\ the degree $n$ homology of the normalized chain complex, of this simplicial ring is the Hochschild homology group 
\[\hhfpn(A,B).\]
Note that this is a graded  $\fp$-module and in certain cases, we write $\hh^{\fp}_{*,*}(A,B)$ to emphasize the internal grading. Since $A$ and $B$ are commutative, $\hh^{\fp}_*(A,B)$ is a bigraded commutative ring. We have  
\[\hhfps(A,B) \cong \torup_*^{A^e}(A,B)\]
where $A^e = A \otimes A^{op}$.

The B\"okstedt spectral sequence is constructed for EKMM spectra in Theorem 2.8 in Chapter \rom{9} of \cite{elmendorf2007rings}. We use the version of this spectral sequence given in Theorem 8 of  \cite{hunter1996thhss} where further properties of the B\"okstedt spectral sequence are discussed. This is summarized in the theorem below. Throughout this work, when we say B\"okstedt spectral sequence we mean the spectral sequence given in the following theorem although this is not the standard terminology in the literature. Let $\mathcal{E}_{H\F_p}^{\cdot \to \cdot}$ denote the arrow category in $H\F_p$-algebras. The objects of $\arrowcat$ are morphisms in $H\F_p$-algebras and the morphisms of $\arrowcat$ are maps of morphisms of $\hfp$-algebras.

\begin{theorem}[B\"okstedt spectral sequence]\label{thm bokstedtss}
Let $X$ be a q-cofibrant $H\F_p$-algebra or commutative $H\F_p$-algebra and also let $Z$ be a q-cofibrant $H\F_p$-algebra or commutative $H\F_p$-algebra. Given $X \to Z$ in $\arrowcat$, there is a spectral sequence 
\[E^2_{s,t} = \hh^{\fp}_{s,t}(X_*,Z_*)\Longrightarrow \textup{THH}^{H\F_p}_{s+t}(X,Z)\]
\[d^r_{s,t}\co E^r_{s,t} \to E^r_{s-r,t+r-1}.\]
where  $X_* = \pi_*X$ and $Z_* = \pi_*Z$.
We often omit the subscript in $d^r_{s,t}$ and write $d^r$.
This spectral sequence is functorial in $\arrowcat$.
\end{theorem}
\begin{remark} \label{rmk cofsmash2}
Given a map $R \to H\F_p$ of commutative $\Sp$-algebras and a q-cofibrant (commutative) $R$-algebra $X$, $H\F_p \wedge_{R} X$ is a q-cofibrant (commutative) $H\F_p$-algebra; see Remark \ref{rmk cofsmash}. Therefore, we can apply Theorem \ref{thm bokstedtss} above to  objects of the form $H\F_p \wedge_{R} X$. 
\end{remark}
For the calculation of the $E^2$ page of this spectral sequence, we often use the following proposition that provides a splitting of the $E^2$ page. 

\begin{proposition} \label{prop tensorsplitting}
Let $A$ be an augmented graded   commutative $\fp$-algebra with a splitting  $A = A_1 \otimes A_2 \otimes \cdots \otimes A_n$ over $\F_p$ where each $A_i$ is also an augmented graded commutative $\F_p$-algebra. In this situation, there is an isomorphism of rings
\[\hhfps(A,\F_p) \cong \hhfps(A_1,\F_p) \otimes \hhfps(A_2,\F_p) \otimes \cdots \otimes \hhfps(A_n,\F_p). \]
Let $A'= A'_1 \otimes A'_2 \otimes \cdots \otimes A'_n$ be another splitting as above. The isomorphism above is natural with respect to  every map $f\co A\to A'$ that splits as $f = f_1 \otimes  \cdots \otimes f_n$ where each $f_i\co A_i \to A_i'$ is a map of augmented graded $\fp$-algebras. 
\end{proposition}
\begin{proof}

Since tensor product of simplicial commutative  rings is degreewise, we have 
\begin{equation}\label{eq a chain of equivalences for tor splitting}
\begin{split}
    CB^{\fp}_\bullet(A,\fp) \cong& A^{\otimes \bullet}\\ \cong& A_1^{\otimes \bullet} \otimes A_2^{\otimes \bullet} \otimes \cdots \otimes A_n^{\otimes \bullet}\\ \cong& CB^{\fp}_\bullet(A_1,\fp) \otimes CB^{\fp}_\bullet(A_2,\fp)\otimes \cdots \otimes  CB^{\fp}_\bullet(A_n,\fp).
\end{split}
\end{equation}
Because  the tensor products are over $\F_p$, we only have flat modules. Therefore,  the homotopy ring of the bottom right hand side above is the tensor product over $1 \leq i \leq n$ of the homotopy rings of the simplicial rings $CB^{\fp}_\bullet(A_i,\fp)$. Therefore, the equivalence in the proposition follows by noting that the homotopy of $CB^{\fp}_\bullet(A_i,\fp)$ is $\hhfps(A_i,\F_p)$ for each $i$. The naturality statement follows by the naturality of the isomorphisms in \eqref{eq a chain of equivalences for tor splitting}.
\end{proof}

The following proposition provides the $E^2$ page of the B\"okstedt spectral sequence above when $X_*$ is a free graded commutative ring and $Z_*=X_*$ or $Z_* = \F_p$. See  Proposition 2.1 in \cite{mcclure1993onthethhbu} for an instance of this result. Note that for a homogeneous element $x$ in a graded ring, we let $\sigma x$ denote a degree $(1,\lv x \rv)$ element.  Furthermore, $\Gamma_{\F_p}(-)$ denotes the divided power algebra over $\fp$ on the given generators. For our purposes, it is sufficient to note that $\Gamma_{\F_p}(x_1,x_2,...)\cong \F_p[x_1,x_2,...]$ as $\fp$-modules (similarly for finitely many generators).  The degree $k\lv x \rv$ element in $\Gamma_{\F_p}(x)$ is denoted by $\gamma_k(x)$; this is called the $k$'th divided power of $x$.

\begin{proposition} \label{prop tor of freecommutative}
For $p$ an odd prime, let $A= \F_p[x_1,x_2,...]\otimes \Lambda_{\F_p}(y_1,y_2,...)$ with even $\lv x_i \rv$ and odd $\lv y_i\rv$ for every $i$. For $p=2$, let $A= \F_2[x_1,x_2,...]$ with no restrictions on $\lv x_i \rv$. We have the following ring isomorphisms
\[\hhfps(A,A) \cong A \otimes \Lambda_{\F_p}(\sigma x_1, \sigma x_2,...) \otimes \Gamma_{\F_p}(\sigma y_1, \sigma y_2,...)\]
\[\hhfps(A,\F_p) \cong \Lambda_{\F_p}(\sigma x_1, \sigma x_2,...) \otimes \Gamma_{\F_p}(\sigma y_1, \sigma y_2,...)\]
where $\Gamma_{\F_p}(\sigma y_1, \sigma y_2,...)$ should be omitted for $p=2$. Furthermore, the map $\hhfps(A,A) \to \hhfps(A,\F_p)$ induced by the augmentation map $\epsilon \co A \to \F_p$ of the second factor is given by $\epsilon \otimes id$ on the right hand side of these isomorphisms where $id$ denotes the identity  map of $\Lambda_{\F_p}(\sigma x_1, \sigma x_2,...) \otimes \Gamma_{\F_p}(\sigma y_1, \sigma y_2,...)$.

This result still holds when there are finitely many polynomial algebra generators and/or finitely many exterior algebra generators of $A$. 
\end{proposition}
\begin{proof}
 We prove the result for odd primes, the $p=2$ case follows similarly. 
 
 We start with the proof of the second isomorphism. Due to Proposition \ref{prop tensorsplitting}, it is sufficient to prove this isomorphism for $A = \fp[x_1]$ and $A = \lambdafp(y_1)$. We start with the proof of the first case; the latter follows similarly.
 
 Let $A = \fp[x_1]$. There is an automorphism of $A^e \cong A \otimes A$ given by  
 \[x_1 \otimes 1 \to x_1 \otimes 1 \text{\ and \ } 1\otimes x_1 \to x_1 \otimes 1 - 1 \otimes x_1.\]
  Note that $A^e= A \otimes A$ since $A$ is graded commutative. We consider the $A\otimes A$-module structure on $A$ after forgetting through this automorphism. Therefore, we obtain an action where the first factor in $A \otimes A$ acts on $A$ in the usual way and the second factor acts trivially.
 
 At this point, we use what is called the K\"unneth formula. Given two graded commutative  rings $A_1$ and $A_2$, let $B_1$ and $C_1$ be commutative $A_1$-algebras and let $B_2$ and $C_2$ be commutative $A_2$-algebras. There is an isomorphism of rings
\begin{equation*}\label{eq kunneth}
\text{Tor}_*^{A_1 \otimes A_2}(B_1 \otimes B_2,C_1 \otimes C_2) \cong \text{Tor}_*^{A_1}(B_1,C_1) \otimes \text{Tor}_*^{A_2}(B_2,C_2). 
\end{equation*} 

To apply the K\"unneth formula in our case, let $A_1=A_2=A$, $B_1=A$ and $B_2=C_1=C_2=\F_p$. We obtain 
 \[\hhfps(A,\fp)\cong\tor_*^{A^e}(A,\fp) \cong \tor_*^{A}(A,\fp) \otimes \tor_*^A(\F_p,\F_p)\cong \lambdafp(\sigma x_1).\]
  This gives the second isomorphism in the proposition for $A = \fp[x_1]$. 
  
  Applying the same arguments and noting $\tor^{\lambdafp(y_1)}(\fp,\fp) = \Gamma_{\fp}(\sigma y_1)$ gives the second isomorphism in the proposition for $A = \lambdafp(y_1)$. 
  
  For the first isomorphism in the proposition, note that the construction $CB^{\fp}_{\bullet}(A,A) = A^{\ot \bullet}$ is symmetric monoidal with respect to the input ring $A$. Therefore, $\hhfps(A,A)$ is also symmetric monoidal. The result follows by applying the previous arguments to the cases $A = \fp[x_1]$ and $A= \lambdafp[y_1]$.
 \end{proof}
 
 We also make use of the following proposition for our Hochschild homology computations. 
 
 \begin{proposition} \label{prop tor of quot of polynomial}
Let $A = \F_p[z]/(z^m)$ where $\lv z \rv$ is even or $p=2$. There is an isomorphism of $\F_p$-modules
\[\hhfps(A,\F_p) \cong \Lambda_{\F_p}(\sigma z) \otimes \Gamma_{\F_p}(\varphi^m z) \]
where $\textup{deg}(\sigma z)  = (1,\lv z \rv)$ and $\textup{deg}(\varphi^m z) = (2,m\lv z \rv )$. 
\end{proposition}
\begin{proof}
 There is a free $A^e$ resolution of $A$  \cite[(1.6.1)]{larsen1992larsen},
 \[ \cdots \xrightarrow{v} \Sigma^{(m+1)\lv z \rv}A^e\xrightarrow{u} \Sigma^{m\lv z \rv}A^e\xrightarrow{v} \Sigma^{\lv z \rv}A^e \xrightarrow{u} A^e \xrightarrow{m}A\]
 where this resolution contains $\Sigma^{\lv z \rv+im \lv z \rv}A^e$ in homological degree $2i+1$ and $\Sigma^{im \lv z \rv}A^e$ in homological degree $2i$. Furthermore, $m$ is the multiplication map of $A$, $u$ multiplies by $ z \otimes 1 - 1 \otimes z$ and $v$ multiplies by $\Sigma_{i=0}^{m-1}z^i\otimes z^{m-i-1}$. Applying $- \otimes_{A^e} \F_p$ to this resolution, all the differentials become trivial and we obtain the desired result.  
\end{proof}
 
\section{On the differentials of the B\"okstedt spectral sequence}\label{sec differentials of the bokstedt ss}

  For the rest of this work, let $\hz$ be q-cofibrant as a commutative $\sph$-algebra and let $\hfp$ be q-cofibrant as a commutative $\hz$-algebra. Since the category of commutative $H\Z$-algebras is the same as the category of commutative $\Sp$-algebras under $H\Z$, cofibrations of commutative $H\Z$-algebras forget to cofibrations of commutative $\Sp$-algebras. This implies that $\hfp$ is also  q-cofibrant as a commutative $\sph$-algebra. Throughout this section, most of the smash products are derived due to our q-cofibrancy assumptions. Whenever a smash product does not necessarily  represent the derived smash product, we warn the reader.
  
  Throughout this section, let $d=\lvert x \rvert = 2p-2$, let $p$ be an odd prime and let $Y$ denote a q-cofibrant $H\Z$-algebra corresponding to the unique non-formal DGA with homology  $\Lambda_{\F_p}(x)$, see Example \ref{ex duggershipley}. The degree $0$ Postnikov section map of $Y$ represents a map $Y\to \hfp$ in the homotopy category of $\hz$-algebras. Since  all objects are q-fibrant in EKMM spectra,  $\hfp$ is q-fibrant as an $\hz$-algebra. Furthermore, $Y$ is q-cofibrant by assumption. Therefore, the degree $0$ Postnikov section map of $Y$ provides a map 
\[\epsilon_Y \co Y \to \hfp\]
of $\hz$-algebras.
 
 The proof of Theorem \ref{thm nontrivial} relies on the classification of Postnikov extensions of DGAs with truncated polynomial homology $\fp[x]/(x^m)$ and $d$ Postnikov section equivalent to $Y$. Due to our earlier discussions, this boils down to  Hochschild homology computations. In this section, we compute the relevant differentials in the B\"okstedt spectral sequences computing these Hochschild homology groups.
\subsection{The B\"okstedt spectral sequence for $Y$} \label{subsec bokstedt ss for Y}

 We start with the B\"okstedt spectral sequence computing the Hochschild homology groups that classify the relevant Postnikov extensions of $Y$. Namely, we are interested in the B\"okstedt spectral sequence computing  $\text{THH}^{H\F_p}(H\F_p \wedge_{H\Z}Y, H\F_p)$ where $H\F_p \wedge_{H\Z}Y$ acts on $\hfp$ via the composite map
  \begin{equation}\label{eq augmentation map of fp wdgz Y}
 H\F_p \wedge_{H\Z} Y \xrightarrow{id \wedge_{H\Z} \epsilon_Y} H\F_p \wedge_{H\Z} H\F_p \xrightarrow{m} H\F_p.
   \end{equation}
  Here, $m$ denotes the multiplication map on $H\F_p$ and $id$ denotes the identity map. Due to  Theorem \ref{thm bokstedtss}, this spectral sequence is given by:
 \begin{equation*} \label{eq ssforY}
E^2_{s,t} = \hh^{\fp}_{*,*}(\pis(\hfp \wdgz Y),\F_p)\Longrightarrow \text{THH}^{H\F_p}_{s+t}(H\F_p \wedge_{H\Z}Y, H\F_p).
\end{equation*}
 We let $E$ denote this spectral sequence for the rest of this section.

  By Lemma \ref{lemma homology}, we have an isomorphism of rings \[\pis(\hfp \wdgz Y) \cong \Lambda_{\F_p}(\tau_0)\otimes \Lambda_{\F_p}[x]\] 
  where $\lv \tau_0 \rv = 1$   and  $\lv x \rv = 2p-2$. Recall that $d = \lv x \rv = 2p-2$ is a standing assumption in this section.
We apply Proposition \ref{prop tensorsplitting} for the case $A_1 = \Lambda_{\F_p}(\tau_0)$, $A_2 = \Lambda_{\F_p}(x)$ and obtain
\begin{equation*} 
E^2 \cong \hhfps(\Lambda_{\F_p}(\tau_0),\F_p) \otimes \hhfps(\Lambda_{\F_p}(x),\F_p).  
\end{equation*}
By Proposition \ref{prop tor of freecommutative}, we have $\hhfps(\Lambda_{\F_p}(\tau_0),\F_p) \cong \Gamma_{\F_p}(\sigma \tau_0)$ where $\text{deg}(\sigma \tau_0)= (1,1)$. Note that we cannot apply this proposition for the second factor because $\lv x \rv$ is even. In this case, by Proposition \ref{prop tor of quot of polynomial} we have $\hhfps(\Lambda_{\F_p}(x),\F_p) \cong \Lambda_{\F_p}(\sigma x) \otimes \Gamma_{\F_p}(\varphi^2x)$ where $\text{deg}(\sigma x) = (1,d)$ and $\text{deg}(\varphi^2x) = (2,2d)$. We obtain
\begin{equation*}
E^2 \cong \Gamma_{\F_p}(\sigma \tau_0) \otimes \Lambda_{\F_p}(\sigma x) \otimes \Gamma_{\F_p}(\varphi^2x).
\end{equation*}
The following is a picture of the $E^2$-page where the horizontal axis denotes the homological degree and the vertical axis denotes the internal degree.
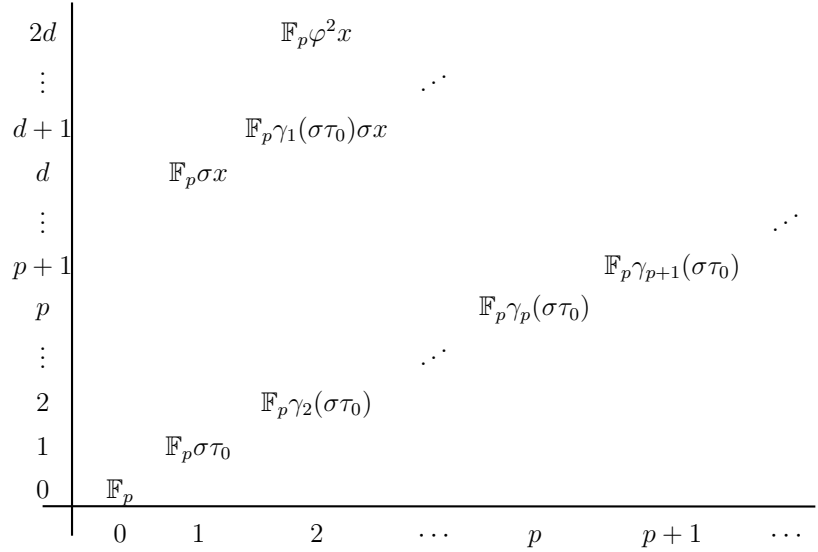
\begin{center}
\begin{tikzpicture}[thick,scale=0.2, every node/.style={scale=0.85}]
\matrix (m) [matrix of math nodes,
             nodes in empty cells,
             nodes={minimum width=6.95ex,
                    minimum height=7ex,
                    outer sep=-5pt},
             column sep=-0.3ex, row sep=-3.4ex,
             text centered,anchor=center]{
&2d  &    &    &  \F_p \varphi^2x  &   &  &                 \strut&  \\
&\vdots  &    &   &    & \reflectbox{$\ddots$}&    &  &             \\
&d+1  &   \strut  &     &  \fp \gamma_1(\sigma \tau_0) \sigma x   &   & &  &  &              \\
&d&     &  \fp \sigma x &   &  &     &  & \\  
&\vdots  &    &    &    &   &    & \strut&   \reflectbox{$\ddots$}                   \\
&p+1  &   \strut  &     &     &  &  &
         \F_p\gamma_{p+1}(\sigma \tau_0)  & \\
&p&     &  &   &  & \fp \gamma_p(\sigma \tau_0)               &  &\\
&\vdots      &    &   &    & \reflectbox{$\ddots$} & & &  \\
&2      &     &    & \fp \gamma_2(\sigma \tau_0)  &  &          &  & \\
&1 &    &  \F_p\sigma \tau_0  & & & &  & \\
&0      &  \F_p   &    &    &     &    & & \\
&\quad\strut &   0 &  1  &  2  &  \cdots &  p  &  p+1 &        \cdots \\        \\};
\draw[thick] (m-1-2.north east) -- (m-12-2.east) ;
\draw[thick] (m-12-2.north) -- (m-12-9.north east) ;
\end{tikzpicture}
\end{center}
 \begin{lemma}\label{lemma e2 page of the bokstedt of Y}
 Let $E$ denote the B\"okstedt spectral sequence computing  $\thh^{H\F_p}(H\F_p \wedge_{H\Z}Y, H\F_p)$ as above. 
 The $E^2$ page is given by
  \begin{equation*}
     \begin{split}
E^2 
& \cong \Gamma_{\F_p}(\sigma \tau_0) \otimes \Lambda_{\F_p}(\sigma x) \otimes \Gamma_{\F_p}(\varphi^2x).\\
     \end{split}
 \end{equation*}  
 For every $z \in  \Gamma_{\F_p}(\sigma \tau_0) \otimes \Lambda_{\F_p}(\sigma x) \subset E^2$ and $1<r<p-1$,  
 \[d^rz= 0\]
 and if $z\neq 0$, then $z$ is not in the image of $d^r$. 
 
 \end{lemma}

 \begin{proof}
The identification of the $E^2$ page is proved before the statement of the lemma. We only need to show the statement regarding the differentials. Indeed, this follows solely by degree considerations. We provide the details below. Recall that \[\text{deg}(\sigma \tau_0) = (1,1), \  \text{deg}(\gamma_k(\sigma \tau_0))  = (k,k),\]
  \[ \text{deg}(\sigma x) = (1,2p-2) \text{\ and\ } \text{deg}(\varphi^2 x) = (2,4p-4).\]
  Let $s$ denote the first coordinate, i.e.\ the homological degree and $t$ denote the second coordinate, i.e.\ the internal grading of $E^2$. Note that the elements $\gamma_k(\sigma \tau_0)$ lie on the line $t =s$ and the elements $\gamma_k(\sigma \tau_0) \sigma x$ are on the line $t = s +2p-3$. Furthermore, $\varphi^2x$ is on the line $t = s + 4p-6$. Given an element $y$ on the line $t = s+c_1$ and another element $z$ on the line $t = s+ c_2$ we have $\text{deg}(y) = (m, m+ c_1)$ and $\text{deg}(z) = (n, n+ c_2)$ for some $m$ and $n$. Therefore, $\text{deg}(yz) = (m+n,m+n+c_1+c_2)$ and this shows that $yz$ lies on the line $t= s + c_1 + c_2$. Note that we denote $y \ot z$ by $yz$. This in particular shows that every element on $E^2$ other than the ones in $\Gamma_{\F_p}(\sigma \tau_0) \otimes \Lambda_{\F_p}(\sigma x)$ lie on a line $y= x+ c$ for some $c \geq 4p-6$, rest of the elements, i.e.\ the ones in $\Gamma_{\F_p}(\sigma \tau_0) \otimes \Lambda_{\F_p}(\sigma x)$, are either on $t =s $ or $t= s+ 2p-3$.
  
  Recall that  
  \[d^r_{s,t} \co E^r_{s,t} \to E^r_{s-r,t+r-1}.\]
  Therefore a class on the line $t = s+ c$ hits an element on the line $t = s+ c+2r-1$ under $d^r$.
  
  Since there are no non-trivial elements on $E^2$ who are on a line $t= s+c$ with a $c<0$, the elements of the form $\gamma_k(\sigma \tau_0)$ do not get hit by the differentials for every $k$. Furthermore, $d^r \gamma_k(\sigma \tau_0)$ lies on the line $t = s+ 2r-1$ and for $1<r<p-1$, we have $1<2r-1< 2p-3$. Since there are no non-trivial classes on these lines, this shows that $d^r\gamma_k(\sigma \tau_0)= 0$ for every $k \geq 0$ and $1<r<p-1$. 
  
  Now we consider elements of the form $\gamma_k(\sigma \tau_0) \sigma x$ for some $k$. These elements are on the line $t=s+2p-3$. They possibly receive differentials from the lines $t = s+c$ with $c<2p-3$ but the only non-trivial classes on these degrees are of the form $\gamma_j(\sigma \tau_0)$ and we already showed that these classes carry trivial differentials up to the $p-2$ page. Therefore $\gamma_k(\sigma \tau_0) \sigma x$ is not in the image of $d^r$ for every $1<r<p-1$ and every $k$. Furthermore, $d^r( \gamma_k(\sigma \tau_0) \sigma x)$ lies on the line $t = s+ 2p-3 +2r-1$ and for $1<r<p-1$, we have $2p-2<2p-3+2r-1 <4p-6$. By the discussion above, there are no non-trivial classes on these lines and we deduce that $d^r( \gamma_k(\sigma \tau_0) \sigma x) = 0$ for every $k \geq 0$ and $1<r<p-1$.
  \end{proof}
  The following lemma provides our first non-trivial differential in the B\"okstedt spectral sequence $E$.
  \begin{lemma}\label{lem first nontrivial differential}
  In the B\"okstedt spectral sequence $E$ computing $\textup{THH}^{H\F_p}_*(H\F_p \wedge_{H\Z}Y, H\F_p)$, we have: 
 \[d^{p-1} \gamma_p(\sigma \tau_0) =  \sigma x\]
 up to a unit of $\fp$. 
  \end{lemma}
  \begin{proof}
   We assume to the contrary that $d^{p-1} \gamma_p(\sigma \tau_0) \neq c\sigma x$ for every unit $c\in \fp$  and obtain a contradiction.  Let $\psi\co Y \to H\F_p$ denote the degree zero Postnikov section map of $Y$. This is a Postnikov extension of type $(\F_p,2p-2)$ in $H\Z$-algebras. Due to Theorem \ref{thm pstnk dugger shipley}, such Postnikov extensions are classified by:
   \begin{equation*}\label{eq hhcoh of fp over z}
  \text{Der}^{2p-1}_{H\Z}(H\F_p,H\F_p)/\aut(\fp) \cong \thh_{H\Z}^{2p}(\hfp,\F_p)/\aut(\fp)   \cong \F_p/\aut(\fp) \cong \{ [0],[1]\}
     \end{equation*}
  where the first isomorphism follows by Theorem \ref{thm aq coh to thhcoh} and the second isomorphism is shown in \cite[3.15]{dugger2007topological}. 
 Since $Y$ is not formal as a DGA, the Postnikov extension $Y \to H\F_p$ is obtained from the non-trivial element above \cite[3.15]{dugger2007topological}.
  
  Let $k_{2p-2}$ denote a representative of this non-trivial derivation. This means that we have a pull back square 
\begin{equation*} \label{eq posntikov pullback square}
 \begin{tikzcd}
 Y \arrow [r]
 \arrow[d,"\psi"]
 & \hfp
 \arrow[d,"i"]
 \\
 \hfp \arrow[r,"k_{2p-2}"]
 &
 \hfp \vee \Sigma^{2p-1}H\F_p
 \end{tikzcd}
 \end{equation*}
of $\hz$-algebras where  $i$ denotes the trivial derivation. This shows that the map $\psi \co Y\to \hfp$ pulls back the non-trivial derivation $k_{2p-2}$ to a trivial derivation. Indeed, we obtain a contradiction by showing that $k_{2p-2} \circ \psi$  represents a non-trivial derivation. 
 
 Now we formulate this statement using Hochschild homology. By the naturality of the isomorphism in Theorem \ref{thm aq coh to thhcoh}, $k_{2p-2} \circ \psi$ is nontrivial if  the map 
\[\F_p \cong \thh_{H\Z}^{2p}(\hfp,H\F_p) \to \thh_{H\Z}^{2p}(Y,H\F_p)\]
induced by $\psi$ carries $1 \in \F_p$ to a non-trivial element in $\thh_{H\Z}^{2p}(Y,H\F_p)$. Note that the isomorphism above follows by \cite[3.15]{dugger2007topological}. The natural correspondence in Proposition \ref{prop thh coh to thh ho} shows that it is also sufficient to show that the $\F_p$-dual of the nontrivial element in $\thh^{H\Z}_{2p}(\hfp,\F_p)\cong \F_p$ is carried to a non-trivial element in the $\F_p$-dual of $\thh^{H\Z}_{2p}(Y,\F_p)$. This is the case precisely when $1 \in \thh^{H\Z}_{2p}(\hfp,\F_p)\cong \F_p$ is in the image of the map $\thh^{H\Z}_{2p}(Y,\F_p) \to \thh^{H\Z}_{2p}(\hfp,\F_p)$. Due to  Proposition \ref{prop thhchangeofbase}, it is also sufficient to show that the map \begin{equation}\label{eq another map of thh}
    \text{THH}^{H\F_p}_{2p}(H\F_p \wedge_{H\Z} Y,H\F_p) \to \text{THH}^{H\F_p}_{2p}(H\F_p \wedge_{H\Z} \hfp,H\F_p)\cong \F_p
\end{equation}
induced by  $H\F_p \wedge_{H\Z} \psi$ is nontrivial, i.e.\ $1 \in \F_p$ is in the image of this map. 

  Let $\overbar{E}$ denote the B\"okstedt spectral sequence in Theorem \ref{thm bokstedtss} computing 
  \[ \thh^{\hfp}_*(\hfp \wdgz \hfp,\hfp)\]
  where $\hfp \wdgz \hfp$ acts on $\hfp$ through the multiplication map $H\F_p \wedge_{H\Z} H\F_p \xrightarrow{m} H\F_p$. Because $\pi_*(H\F_p \wedge_{H\Z} H\F_p) \cong \lambdatau$ with $\lv \tau_0 \rv = 1$, we have 
  \[\overbar{E}^2 \cong \hhfps(\lambdatau,\F_p) \cong \Gamma_{\F_p}(\sigma \tau_0).\]
   By degree reasons, all the differentials are trivial on the $\overbar{E}^2$ page and after. 
   
  It follows by  Lemma \ref{lemma homology} that the map \[\pi_*(H\F_p \wedge_{H\Z} \psi) \co \pi_*(H\F_p \wedge_{H\Z} Y) \cong \lambdatau \otimes \Lambda_{\F_p}(x) \to \pi_*(H\F_p \wedge_{H\Z} H\F_p) \cong \lambdatau\]
  is given by  $id \otimes \epsilon_{\Lambda_{\F_p}(x)}$ where $\epsilon_{\Lambda_{\F_p}(x)}$ is the augmentation map $\Lambda_{\F_p}(x) \to \F_p$ and $id$ denotes the identity map of $\lambdafp(\tau_0)$.
  There is an induced map of spectral sequences $\psi_* \co E^* \to \overbar{E}^*$; the second page of  $E$ is given in Lemma \ref{lemma e2 page of the bokstedt of Y}. It follows by the naturality of the tensor splitting in Proposition \ref{prop tensorsplitting} that the map 
  \[\psi_2 \co E^2\cong \Gamma_{\F_p}(\sigma \tau_0) \otimes \Lambda_{\F_p}(\sigma x) \otimes \Gamma_{\F_p}(\varphi^2x) \to \overbar{E}^2\cong \Gamma_{\F_p}(\sigma \tau_0)\]
  is given by the identity map  on $\Gamma_{\F_p}(\sigma \tau_0)  \subset E^2$ and it is 
  trivial on the other positive-dimensional generators. In particular, we have 
  \[\psi_2(\gamma_p(\tau_0)) = \gamma_p(\tau_0).\]
  
   Since it is the only non-trivial class in total degree $2p$, the element $\gamma_p(\sigma \tau_0) \in \overbar{E}^2$ represents a non-trivial element on the right hand side of \eqref{eq another map of thh}. Therefore in order to show that the map in \eqref{eq another map of thh} is non-trivial, i.e.\ in order to obtain the contradiction we need, it is sufficient to show that $\gamma_p(\sigma \tau_0) \in E^2$ survives to the $E^\infty$ page because  $\psi_2(\gamma_p(\sigma \tau_0)) = \gamma_p(\sigma \tau_0)$. In other words, it is sufficient   to show that $d^r\gamma_p(\sigma \tau_0)= 0$ for every $r\geq2$.
  
  Due to Lemma \ref{lemma e2 page of the bokstedt of Y}, $d^r\gamma_p(\sigma \tau_0) = 0$ for $1<r<p-1$ in $E$. Since $\gamma_p(\sigma \tau_0)$ is in homological degree $p$, $d^r\gamma_p(\sigma \tau_0)= 0$ for $r>p-1$. Therefore the only possible non-trivial differential on $\gamma_p(\sigma \tau_0)$ is $d^{p-1}$. We have \[\text{deg}(d^{p-1}\gamma_p(\sigma \tau_0)) = (p-(p-1),p+p-1-1) = (1,2p-2).\]
  The only non-trivial class of $E^{p-1}$ in this degree is indeed $\sigma x$ but we assumed that $d^{p-1} \gamma_p(\sigma \tau_0) \neq c\sigma x$ for every unit $c$ and therefore $d^{p-1}\gamma_p(\sigma \tau_0) = 0$. This shows that $\gamma_p(\sigma \tau_0)$ survives to the $E^\infty$ page. This provides the contradiction we need since $\gamma_p(\sigma \tau_0)$ hits the element $\gamma_p(\sigma \tau_0) \in \overbar{E}^2$  which represents a non-trivial element on the right hand side of \eqref{eq another map of thh}.

 \end{proof}

 We obtain more non-trivial differentials on $E^{p-1}$  by using the  nontrivial differentials in the B\"okstedt  spectral sequence  computing 
 \[\text{THH}^{H\F_p}(H\F_p \wedge H\F_p, H\F_p \wedge H\F_p) \cong H\F_p \wedge \text{THH}^{\Sp}(H\F_p,H\F_p)\]
 where $\hfp \wdg \hfp$ acts on $\hfp \wdg \hfp$ via the identity map $H\F_p \wedge H\F_p \xrightarrow{} H\fp \wedge H\fp$. Let $\hat{E}$ denote this spectral sequence; we start by describing $\hat{E}$. 
 
 The calculation of $\text{THH}^{\Sp}(H\F_p,H\F_p)$ and the description of the differentials in this spectral sequence is due to B\"okstedt \cite{bokstedt1985topological}. In \cite{hunter1996thhss}, Hunter provides a new proof of these differentials in a modern setting of spectra.
 
 Note that $\pi_*(H\F_p \wedge H\F_p) \cong \dsa$ where $\dsa$  denotes the dual Steenrod algebra. The dual Steenrod algebra is the following free graded commutative $\F_p$-algebra 
 \[ \mathcal{A}_* = \F_p[\xi_r \ \lvert \ 
 r\geq 1] \otimes \Lambda_{\F_p}(\tau_s \ \lvert \  s \geq 0)\]
where $\lvert \xi_r \rvert =  2(p^r-1)$ and $\lvert \tau_s \rvert =  2p^s-1$ \cite{milnor1958steenrod}. It follows by Proposition \ref{prop tor of freecommutative} that the $\hat{E}^2$ page of this B\"okstedt spectral sequence is given by 
\begin{equation}\label{eq another eq e2hat}
    \hat{E}^2 \cong \hhfps(\dsa,\dsa) \cong \dsa \otimes \Lambda_{\F_p}[\sigma \xi_r \ \lvert \ 
 r\geq 1] \otimes \Gamma_{\F_p}(\sigma \tau_s \ \lvert \  s \geq 0). 
\end{equation}
For this spectral sequence,  $\hat{d}^r=0$ for $2\leq r<p-1$ and $\hat{d}^{p-1} \gamma_k (\sigma \tau_i) =  \gamma_{k-p}( \sigma \tau_i)\sigma \xi_{i+1}$ for $k \geq p$ \cite[Theorem 1]{hunter1996thhss}. This is a spectral sequence of $\dsa$-algebras \cite[\rom{9}.2.8]{elmendorf2007rings}. Together with this $\dsa$-algebra structure, the formula for $\hat{d}^{p-1}$ determines all the non-trivial differentials on $\hat{E}^{p-1}$. As shown in \cite{bokstedt1985topological},  $\hat{d}^r=0$ for $r>p-1$.

Let $\Sp\cof \overbar{H\F}_p \trfib H\F_p$ denote a q-cofibrant replacement in associative $\Sp$-algebras. Note that by  \cite[\rom{7}.6.5 and \rom{7}.6.7]{elmendorf2007rings}, $H\F_p \wedge \overbar{H\F}_p$ represents the derived smash product. Therefore, there is a weak equivalence 
\[\hfp \wdg \overbar{H\F}_p \we \hfp \wdg \hfp\]
of $\hfp$-algebras. Through this equivalence, one obtains an isomorphism of spectral sequences between $\hat{E}$ and the B\"okstedt spectral sequence corresponding to the identity map $\hfp \wdg \overbar{H\F}_p \to \hfp \wdg \overbar{H\F}_p$. From now on, we abuse our previous notation and let $\hat{E}$ denote the latter spectral sequence.

To simplify our discussion, we use the B\"okstedt spectral sequence computing   \[\text{THH}^{H\F_p}(H\F_p \wedge \obhfp,H\F_p)\] where $H\F_p \wedge \obhfp$ acts on $\hfp$ through a given map \[\rho \co \hfp \wdg \obhfp \to \hfp\] of $\hfp$-algebras. Note that for the moment, we only assume that $\rho$ is a map of $\hfp$-algebras; later, we will specify this map. Let $\tilde{E}$ denote the  B\"okstedt spectral sequence corresponding to $\rho$. We have: 
\[\tilde{E}^2 \cong \hhfps(\dsa,\F_p) \cong \Lambda_{\F_p}[\sigma \xi_r \ \lvert \ 
 r\geq 1] \otimes \Gamma_{\F_p}(\sigma \tau_s \ \lvert \  s \geq 0),  \]
 where the second isomorphism follows by Proposition \ref{prop tor of freecommutative}.
 In order to deduce the differentials on this spectral sequence, we use the map 
 \[\text{THH}^{H\F_p}( H\F_p \wedge \obhfp,H\F_p \wedge \obhfp) \to \text{THH}^{H\F_p}(H\F_p \wedge \obhfp,H\F_p)\] 
 induced by the map $\rho$. This induces a map of spectral sequences $\hat{E} \to \tilde{E}$. On the second page, this map is given by $\epsilon_{\dsa} \otimes id$ where $\epsilon_{\dsa}$ denotes the augmentation map of $\dsa$ on the first factor of $\hat{E}^2$ in \eqref{eq another eq e2hat} and $id$ denotes the identity map of the rest of the factors of $\hat{E}^2$; this follows by Proposition \ref{prop tor of freecommutative}. The map of spectral sequences $\hat{E} \to \tilde{E}$ determines the differentials of $\tilde{E}$. We obtain the following lemma. 
\begin{lemma}\label{lem boksted ss for hfp}
 As above, let $\tilde{E}$ denote the B\"okstedt spectral sequence corresponding to a map $\rho \co \hfp \wdg \obhfp \to \hfp$ of $\hfp$-algebras. We  have:
\[\tilde{E}^2 \cong \hhfps(\dsa,\F_p) \cong \Lambda_{\F_p}[\sigma \xi_r \ \lvert \ 
 r\geq 1] \otimes \Gamma_{\F_p}(\sigma \tau_s \ \lvert \  s \geq 0),  \] $\tilde{d}^r = 0$ for $2\leq r<p-1$, \[\tilde{d}^{p-1} \gamma_k (\sigma \tau_i) = \gamma_{k-p}( \sigma \tau_i)\sigma \xi_{i+1} \]
for $k\geq p$ and $\tilde{d}^r=0$ for $r\geq p$.
\end{lemma}

The following lemma allows us to compare the spectral sequences $\tilde{E}$ and $E$. 
 
 \begin{lemma} \label{lemma map of underlying things}
Let $\overbar{H\F}_p$ be as above. There exists a map $\eta \co H\F_p \wedge \overbar{H\F}_p \to H\F_p \wedge_{H\Z} Y$ of $H\F_p$-algebras that satisfies the following properties.
 \begin{enumerate}
     \item We have $\pis(\eta)(\tau_0) = \tau_0$ (up to a unit of $\fp$).
     \item There is a map of $\hfp$-algebras $\rho \co \hfp \wdg \overbar{H\F}_p \to \hfp$ that render the map $\eta$ to a map of $\hfp$-algebras over $\hfp$ with respect to the $\hfp$-algebra structure of \eqref{eq augmentation map of fp wdgz Y}.
     \item The induced map of spectral sequences $\eta_* \co \tilde{E}^* \to E^*$  carries $\Gamma_{\F_p} (\sigma \tau_0)\subset \tilde{E}^2$   isomorphically  to $\Gamma_{\F_p} (\sigma \tau_0)\subset E^2$
 \end{enumerate}
 \end{lemma}
 
 \begin{proof}
  Let $Z$ denote the $H\z$-algebra corresponding to the the formal DGA with homology $\Lambda_{\F_p}[x]$. 
 Because the DGAs corresponding to $Y$ and  $Z$ are topologically equivalent, $Y$ and $Z$ are weakly equivalent as $\Sp$-algebras, see Example \ref{ex duggershipley}.  Let $\Sp \cof cZ \trfib Z$ denote a q-cofibrant replacement of $Z$ in $\Sp$-algebras. Since $Y$ is q-fibrant as an $\Sp$-algebra, there is a weak equivalence $\nu \co cZ \we Y$ of $\Sp$-algebras.
 
 Since there is a map of DGAs from $\fp$ to the formal DGA with homology $\lambdafp[x]$, there is a map $\hfp \to Z$ in the homotopy category of $\hz$-algebras. This forgets to a map $\hfp \to Z$ in the homotopy category of $\sph$-algebras. In turn, we obtain a map  $\phi\co \overbar{H\F}_p \to cZ$ of  $\sph$-algebras as $cZ$ is q-fibrant and weakly equivalent to $Z$ and $\overbar{H\F}_p$ is q-cofibrant. 

This gives us a map $\nu \circ \phi \co \overbar{H\F}_p \to Y$. From this, we obtain the desired map $\eta$ through the following composite.
\[\eta \co H\F_p \wedge \overbar{H\F}_p \xrightarrow{H\F_p \wedge (\nu \circ \phi)} H\F_p \wedge Y \to H\F_p \wedge_{H\Z} Y\]
Note that $\hfp \wdg Y$ may not be representing the derived smash product but this does not cause a problem.

Now we prove our claim $\pis(\eta)(\tau_0) = \tau_0$.  We have the following commuting diagram. 

\begin{equation*}
    \begin{tikzcd}[row sep=normal, column sep = large]
    \hfp \wdg \overbar{H\F}_p \ar[r,"\hfp \wdg (\nu \circ \phi)"]&\hfp \wdg Y  \ar[d] \ar[r,"\hfp \wdg \epsilon_Y"] &\hfp \wdg \hfp\ar[d] \\
    & \hfp \wdgz Y \ar[r,"\hfp \wdgz \epsilon_Y"]& \hfp \wdgz \hfp 
    \end{tikzcd}
\end{equation*}
The composite of the top horizontal arrows is a weak equivalence since it is given by the map $\hfp \wdg (\epsilon_Y \circ \nu \circ \phi)$ where the $\sph$-algebra map $\epsilon_Y \circ \nu \circ \phi\co \overbar{H\F}_p \to \hfp$ has to be a weak equivalence because it carries the unit to the unit in homotopy. In degree $1$, $\mathcal{A}_*$ is generated by $\tau_0$ as an $\fp$-module. Therefore, the composite of the top horizontal arrows carries $\tau_0$ to $\tau_0$ up to a unit. At the level of homotopy groups, the right  hand vertical map is the graded ring map $\dsa \to \lambdatau$ that carries $\tau_0$ to $\tau_0$. This shows that the $\tau_0$ on the top left corner is carried to the $\tau_0$ on the bottom right corner. Due to Lemma \ref{lemma homology}, we have 
\[\pis(\hfp \wdgz Y) \cong \lambdatau \otimes \lambdafp(x).\]
In particular, $\pi_1(\hfp \wdgz Y)$ is generated by $\tau_0$ as an $\fp$-module. Since the $\tau_0$ on the upper left corner travels to a non-trivial element on the bottom right corner, we deduce that it hits  $\tau_0$ up to a unit in $\pis(\hfp \wdgz Y)$. In other words,  $ \pis(\eta)(\tau_0) = \tau_0$ up to a unit of $\fp$ as desired. This finishes the proof of the first item of the lemma. 

We define the map $\rho$ through the composite:
  \[\rho \co H\F_p \wedge \overbar{H\F}_p \xrightarrow{\eta} H\F_p \wedge_{H\Z} Y \xrightarrow{id \wedge_{H\Z} \epsilon_Y} H\F_p \wedge_{H\Z} H\F_p \xrightarrow{m} H\F_p,\]
  where $m$ denotes the multiplication map of $H\F_p$ and $id$ denotes the identity map of $\hfp$. It is clear that $\rho$ is a map of $\hfp$-algebras. Since the map $\hfp \wdgz Y \to \hfp$ is  the composite of the last two maps above, $\eta$ is a map of $\hfp$-algebras over $\hfp$ as desired.

Now we prove the last item in the lemma. The canonical inclusion  $i \co \lambdatau \to \dsa$ induces the following composite
 \begin{equation*}
 \begin{split}
     &\hhfps(\lambdatau,\F_p) \to \hhfps(\dsa,\F_p)\cong \tilde{E}^2 \\
     &\xrightarrow{\eta_2} \hhfps(\lambdatau \otimes \Lambda_{\F_p}(x), \F_p) \cong E^2
     \end{split}
 \end{equation*}
 which is given by 
  \begin{equation*}
     \begin{split}
 \Gamma_{\F_p}(\sigma \tau_0)&\to \Lambda_{\F_p}[\sigma \xi_r \ \lvert \ 
 r\geq 1] \otimes \Gamma_{\F_p}(\sigma \tau_s \ \lvert \  s \geq 0)\\  
 &\xrightarrow{\eta_2} \Gamma_{\F_p}(\sigma \tau_0) \otimes \Lambda_{\F_p}(\sigma x) \otimes \Gamma_{\F_p}(\varphi^2x)
     \end{split}
\end{equation*}
due to Lemmas \ref{lemma e2 page of the bokstedt of Y} and \ref{lem boksted ss for hfp}. The first map in this composite is the canonical inclusion by the naturality of the tensor splitting in Proposition \ref{prop tensorsplitting}. Therefore, it is sufficient to show that the composite map is the canonical inclusion. Note that the composite map is induced by the map of rings $ \pis(\eta) \circ i$. This map is the canonical inclusion $\lambdatau \to \lambdatau \otimes \Lambda_{\F_p} (x)$ since $ \pis(\eta) (\tau_0) = \tau_0$ (up to a unit). Using this, in combination with the naturality of the tensor splitting in Proposition \ref{prop tensorsplitting}, we deduce that the composite is the canonical inclusion. This gives the desired result as the first map is also the canonical inclusion.

 \end{proof}

The following lemma provides all the non-trivial differentials of $E$ that we need to know.
\begin{remark}
Since we do not assume that $Y$ is an $E_n$ $\hz$-algebra  for $n>1$, we also do no assume that the B\"okstedt spectral sequence $E$ is multiplicative. Nevertheless, to simplify notation, we sometimes denote the elements of the form $x \otimes y$ as $xy$ in the tensor splitting provided by Proposition \ref{prop tensorsplitting}. This makes sense since the second page of the B\"okstedt spectral sequences we consider have canonical ring structures (omitting the differentials) and the tensor splittings provided by Proposition \ref{prop tensorsplitting} respect these ring structures.
\end{remark}
\begin{lemma}\label{lemma differentials for Y}
 In the spectral sequence  $E$ of Lemma \ref{lemma e2 page of the bokstedt of Y}, we have
\[d^{p-1}\gamma_k(\sigma \tau_0)= \gamma_{k-p}(\sigma \tau_0)\sigma x\]
for every $k \geq p$. 
\end{lemma}
 \begin{proof}

  By Lemma \ref{lemma map of underlying things}, there is a map 
  \begin{equation*} \label{eq mapfpwedgefptofpwedgex}
  \eta \co H\F_p \wedge \overbar{H\F}_p \to H\F_p \wedge_{H\Z} Y
  \end{equation*}
  of $H\F_p$-algebras satisfying the properties listed in Lemma \ref{lemma map of underlying things}. See Lemma \ref{lem boksted ss for hfp} for a description of the B\"okstedt spectral sequence $\tilde{E}$ computing $\thh^{\hfp}(\hfp \wdg \overbar{H\F}_p,\hfp)$.
  
  The map $\eta$ induces a map of spectral sequences $\eta_* \co \tilde{E}^* \to E^*$. By Lemma \ref{lemma e2 page of the bokstedt of Y}, the non-trivial elements of $\Gamma_{\F_p}(\sigma \tau_0)  \otimes \Lambda_{\F_p}(\sigma x) \subset E^2$ survive non-trivially to $E^{p-1}$. Similarly, the non-trivial elements of $\Gamma_{\F_p}(\sigma \tau_0)  \otimes \Lambda_{\F_p}(\sigma \xi_1) \subset \tilde{E}^2$ survive non-trivially to $\tilde{E}^{p-1}$ due to Lemma \ref{lem boksted ss for hfp}. Combining this with Lemma \ref{lemma map of underlying things}, we obtain that $\eta_{p-1}$   carries $\Gamma_{\F_p}(\sigma \tau_0) \subset \tilde{E}^{p-1}$ isomorphically to $\Gamma_{\F_p}(\sigma \tau_0) \subset E^{p-1}$. This provides the third equality in:  
  \begin{equation*} \label{eq psi nontrivial}
  \eta_{p-1}(\sigma \xi_1) = \eta_{p-1}(\tilde{d}^{p-1} \gamma_p(\sigma \tau_0)) = d^{p-1} \eta_{p-1}(\gamma_p(\sigma \tau_0)) = d^{p-1} \gamma_p(\sigma \tau_0) = \sigma x,
  \end{equation*}
  where the first equality follows by Lemma \ref{lem boksted ss for hfp}, the second equality follows by the fact that $\eta_*$ is a map of spectral sequences and the last equality follows by Lemma \ref{lem first nontrivial differential}. This implies that \[\eta_2(\sigma \xi_1) = \sigma x.\]
  By Lemma \ref{lemma map of underlying things}, we have $\eta_2(\gamma_k(\sigma \tau_0)) = \gamma_k(\sigma \tau_0)$. Since $\eta_2$ is a map of rings, we deduce that 
  \[\eta_{2}( \gamma_k(\sigma \tau_0)\sigma \xi_1 ) = \gamma_k(\sigma \tau_0) \sigma x.\]
  Note that for the previous statement, we did not assume that the spectral sequences $E$ and $\tilde{E}$ are multiplicative spectral sequences. We only use the fact that  $E^2_{*,*}$ and $\tilde{E}^2_{*,*}$ have  ring structures and  $\eta_2$ is a ring map.
  By the triviality of the relevant differentials on $E$ and $\tilde{E}$, we deduce the following.
  \begin{equation}\label{eq anoth eq proof of lemma on nontrivial diff for Y}
      \eta_{p-1}( \gamma_k(\sigma \tau_0)\sigma \xi_1 ) = \gamma_k(\sigma \tau_0) \sigma x
  \end{equation}

We are now ready to deduce the non-trivial differentials claimed in the lemma. We have the following for every $k \geq p$.
\begin{equation*}
    \begin{split}
        d^{p-1}\gamma_k(\sigma \tau_0) &= d^{p-1} \eta_{p-1}(\gamma_k(\sigma \tau_0))\\
        &= \eta_{p-1}(\tilde{d}^{p-1}\gamma_k(\sigma \tau_0))\\ 
        &= \eta_{p-1}(\gamma_{k-p}(\sigma \tau_0) \sigma \xi_1)\\ &= \gamma_{k-p}(\sigma \tau_0) \sigma x
    \end{split}
\end{equation*}
Here, the first equality follows by Lemma \ref{lemma map of underlying things},  the second equality follows by the fact that $\eta_*$ is a map of spectral sequences, the third equality follows by Lemma \ref{lem boksted ss for hfp} and the last equality follows by \eqref{eq anoth eq proof of lemma on nontrivial diff for Y}. 
  \end{proof}
  The following is a picture of a part of the $E^{p-1}$ page. Note that the horizontal axis denotes the homological degree and the vertical axis denotes the internal degree.

\begin{center}
\begin{tikzpicture}[thick,scale=0.2, every node/.style={scale=0.85}]
\matrix (m) [matrix of math nodes,
             nodes in empty cells,
             nodes={minimum width=6.95ex,
                    minimum height=7ex,
                    outer sep=-5pt},
             column sep=-0.3ex, row sep=-3.4ex,
             text centered,anchor=center]{
&2d  &    &    &  \F_p \varphi^2x  &   &  &                 \strut&  \\
&\vdots  &    &   &    & \reflectbox{$\ddots$}&    &  &             \\
&d+1  &   \strut  &     &  \fp \gamma_1(\sigma \tau_0) \sigma x   &   & &  &  &              \\
&d&     &  \fp \sigma x &   &  &     &  & \\  
&\vdots  &    &    &    &   &    & \strut&   \reflectbox{$\ddots$}                   \\
&p+1  &   \strut  &     &     &  &  &
         \F_p\gamma_{p+1}(\sigma \tau_0)  & \\
&p&     &  &   &  & \fp \gamma_p(\sigma \tau_0)               &  &\\
&\vdots      &    &   &    & \reflectbox{$\ddots$} & & &  \\
&2      &     &    & \fp \gamma_2(\sigma \tau_0)  &  &          &  & \\
&1 &    &  \F_p\sigma \tau_0  & & & &  & \\
&0      &  \F_p   &    &    &     &    & & \\
&\quad\strut &   0 &  1  &  2  &  \cdots &  p  &  p+1 &        \cdots \\
  };
    \draw[-stealth] (m-7-7) -- (m-4-4);
        \draw[-stealth] (m-6-8) -- (m-3-5);
        \draw[-stealth] (m-5-9) -- (m-2-6);
\draw[thick] (m-1-2.north east) -- (m-12-2.east) ;
\draw[thick] (m-12-2.north) -- (m-12-9.north east) ;
\end{tikzpicture}
\end{center}
  \subsection{The B\"okstedt spectral sequence for extensions of $Y$}
  
  Here, we set up and describe the spectral sequences we use for various Postnikov extensions of $Y$.  Recall that throughout this section, we have $d = \lv x \rv = 2p-2$.
  \begin{remark}
  We do not yet claim the existence of an $\hz$-algebra $X$ with homotopy ring $\fp[x]/(x^m)$ and $2p-2$ Postnikov section equivalent to $Y$ for $m>2$, c.f.\ Theorem \ref{thm nontrivial}.  Before we prove such an $X$ exists, we prove several properties of a hypothetical such $X$ without claiming that it exists. Note that our statements below regarding $X$ would be vacuously true if such an $X$ did not exist. 
  \end{remark}
  
  Let $X$ be an $\hz$-algebra whose $2p-2$ Postnikov section is equivalent to   $Y$ and homotopy ring is $\fp[x]/(x^m)$ for $m>2$. We always assume that $X$ is q-cofibrant. Since $Y$ is q-fibrant, the $2p-2$ Postnikov section map of $X$ provides a map 
  \[\psi \co X \to Y\]
  of $\hz$-algebras serving as the $2p-2$ Postnikov section of $X$. Furthermore, we consider $\hfp \wdgz X$ as an $\hfp$-algebra over $\hfp$ through the composite map 
  \begin{equation}\label{eq def of upsilon}
 \upsilon \co \hfp \wdgz X \xrightarrow{id \wdgz \psi} \hfp \wdgz Y \xrightarrow{id \wdgz \epsilon_Y} \hfp \wdgz \hfp \xrightarrow{m} \hfp. \end{equation}
  In this situation, the map $\hfp \wdgz \psi$ is a map of $\hfp$-algebras over $\hfp$.
  
   Let $\overline{E}$ denote the  B\"okstedt spectral sequence computing $\thh^{\hfp}_*(\hfp \wdgz X , \hfp)$ where $\hfp \wdgz X $ acts on $\hfp$ via  $\upsilon$.  Note that Lemma \ref{lemma homology} provides 
   \[\pis(\hfp \wdgz X) \cong \lambdatau \otimes \F_p[x]/(x^m)\] where $\lv \tau_0 \rv = 1$. We have: 
   \begin{equation*} \label{eq ss e2pagefor X}
    \begin{split}
        \overbar{E}^2 &\cong \hhfps(\lambdatau \otimes \F_p[x]/(x^m),\F_p) \\
       & \cong \hhfps(\lambdatau,\F_p) \otimes \hhfps(\F_p[x]/(x^m),\F_p) \\
       & \cong  \Gamma_{\fp}(\sigma \tau_0) \otimes \Lambda_{\fp}(\sigma x) \otimes \Gamma_{\F_p}(\varphi^m x),\\
    \end{split}
\end{equation*}
where the second equality follows by Proposition \ref{prop tensorsplitting} and the last equality follows by Propositions \ref{prop tor of freecommutative} and \ref{prop tor of quot of polynomial}.
  \begin{lemma}\label{lem bokstedt ss for X}
  Let $X$ be an $\hz$-algebra with $2p-2$ Postnikov section equivalent to $Y$ and homotopy ring $\fp[x]/(x^m)$ for some $m>2$. Let $\overbar{E}$ denote the B\"okstedt spectral sequence corresponding to the map $\upsilon$ given above. We have 
  \[\overbar{E}^2 \cong  \Gamma_{\fp}(\sigma \tau_0) \otimes \Lambda_{\fp}(\sigma x) \otimes \Gamma_{\F_p}(\varphi^m x) \Longrightarrow{} \thh_*^{\hfp}(\hfp \wdgz X, \hfp).\]
  For every $z \in  \Gamma_{\F_p}(\sigma \tau_0) \otimes \Lambda_{\F_p}(\sigma x) \subset \overbar{E}^2$ and $1<r<p-1$,  
 \[\overbar{d}^rz= 0\]
 and if $z\neq 0$, then $z$ is not in the image of $\overbar{d}^r$. Futhermore,
  
  \[\overbar{d}^{p-1}\gamma_k(\sigma \tau_0)= \gamma_{k-p}(\sigma \tau_0)\sigma x\]
for every $k \geq p$. 
  \end{lemma}
  \begin{proof}
   The identification of the second page is established before the statement of the lemma. 
   
   The statement on the triviality of the differentials follows by degree reasons. The proof of this fact is as in the proof of Lemma \ref{lemma e2 page of the bokstedt of Y}.
   
   What is left to prove is the last statement.  Let $E$ denote the spectral sequence calculating $\textup{THH}^{\hfp}(\hfp \wdgz Y,\F_p)$ as before. The $E^2$ page of this spectral sequence is given by: 
   \[E^2 \cong \Gamma_{\F_p}(\sigma \tau_0) \otimes \Lambda_{\F_p}(\sigma x) \otimes \Gamma_{\F_p}(\varphi^2x),\] 
   see Lemma \ref{lemma e2 page of the bokstedt of Y}. Note that $\overbar{E}^2$ is the same as $E^2$ except that instead of the generator $\varphi^2 x$ in degree $(2,2d)$, $\overbar{E}^2$ has a generator $\varphi^m x$ in degree $(2,md)$. 
   
   To compare the spectral sequences $E$ and $\overbar{E}$, we use the map  $\hfp \wdgz \psi$. This is a map of  $\hfp$-algebras over $\hfp$ and therefore induces a map of spectral sequences $\psi_*\co \overbar{E} \to E$. 

To understand $\psi_*$, we need to know the induced map \[\pis(\hfp \wdgz \psi) \co \pis(\hfp \wdgz X) \to \pis(\hfp \wdgz Y).\] In Lemma \ref{lemma homology} below, we show that this map is given by:
\[\pis(\hfp \wdgz \psi) = id \otimes \pis \psi \co \lambdatau \otimes \F_p[x]/(x^m) \to \lambdatau \otimes \Lambda_{\F_p}(x),\] where $id$ denotes the identity map of $\lambdatau$ and  $\pis \psi \co \F_p[x]/(x^m) \to \Lambda_{\F_p}(x)$ is the ring map that sends $x$ to $x$.

Now we show that the induced map $\overbar{E}^2\to E^2$ carries $\Gamma_{\F_p}(\sigma \tau_0)\otimes \Lambda_{\F_p}(\sigma x)\otimes 1 \subset \overbar{E}^2$ isomorphically onto $\Gamma_{\F_p}(\sigma \tau_0)\otimes \Lambda_{\F_p}(\sigma x)\otimes 1 \subset E^2$ in the canonical way. By the splitting of $\pis(\hfp \wdgz \psi)$ above and the naturality of the splitting in Proposition \ref{prop tensorsplitting}, we obtain that the map $\overbar{E}^2 \to E^2$ is of the form $id \otimes f$ where $id$ denotes the identity map of $\Gamma_{\fp}(\sigma \tau_0)$ and $f$ denotes the map 
\begin{equation*}
    \begin{split}
\hhfps(\F_p[x]/(x^m),\F_p)\cong \Lambda_{\F_p}&(\sigma x) \otimes \Gamma_{\F_p}(\varphi^m x)\\ &\to  \hhfps(\Lambda_{\F_p}(x),\F_p)\cong \Lambda_{\F_p}(\sigma x) \otimes \Gamma_{\F_p}(\varphi^2 x)
\end{split}
\end{equation*}
induced by $\pis\psi$.
 Furthermore, $\sigma x\otimes 1$ on the left hand side is represented by $x \in \F_p[x]/(x^m)$ in the first simplicial degree of the simplicial ring $\F_p[x]/(x^m)^{\otimes \bullet}$ whose homotopy (homology of its normalized chain complex) is $\hhfps(\F_p[x]/(x^m),\F_p)$. Similarly, $\sigma x\otimes 1$ on the right hand side is represented by $x \in \Lambda_{\F_p}(x)$ in the first simplicial degree of the simplicial ring $\Lambda_{\F_p}(x)^{\otimes \bullet}$ whose homotopy is $\hhfps(\Lambda_{\F_p}(x),\F_p)$. Since $\pis\psi(x)=x$, we obtain that $f(\sigma x\ot 1)= \sigma x\ot 1$. This proves that the map $\psi_2\co \overbar{E}^2 \to E^2$, given by $id \otimes f$, carries  $\Gamma_{\F_p}(\sigma \tau_0)\otimes \Lambda_{\F_p}(\sigma x)\otimes 1 \subset \overbar{E}^2$  isomorphically onto $\Gamma_{\F_p}(\sigma \tau_0)\otimes \Lambda_{\F_p}(\sigma x)\otimes 1 \subset E^2$.
 
 Using the second statement in the lemma and  Lemma \ref{lemma e2 page of the bokstedt of Y}, we deduce that \[\psi_{p-1}\co \overbar{E}^{p-1} \to E^{p-1}\] also carries $\Gamma_{\F_p}(\sigma \tau_0)\otimes \Lambda_{\F_p}(\sigma x)\otimes 1 \subset \overbar{E}^{p-1}$  isomorphically onto $\Gamma_{\F_p}(\sigma \tau_0)\otimes \Lambda_{\F_p}(\sigma x)\otimes 1 \subset E^{p-1}$. From this and Lemma \ref{lemma differentials for Y}, we deduce that 
\begin{equation*}\label{eq nontrivldifferentials for X}
\overbar{d}^{p-1}\gamma_k(\sigma \tau_0)= \gamma_{k-p}(\sigma \tau_0)\sigma x
\end{equation*}
for every $k \geq p$ for the spectral sequence $\overbar{E}$.
  \end{proof}

 \section{Classification of Postnikov extensions}\label{sec proof of the main theorems}
 
 We start this section with   classification of various Postnikov extensions of interest. At the end of this section, we prove Theorems \ref{thm nontrivial} and \ref{thm hochschild}.
 
 Throughout this section, let $d= \lv x \rv = 2p-2$ and let  $Y$ denote a q-cofibrant $\hz$-algebra corresponding to the unique non-formal DGA with homology $\lambdafp(x)$; see Example \ref{ex duggershipley}. Note that in this section, we do not assume that $p$ is odd. 
 \subsection{Postnikov extensions of DGAs with truncated polynomial homology}
 Let $X$ be a q-cofibrant $H\Z$-algebra whose $d$ Postnikov section is  equivalent to $Y$ and whose homotopy ring is  $\F_p[x]/(x^m)$  for some $m>1$. As before, we do not assume the existence of such an $X$ for $m>2$. On the other hand, we have $X \simeq Y$ when $m=2$.
 
 Let $\overbar{E}$ denote the spectral sequence computing $\thh^{\hfp}(\hfp \wdgz X,\hfp)$. For odd $p$, this spectral sequence is described in Lemma \ref{lem bokstedt ss for X} for $m>2$, and it is described in Lemma \ref{lemma e2 page of the bokstedt of Y} for $m=2$ (where it is denoted by $E$). The construction of this spectral sequence for $p=2$ follows as in the odd prime case.
 
 \begin{proposition} \label{prop HHofX}
Let $X$ be an $\hz$-algebra as above and let $d = \lv x \rv = 2p-2$. For degrees less than $md+2$: \[\textup{THH}^{\hfp}_*(\hfp \wdgz X,H\F_p) \cong \F_p[\mu]/(\mu^p),\] with $\lv \mu \rv = 2$. In degree $md+2$, we have  \[\textup{THH}^{\hfp}_{md+2}(\hfp \wdgz X,H\F_p) \cong \F_p.\] 
For odd $p$, this Hochschild homology group is generated by the class $\varphi^m x \in \overbar{E}^2$.
\end{proposition}

\begin{proof}
We have 
\begin{equation*} 
 \overbar{E}^2_{s,t} = \hh^{\fp}_{s,t}(\pis(\hfp \wdgz X),\F_p)\Longrightarrow \text{THH}^{H\F_p}_{s+t}(H\F_p \wedge_{H\Z}X, H\F_p).
\end{equation*}

We start with the case $p=2$. By Lemma \ref{lemma homologyfor p eq 2}, we have $\pis(\hft \wdgz X)\cong \F_2[\xi_1]/(\xi_1^{2m})$  where $\lv \xi_1 \rv=1$. Therefore using Proposition \ref{prop tor of quot of polynomial}, we obtain 
\begin{equation*}\label{eq p is 2 e2 of X}
\overbar{E}^2 \cong \hhfts(\F_2[\xi_1]/(\xi_1^{2m}),\F_2) \cong \Lambda_{\F_2}(\sigma \xi_1) \otimes \Gamma_{\F_2}(\varphi^{2m} \xi_1)\end{equation*}
with $\textup{deg}(\sigma \xi_1)= (1,1)$ and $\text{deg}(\varphi^{2m} \xi_1)= (2,2m)$. By degree reasons,  this spectral sequence degenerates at the $\overbar{E}^2$ page, i.e.\ $\overbar{d}^r=0$ for $r \geq 2$. This completes the proof of the proposition for $p=2$.

Let $p$ be an odd prime. We have 
\begin{equation*} \label{eq E2termforY for odd p}
\overbar{E}^2 \cong \Gamma_{\F_p}(\sigma \tau_0) \otimes \Lambda_{\F_p}(\sigma x) \otimes \Gamma_{\F_p}(\varphi^mx).
\end{equation*}
For every $z\in \Gamma_{\F_p}(\sigma \tau_0) \otimes \Lambda_{\F_p}(\sigma x) \otimes 1 $ and  $1<r<p-1$,
\[\overbar{d}^r(z)= 0  \]
and if $z\neq 0$, then $z$ is not in the image of $\overbar{d}^r$. By Lemmas \ref{lemma differentials for Y} and \ref{lem bokstedt ss for X}, 
\[\overbar{d}^{p-1}\gamma_k(\sigma \tau_0)= \gamma_{k-p}(\sigma \tau_0)\sigma x\]
for every $k \geq p$. See the figure at the end of  Section \ref{subsec bokstedt ss for Y} for an image of the $\overbar{E}^{p-1}$ page of this spectral sequence for $m=2$. For $m>2$ the image would be similar except the fact that the class $\varphi^2 x$ would be replaced by the class $\varphi^m x$.

For the proof of the first statement, we start by noting that every element of total degree less than $md+2$ lies in $\Gamma_{\F_p}(\sigma \tau_0) \otimes \Lambda_{\F_p}(\sigma x) \otimes 1 \subset \overbar{E}^2$.  With the differentials mentioned above, the elements $\gamma_k(\sigma \tau_0) \sigma x$ for $k\geq 0$ represent the trivial class on the $\overbar{E}^\infty$ page and the elements $\gamma_l(\sigma \tau_0)$ for $l \geq p$ do not survive to the $\overbar{E}^\infty$ page. The elements $\gamma_k(\sigma \tau_0)$ for  $k<p$ carry trivial differentials up to the $\overbar{E}^{p-1}$ page; the  rest of the differentials on these elements are also trivial because they have homological degree less than $p$. Also note that there are no elements that could hit them because every element of larger homological degree also have larger internal degree. Therefore, the elements $\gamma_k(\sigma \tau_0)$ for  $k<p$ survive non-trivially to the $\overbar{E}^{\infty}$ page. This shows that below total degree $md+2$, what survives to the $\overbar{E}^{\infty}$ page is the truncation of $\Gamma_{\F_p}(\sigma \tau_0)\otimes 1 \otimes 1$ above total degree $d$. In other words, for total degrees less than $md+2$, we have 
\[\overbar{E}^\infty \cong \fp[\sigma \tau_0]/(\sigma \tau_0^{p}).\]
This completes the proof of the first part of the proposition. 

Now we prove the second part of the proposition. Note that the classes $\gamma_k(\sigma \tau_0) \sigma x$ have odd total degree. Considering this, one observes that the only classes on the $\overbar{E}^2$ page with total degree $md+2$ are $\gamma_{md/2+1}(\sigma \tau_0)$ and $\varphi^mx$. As mentioned before, $\gamma_{md/2+1}(\sigma \tau_0)$ do not survive to the $\overbar{E}^\infty$ page since $md/2+1\geq p$. Therefore, it is sufficient to show that $\varphi^m x$ survives non-trivially to the $\overbar{E}^\infty$ page. Since $\varphi^mx$ has homological degree $2$, all  the differentials on $\varphi^mx$ are trivial. Furthermore, an element that could possibly hit $\varphi^mx$ should have odd total degree and internal degree less than that of $\varphi^mx$. This is the case only for elements of the form $\gamma_k(\sigma \tau_0) \sigma x$. However, these elements gets to carry trivial differentials up to $\overbar{E}^{p-1}$. Furthermore, they are in the image of the differentials on $\overbar{E}^{p-1}$; therefore, they carry trivial differentials on $\overbar{E}^{p-1}$ and after too. We deduce that the elements $\gamma_k(\sigma \tau_0) \sigma x$ cannot hit $\varphi^mx$. Therefore, $\varphi^mx$ survives non-trivially to the $\overbar{E}^\infty$ page.  

\end{proof}

The corollary below provides the classification of type $(\fp,md)$ Postnikov extensions of $X$, see Theorem \ref{thm pstnk dugger shipley}.
\begin{corollary}\label{cor aq cohomology groups of X}
Let $X$ be an $\hz$-algebra as above. There are isomorphisms:
\begin{equation*}
    \begin{split}
\textup{Der}^{md+1}_{\hz}(X,H\fp)/\aut(\fp)&\cong \textup{THH}^{\hfp}_{md+2}(\hfp \wdgz X,H\F_p)/\aut(\fp) \\&\cong \F_p/\aut(\fp)\\ &= \{[0],[1]\}.
\end{split}
\end{equation*}
\end{corollary}
\begin{proof}
 The first isomorphism follows by Theorem \ref{thm aq coh to thhcoh},  Proposition \ref{prop thh coh to thh ho} and Proposition \ref{prop thhchangeofbase}. The second isomorphism follows by Proposition \ref{prop HHofX} and the last isomorphism follows by the fact that $\fp$ is a field. 
\end{proof}

The type $(\fp,md)$ Postnikov extension of $X$ corresponding to the trivial element in the corollary above provides the square zero extension 
\[X \oplus \Sigma^{md}\hfp.\]
We call this the trivial extension. In the homotopy ring of the trivial extension, $x^m=0$. Since our goal is to obtain Postnikov extensions that build up to an $\hz$-algebra with polynomial homotopy ring, we are not interested in this trivial extension. In Proposition \ref{prop inductivestep}, we show that the Postnikov extension corresponding to the non-trivial element do satisfy $x^m \neq 0$ as desired. 

\begin{lemma} \label{prop internaldegreeagree}
 Let $f \co A_* \to C_*$ be a map of non-negatively graded commutative $\fp$-algebras over a non-negatively graded commutative $\fp$-algebra $B_*$. If $f$  is an isomorphism below degree $n$ for some $n>0$, then the induced map
\[ \hh^{\fp}_{*,*}(A_*,B_*) \to \hh^{\fp}_{*.*}(C_*,B_*)\]
is also an isomorphism below internal degree $n$. 
\end{lemma}
\begin{proof}
 Let $A_*^{\otimes \bullet}\otimes B_*$ denote the simplicial ring whose degree $m$ homotopy (i.e.\ degree $m$ homology of its  normalization) is $\hh^{\F_p}_{m,*}(A_*,B_*)$ and let $C_*^{\otimes \bullet}\otimes B_*$ denote the simplicial ring whose degree $m$ homotopy  is $\hh^{\F_p}_{m,*}(C_*,B_*)$. The map $A_* \to C_*$ induces a map of simplicial graded rings $A_*^{\otimes \bullet}\otimes B_* \to C_*^{\otimes \bullet}\otimes B_*$ which  is an isomorphism below internal degree $n$  at each simplicial degree. Since the homotopy of a simplicial graded ring is calculated (internal) degree wise, the induced map in homotopy is an isomorphism below internal degree $n$. 
\end{proof}

\begin{proposition} \label{prop inductivestep}
Let $X$ denote an $H\Z$-algebra whose homotopy ring is $\F_p[x]/(x^m)$ for some $m> 1$ and whose $d$ Postnikov section is equivalent to $Y$ where $d = \lv x \rv = 2p-2$. There is a unique non-trivial $(\F_p,md)$ type Postnikov extension of $X$. The homotopy ring of this non-trivial extension is $\F_p[x]/(x^{m+1})$. Indeed, this is the only Postnikov extension of $X$ with homotopy ring $\F_p[x]/(x^{m+1})$.
\end{proposition}
\begin{proof}
    Due to Corollary \ref{cor aq cohomology groups of X} and Theorem \ref{thm pstnk dugger shipley}, there are only two type $(\fp,md)$ Postnikov extensions of $X$ up to equivalences of Postnikov extensions. This proves the first statement of the proposition. As discussed before the proposition, the trivial extension results in a homotopy ring that is not isomorphic to $\fp[x]/(x^{m+1})$. This shows that the second statement in the proposition implies the third statement. 
    
    Now we prove the second statement of the proposition. 
   Let $\psi \co Z\to X$ denote the unique non-trivial type $(\fp,md)$  Postnikov extension of $X$. This implies that $\pi_*Z\cong \pi_*X \oplus \Sigma^{md}\F_p$ as graded abelian groups. 

The Postnikov section map $\psi \co Z \to X$ is a map of $H\Z$-algebras. Therefore, the $md-2$ truncation of $\pi_*Z$ and $\pi_*X$ are isomorphic as rings. There are only two possible ring structures on $\pi_*Z$. One of them is $\F_p[x]/(x^{m+1})$ and the other one is the square zero extension $\F_p[x]/(x^m) \oplus \Sigma^{md}\F_p$. In order to prove the proposition, we need to show that $\pi_*Z \cong \F_p[x]/(x^{m+1})$ as a ring. We assume to the contrary that there is a ring isomorphism $\pi_*Z\cong \F_p[x]/(x^m) \oplus \Sigma^{md}\F_p$  and obtain a contradiction. 

The contradiction we obtain is the following. Let 
\[k_{md} \in  \text{Der}_{\hz}^{md+1}(X,H\F_p)\cong \fp\] denote a non-trivial derivation resulting in   the Postnikov extension $\psi$ and let $i$ denote the trivial derivation. This means that there is the following homotopy pullback square. 
\begin{equation*} 
 \begin{tikzcd}
 Z \arrow [r]
 \arrow[d,"\psi"]
 & X[0]
 \arrow[d,"i"]
 \\
 X \arrow[r,"k_{md}"]
 &
 X[0] \vee \Sigma^{md+1}H\F_p
 \end{tikzcd}
 \end{equation*}
This shows that the map $\psi \co Z\to X$ pulls back the non-trivial derivation $k_{md}$ to the trivial derivation. Assuming $\pi_*Z \cong \pi_*X \oplus \Sigma^{md}\F_p$ as a ring, we obtain a contradiction by showing that the derivation $\psi^* k_{md}$ given by $k_{md} \circ \psi$ is non-trivial. 

Now we formulate this statement using Hochschild homology. This is similar to the argument in the proof of Lemma \ref{lem first nontrivial differential}. By the naturality of the isomorphism in Theorem \ref{thm aq coh to thhcoh}, $k_{md} \circ \psi$ is nontrivial if  the map 
\[\F_p \cong \thh_{H\Z}^{md+2}(X,H\F_p) \to \thh_{H\Z}^{md+2}(Z,H\F_p)\]
induced by $\psi$ carries $1 \in \fp$ to a non-trivial element on the right hand side.
By     Propositions \ref{prop thh coh to thh ho} and   \ref{prop thhchangeofbase}, it is also sufficient to show that the map 
\begin{equation} \label{eq map to show nontrivial}
    \text{THH}^{H\F_p}_{md+2}(H\F_p \wedge_{H\Z} Z,H\F_p) \to \text{THH}^{H\F_p}_{md+2}(H\F_p \wedge_{H\Z} X,H\F_p)\cong \F_p 
\end{equation}
induced by $\psi$ is nontrivial, i.e.\ $1 \in \F_p$ is in the image of this map. 

Let $\widehat{E}$ denote the B\"okstedt spectral sequence computing $\text{THH}^{H\F_p}(H\F_p \wedge_{H\Z} Z,H\F_p)$ where $\hfp \wdgz Z$ acts on $\hfp$ through the composite map
\[\hfp \wdgz Z \xrightarrow{id \wdgz \psi} \hfp \wdgz X \xrightarrow{\upsilon} \hfp.\]
Here, $\upsilon$ is as in \eqref{eq def of upsilon}.  Recall that   $\overbar{E}$ denotes the spectral sequence calculating \[\text{THH}^{H\F_p}(H\F_p \wedge_{H\Z} X,H\F_p).\] With the composite map above, $\hfp \wdgz \psi$ becomes a map of $\hfp$-algebras over $\hfp$. Therefore, there is an induced map of spectral sequences 
\[\psi_* \co \widehat{E} \to \overbar{E}.\]

As before, we have different arguments for $p=2$ and for odd primes. We start with the $p=2$ case where we have $d=2$. By Lemma \ref{lemma homologyfor p eq 2}, we have: \[\pis(\hft \wdgz Z) \cong \F_2[\xi_1]/(\xi_1^{2m}) \otimes \F_p[z],\] 
for degree less than $2m+1$ where $\lv z \rv = 2m$. By Lemma \ref{prop internaldegreeagree}, this shows that 
\[\widehat{E}^2 \cong \hhfts(\F_2[\xi_1]/(\xi_1^{2m}) \otimes \F_2[z],\F_2) \]
for the internal degrees less than $2m+1$. Furthermore, we have 
\begin{equation*} \label{eq oylebirequation}
   \begin{split}
  &\hhfts(\F_2[\xi_1]/(\xi_1^{2m})\otimes \F_2[z],\F_2)\\
  &\cong \hhfts(\F_2[\xi_1]/(\xi_1^{2m}),\F_2) \otimes \hhfts(\F_2[z],\F_2)\\
  &\cong \Lambda_{\F_2}(\sigma \xi_1) \otimes \Gamma_{\F_2}(\varphi^{2m} \xi_1) \otimes \Lambda_{\F_2}(\sigma z)     
   \end{split} 
\end{equation*}
where the first isomorphism is the natural splitting in Proposition \ref{prop tensorsplitting} and the second isomorphism follows by Propositions \ref{prop tor of quot of polynomial} and \ref{prop tor of freecommutative}. 

 By Lemma \ref{lemma homologyfor p eq 2}, the restriction of the map $\pis(\hft \wdgz Z) \to \pis(\hft \wdgz X)$ to degrees less than $2m+1$ is given by the same restriction of the map \[\F_2[\xi_1]/(\xi_1^{2m}) \otimes \F_2[z] \to \F_2[\xi_1]/(\xi_1^{2m})\] that carries $\xi_1$ to $\xi_1$ and $z$ to $0$. This is $id \otimes \epsilon$ where $id$ is the identity map and $\epsilon$ denotes the augmentation map $\F_2[z] \to \F_2$. By the naturality of the tensor splitting in Proposition \ref{prop tensorsplitting}, we obtain that after restricting to internal degrees less than $2m+1$, the map 
\[\psi_2 \co  \widehat{E}^2 \cong \Lambda_{\F_2}(\sigma \xi_1) \otimes \Gamma_{\F_2}(\varphi^{2m} \xi_1) \otimes \Lambda_{\F_2}(\sigma z) \to \overbar{E}^2 \cong \Lambda_{\F_2}(\sigma \xi_1) \otimes \Gamma_{\F_2}(\varphi^{2m} \xi_1)\]
 is given by $id \otimes \epsilon$ where $id$ is the identity map on $\Lambda_{\F_2}(\sigma \xi_1) \otimes \Gamma_{\F_2}(\varphi^{2m} \xi_1)$ and $\epsilon$  denotes the augmentation map  $\Lambda_{\F_2}(\sigma z) \to \F_2$. We obtain that 
 \begin{equation}\label{eq pis2 psi}
     \psi_2(\varphi^{2m} \xi_1) = \varphi^{2m} \xi_1.
 \end{equation}
 
  As mentioned in the proof of Proposition \ref{prop HHofX}, all the differentials are trivial on $\overbar{E}^r$ for $r\geq 2$ due to degree reasons. Furthermore, $\varphi^{2m} \xi_1 \in \overbar{E}^2$ represents the non-trivial element on the right hand side of \eqref{eq map to show nontrivial}. Since it has homological degree $2$, $\varphi^{2m} \xi_1$ in $\widehat{E}^2$ survives to the $\widehat{E}^\infty$ page, i.e.\ $\hat{d}^r \varphi^{2m} \xi_1=0$ for every $r \geq 2$. This shows that $\varphi^{2m} \xi_1 \in \widehat{E}^2$ survives to the $\widehat{E}^\infty$ page and due to \eqref{eq pis2 psi}, it hits a non-trivial element on the right hand side of \eqref{eq map to show nontrivial}. This finishes the proof of the proposition for $p=2$.

Let $p$ be an odd prime. By Lemma \ref{lemma homology}, there is an isomorphism of rings $\pis(\hfp \wdgz Z) \cong \lambdatau \otimes \pi_*Z$. Therefore, we have 
\[\pis(\hfp \wdgz Z)\cong \lambdatau \otimes (\F_p[x]/(x^m) \oplus \Sigma^{md}\F_p) \cong \lambdatau \otimes \F_p[x]/(x^m) \otimes \F_p[z]\]
for degrees less than $md+1$ where $\lv z \rv = md$. By Lemma \ref{prop internaldegreeagree}, we have 
\begin{equation*}
    \begin{split}
        \widehat{E}^2 &\cong \hhfps(\lambdatau \otimes \F_p[x]/(x^m) \otimes \F_p[z],\F_p)
    \end{split}
\end{equation*}
for internal degrees less than $md+1$. As before, we obtain
\begin{equation} \label{eq oyle iki denklem}
    \begin{split}
        &\hhfps(\lambdatau \otimes \F_p[x]/(x^m) \otimes \F_p[z],\F_p)\\
        & \cong \hhfps(\lambdatau,\F_p) \otimes\hhfps(\F_p[x]/(x^m) ,\F_p) \otimes \hhfps(\F_p[z],\F_p)\\
        &\cong \Gamma_{\F_p}(\sigma \tau_0) \otimes \Lambda_{\F_p}(\sigma x) \otimes \Gamma_{\F_p}(\varphi^m x) \otimes \Lambda_{\F_p}(\sigma z).
    \end{split}
\end{equation}
 As in Lemmas \ref{lemma e2 page of the bokstedt of Y} and \ref{lem bokstedt ss for X}, we have 
\begin{equation} \label{eq oyle uc denklem}
\begin{split}
\overbar{E}^2 &\cong \hhfps(\lambdatau \otimes \F_p[x]/(x^m),\F_p)\\
 & \cong \hhfps(\lambdatau,\F_p) \otimes \hhfps(\F_p[x]/(x^m) ,\F_p)\\
&\cong \Gamma_{\F_p}(\sigma \tau_0) \otimes \Lambda_{\F_p}(\sigma x) \otimes \Gamma_{\F_p}(\varphi^m x).\\
\end{split}
\end{equation}

The Postnikov section map $\psi \co Z\to X$ induce the map 
\begin{equation*}
\begin{split}
    id \otimes \pis \psi \co \pis(\hfp \wdgz Z) \cong \lambdatau &\otimes (\F_p[x]/(x^m) \oplus \Sigma^{md}\F_p) \to \\ &\pis(\hfp \wdgz X) \cong \lambdatau \otimes \F_p[x]/(x^m)
    \end{split}
\end{equation*} 
by Lemma \ref{lemma homology}. Restricting to degrees less than $md +1$, this corresponds to the map 
\[id \otimes \epsilon\co \lambdatau \otimes \F_p[x]/(x^m) \otimes \F_p[z] \to \lambdatau \otimes \F_p[x]/(x^m)\]  where $id$ is the identity map of $\lambdatau \otimes \F_p[x]/(x^m)$ and $\epsilon$ denotes the augmentation map $\F_p[z] \to \F_p$. Recall that the tensor splittings in Equations \eqref{eq oyle iki denklem} and \eqref{eq oyle uc denklem} are natural.  Therefore, the restriction of the map $\psi_2 \co \widehat{E}^2 \to \overbar{E}^2$ to internal degrees less than $md+1$ is given by $id \otimes \epsilon^\prime$ where $id$ is the identity map on $\Gamma_{\F_p}(\sigma \tau_0) \otimes \Lambda_{\F_p}(\sigma x) \otimes \Gamma_{\F_p}(\varphi^m x)$ and $\epsilon^\prime$ denotes the augmentation map  $\Lambda_{\F_p}(\sigma z) \to \F_p$ induced by $\epsilon$. 
In particular, the map $\psi_2\co \widehat{E}^2 \to \overbar{E}^2$ satisfies 
\begin{equation}\label{eq p odd psi}
    \psi_2(\varphi^m x) = \varphi^m x.
\end{equation} 
Due to Proposition \ref{prop HHofX},  $\varphi^m x \in \overbar{E}^2$ represents a non-trivial element on the right hand side of \eqref{eq map to show nontrivial}. Since $\varphi^m x \in \widehat{E}^2$ has homological degree $2$, all the differentials are trivial on $\varphi^m x$ and therefore it survives to the $\widehat{E}^\infty$ page. Due to \eqref{eq p odd psi}, we deduce that the  element on the left hand side of \eqref{eq map to show nontrivial} represented by $\varphi^m x \in \widehat{E}^2$  hits a non-trivial element on the right hand side of \eqref{eq map to show nontrivial}. This provides the contradiction we needed.
\end{proof}

\subsection{Proof of Theorem \ref{thm nontrivial}}

We start with the following construction which provides the first example of a non-formal DGA with homology $\fp[x]$. Recall that for this section, we have $d = \lv x \rv = 2p-2$.

\begin{construction}\label{const nonformal dga}
The $m=2$ case of Proposition \ref{prop inductivestep} states that the $\hz$-algebra $Y$ has a unique type $(\fp,2d)$ Postnikov extension $Z_3$ with homotopy ring $\fp[x]/(x^3)$. In other words, there is a map of $\hz$-algebras
\[Z_3 \to Y\]
inducing the map $\fp[x]/(x^3) \to \fp[x]/(x^2)$ that carries $x$ to $x$ in homotopy. Applying Proposition \ref{prop inductivestep} inductively, on obtains a tower of $\hz$-algebras
\[\cdots \to Z_4  \to Z_3 \to Y \to \hfp\]
where $\pis Z_m \cong \fp[x]/(x^{m})$ and the structure maps induce the ring maps that carry $x$ to $x$ in homotopy (except the last map). Taking the homotopy limit of this tower provides an $\hz$-algebra $Z$ with homotopy ring $\fp[x]$. By construction, the $d$ Postnikov section of $Z$ is $Y$. Since  the $d$ Postnikov section of $Z$ corresponds to a non-formal DGA, we conclude that the DGA corresponding to $Z$ is not formal. 
\end{construction}

For the convenience of the reader, we provide a restatement of Theorem \ref{thm nontrivial}. 
\vspace{0.3cm}

\begin{thmn}[\ref{thm nontrivial}]  There is a unique DGA whose homology is  $\F_p[x_{2p-2} ]$ (with $\lv x_{2p-2} \rv = 2p-2$) and whose $2p-2$ Postnikov section is non-formal. Furthermore for  every $m>1$, there is a unique DGA whose  homology is $\F_p[x_{2p-2} ]/(x_{2p-2}^m)$ and whose  $2p-2$ Postnikov section is non-formal. 

\end{thmn}
\vspace{0.3cm}

 Theorem \ref{thm nontrivial} can also be stated as follows. There is a unique $\hz$-algebra whose homotopy ring is  $\F_p[x]$  and whose $2p-2$ Postnikov section is equivalent to $Y$. Furthermore for  every $m>1$, there is a unique $\hz$-algebra whose  homotopy ring is $\F_p[x]/(x^m)$ and whose  $2p-2$ Postnikov section is equivalent to $Y$.

\begin{proof}
 We start with the first statement. Construction \ref{const nonformal dga} provides an $\hz$-algebra $Z$ with homotopy ring $\fp[x]$ and $d$ Postnikov section equivalent to $Y$. As in page 296 of \cite{basterra2013BP}, we can choose a Postnikov tower of $Z$ 
 \[\cdots \to Z[2d]\to Z[d] \to Z[0]\]
 where each $Z[nd]$ is q-cofibrant as an $\hz$-algebra and the structure maps are q-fibrations. Note that we skip the Postnikov sections between $Z[nd]$ and $Z[(n+1)d]$ because  $\pis(Z)$ is trivial between degrees $nd$ and $(n+1)d$. We also have $Z[d] \simeq Y$.

Let $T$ be an $\hz$-algebra with homotopy ring $\fp[x]$ and $d$ Postnikov section equivalent to $Y$. In order to finish the proof of the first statement in the theorem, we need to show that $T$ is weakly equivalent to $Z$ as an $\hz$-algebra. As before, there is a Postnikov tower of $T$ 
\[\cdots \to T[2d] \to T[d] \to T[0]\]
where the structure maps are q-fibrations. We  construct a weak equivalence between the Postnikov towers of $T$ and $Z$. In other words, we construct weak equivalences $f_i \co Z[id] \we T[id]$ of $\hz$-algebras that make the following diagram commute.
\begin{center}
\begin{tikzcd}
\cdots \ar[r]& \ar[r]\ar[d,swap,"\simeq"]\ar[d,"f_2"] Z[2d] & Z[d]\ar[d,swap,"\simeq"]\ar[d,"f_1"]\ar[r]& \ar[d,swap,"\simeq"]\ar[d,"f_0"]  Z[0]\\
\cdots\ar[r] & T[2d] \ar[r] & T[d] \ar[r] & T[0]
\end{tikzcd}
\end{center}

For the base case, note that $Z[0] \simeq \hfp \simeq T[0]$. Since $Z[0]$ is q-cofibrant and $T[0]$ is q-fibrant, we obtain a weak equivalence $f_0$ of $\hz$-algebras.

We consider $Z[d]$ as an $\hz$-algebra over $T[0]$ through the composite map 
\[Z[d] \to Z[0] \xrightarrow{f_0}T[0].\]
This composite map and the map $T[d] \to T[0]$ are Postnikov extensions of $T[0] \simeq \hfp$ of type $(\fp,d)$ in $\hz$-algebras. Up to weak equivalences of Postnikov extensions, there are only two type $(\fp,d)$ Postnikov extensions of $\hfp$ in $\hz$-algebras \cite[3.15]{dugger2007topological}. Furthermore, the trivial extension provides an $\hz$-algebra that is not weakly equivalence to $Y$ \cite[3.15]{dugger2007topological}. Since  $Z[d]\simeq Y \simeq T[d]$, we deduce that the Postnikov extensions  $T[d] \to T[0]$ and $Z[d] \to T[0]$ are weakly equivalent as  type $(\fp,d)$ Postnikov extensions of $T[0]\simeq \hfp$. In particular, $T[d]$ and $Z[d]$ are weakly equivalent as $\hz$-algebras over $T[0]$. 

The category of $\hz$-algebras over $T[0]$ carries a model structure where the  fibrations, cofibrations and weak equivalence are precisely the morphisms that forget to fibrations, cofibrations and weak equivalences of $\hz$-algebras respectively. In particular, $T[d]$ is fibrant as an $\hz$-algebra over $T[0]$ and $Z[d]$ is cofibrant as an $\hz$-algebra over $T[0]$. Therefore, there is a weak equivalence $f_1\co Z[d] \we T[d]$ of $\hz$-algebras over $T[d]$. This provides the map $f_1$ we needed.

We do induction through the Postnikov towers. For a given $n\geq 1$, assume that there are weak equivalences $f_\ell\co Z[\ell d] \we T[\ell d]$ for every $\ell \leq n$ that commute with the structure maps. We consider $Z[(n+1)d]$ as an $\hz$-algebra over $T[nd]$ through the composite 
\[Z[(n+1)d] \to Z[nd] \we T[nd].\]
 The maps $Z[(n+1)d] \to T[nd]$ and $T[(n+1)d]\to T[nd]$ are Postnikov extensions of $T[nd]$ of type $(\fp,(n+1)d)$. Furthermore, these are non-trivial Postnikov extensions as their homotopy rings are given by $\fp[x]/(x^{n+2})$. Since $T[nd]$ satisfies the hypothesis of Proposition \ref{prop inductivestep},  there is a unique non-trivial type $(\fp,(n+1)d)$ Postnikov extension of $T[nd]$. In particular, $Z[(n+1)d]\to T[nd]$ and $T[(n+1)d]\to T[nd]$ are equivalent as type $(\fp,(n+1)d)$ Postnikov extensions of $T[nd]$. In particular, $Z[(n+1)d]$ and $T[(n+1)d]$  are weakly equivalent as $\hz$-algebras over $T[nd]$. Since $Z[(n+1)d]$ is cofibrant and $T[(n+1)d]$ is fibrant in the model category of $\hz$-algebras over $T[nd]$, we obtain the desired weak equivalence $f_{n+1}$ that make the relevant structure maps commute. 

This finishes the inductive process and we obtain a weak equivalence between the Postnikov towers of $Z$ and $T$. Since $Z$ and $T$ are equivalent to the homotopy limits of their respective Postnikov towers, we obtain that the $\hz$-algebras $T$ and $Z$ are weakly equivalent as desired. This finishes the proof of the first statement of the theorem.

The proof of the second statement of the theorem follows in a similar manner. In Construction \ref{const nonformal dga}, we show that for every $m>1$, there  exists an $\hz$-algebra $Z_m$ with homotopy ring $\fp[x]/(x^{m})$ and $d$ Postnikov section equivalent to $Y$. In this case, the Postnikov tower of $Z_m$ is bounded above. The proof of the uniqueness statement  follows by the same arguments as before except that it involves Postnikov towers with finitely many stages.
 
\end{proof}
\subsection{Proof of Theorem \ref{thm hochschild}} 
We start with a restatement of Theorem \ref{thm hochschild}.
\vspace{0.3 cm}
\begin{thmn}[\ref{thm hochschild}]
Let $X$ be the non-formal DGA with homology $\F_p[x_{2p-2} ]$ given in Theorem \ref{thm nontrivial}. We have  \[\textup{HH}^{\Z}_*(X,\F_p) = \F_p[\mu]/(\mu^p) \ \text{with} \ \lvert \mu \rvert = 2.\]
\end{thmn}
\vspace{0.3 cm}

\begin{proof}
   There are isomorphisms
   \[\hh^{\Z}_*(X, \hfp) \cong \thh^{\hz}_*(X,\hfp) \cong \thh^{\hfp}_*(\hfp \wdgz X, \hfp).\]
   Note that we abuse notation and denote the $\hz$-algebra corresponding to $X$ also by $X$. As usual, we compute the right hand side by using the B\"okstedt spectral sequence. 
   
   We start with the case $p=2$. By Lemma \ref{lemma homologyfor p eq 2}, we have $\pis(\hft \wdgz X) \cong \F_2[\xi_1]$ with $\lv \xi_1 \rv =1$. This shows that the second page of the spectral sequence calculating $\text{THH}^{H\F_2}(H\F_2 \wdgz X,H\F_2)$ is given by  $\hhfts(\F_2[\xi_1],\F_2)\cong \Lambda_{\F_2}(\sigma \xi_1)$ due to  Proposition \ref{prop tor of freecommutative}. All the differentials are trivial by degree reasons and this proves the lemma for $p=2$, i.e.\ we have
\[\hhfts(H\F_2 \wdgz X,\F_2) \cong \F_2[\mu]/(\mu^2) \text{\ with } \lv \mu \rv =2.\]

Let $p$ be an odd prime.  By Lemma \ref{lemma homology}, we have \[\pis(\hfp \wdgz X) \cong \lambdatau \otimes  \F_p[x].\] 
It follows by Propositions \ref{prop tensorsplitting} and \ref{prop tor of freecommutative}  that the spectral sequence calculating \[\text{THH}^{\hfp}(\hfp \wdgz X,H\F_p)\] is given by
\[\overbar{E}^2 \cong \hhfps(\lambdatau \otimes  \F_p[x],\F_p) \cong \Gamma_{\F_p}(\sigma \tau_0) \otimes \Lambda_{\F_p}(\sigma x).\]
Since the $2p-2$ Postnikov section of $X$ is non-formal, and since there is a unique non-formal DGA with homology $\lambdafp(x)$, we deduce that the $2p-2$ Postnikov section of $X$ is equivalent to $Y$. Using the Postnikov section map $X\to Y$ and arguing as in the proof of Lemma \ref{lem bokstedt ss for X}, one sees that there is a map of spectral sequences 
\[\overbar{E}^2\cong \Gamma_{\F_p}(\sigma \tau_0) \otimes \Lambda_{\F_p}(\sigma x) \to E^2 \cong \Gamma_{\F_p}(\sigma \tau_0) \otimes \Lambda_{\F_p}(\sigma x) \otimes \Gamma_{\F_p}(\varphi^2x)\]
given by the canonical inclusion where $E$ denotes the spectral sequence in Lemma \ref{lemma e2 page of the bokstedt of Y}.

Using this map and Lemma \ref{lemma e2 page of the bokstedt of Y}, we deduce that $\overbar{d}_r = 0$ for $1<r<p-1$. Furthermore, Lemma \ref{lemma differentials for Y} implies  that all the non-trivial differentials of $\overbar{E}^{p-1}$ are given by 
\[\overbar{d}^{p-1}\gamma_k(\sigma \tau_0)= \gamma_{k-p}(\sigma \tau_0)\sigma x.\]
where $k \geq p$. This shows that the $\overbar{E}^{\infty}$ page is given by the truncation of $\Gamma_{\F_p}[\sigma \tau_0]$ above total degree $2p-2$. This proves that 
\[\thh^{\hfp}_*(\hfp \wdgz X, \hfp) \cong \fp[\mu]/(\mu^p)\]
as desired.

  \end{proof}

\begin{remark}\label{rmk hh of formal}
It is interesting to compare this with the Hochschild homology groups of the formal DGA $Z$ with homology $\fp[x_{2p-2}]$.  Since $Z$ is an $\fp$-DGA, 
\[\fp \otimes^{\mathbb{L}}_{\z}Z \cong (\fp \otimes^{\mathbb{L}}_{\z} \fp) \otimes^{\mathbb{L}}_{\fp} Z.\]
The homology of the $\fp$-DGA $\fp \otimes^{\mathbb{L}}_{\z} \fp$ is an exterior algebra with a single generator in degree $1$. A straightforward Andr\'e--Quillen cohomology computation shows that there is a unique $\fp$-DGA with such homology, see Theorem \ref{thm pstnk dugger shipley}. We deduce that  the $\fp$-DGA $\fp \otimes^{\mathbb{L}}_{\z} \fp$ is formal. This shows that $\fp \otimes^{\mathbb{L}}_{\z} Z$ is also formal. In particular,  all the differentials in the  B\"okstedt spectral sequence computing $\text{HH}^{\fp}_*(\fp \otimes^{\mathbb{L}}_{\z} Z, \F_p)$ are trivial. We obtain
\[\text{HH}^{\Z}_*(Z, \F_p) \cong \text{HH}^{\fp}_*(\fp \otimes^{\mathbb{L}}_{\z} Z, \F_p) \cong  \Gamma(y) \otimes \Lambda(z)\] where  $\lvert y \rvert =2$ and $\lvert z \rvert = 2p-1$.
\end{remark}
 
\section{Homology calculations} \label{sec homology lemmas}
 To calculate the second page of the B\"okstedt spectral sequence  in various cases, we need to know the ring structure on  $\pi_*(H\F_p \wedge_{H\Z} X)$ where $X$ denotes an $H\Z$-algebra of interest. 
 
 We make use of the Tor spectral sequence of \cite[\rom{4}.4.1]{elmendorf2007rings}.
 Let $R$ be a commutative $\Sp$-algebra and let $K$ and $L$ be $R$-modules. For a spectrum $E$ and a map $\phi$ of spectra, we denote $\pi_*E$ and $\pis \phi$ by $E_*$ and $\phi_*$ respectively. The Tor spectral sequence is given by the following.
\begin{equation*}\label{eq tor spectral sequence}
E^2_{s,t} = \torup^{R_*}_{s,t}(K_*,L_*) \Longrightarrow \pi_{s+t}(K \wedge_{R}L)
\end{equation*}
Note that throughout this section, all the  smash products are assumed to be the corresponding derived smash products.
 \begin{lemma} \label{lemma module homology}
 Let $E$ and $F$ be $H\Z$-modules whose homotopy groups are $\F_p$-modules. The second page of the Tor spectral sequence calculating  $\pi_*(E \whz F)$ is given by 
 \[E^2 \cong \Lambda_{\F_p}(z) \otimes \pi_*E \otimes \pi_*F\]
 where $\text{deg}(z) = (1,0)$. Furthermore, $d^r=0$ for $r \geq 2$.
 
 If $\pi_*E$ and $\pi_*F$ are concentrated in degrees that are multiples of some $d \geq 2$, then there is an isomorphism of graded abelian groups
 \[\pi_*(E \whz F)\cong \Lambda_{\F_p}(z) \otimes \pi_*E \otimes \pi_*F\]
 where $\lv z \rv = 1$. This isomorphism is natural in the following sense. Given $\hz$-module maps $f \co E \to E'$ and $g \co F \to F'$ where $E'_*$ and $F'_*$ are graded $\fp$-modules concentrated in degrees that are multiples of $d$, then the induced map $\pis(f \wdgz g)$ is given by $id \otimes f_* \otimes g_*$ where $id$ is the identity map of $\Lambda_{\F_p}(z)$.
 \end{lemma}
 
 \begin{proof}
  We have the following standard resolution of $\F_p$ as a $\Z$-module
  \[ \cdots \to 0 \to \Z \xrightarrow{.p}  \Z\]
  where the last map denotes multiplication by $p$. To obtain a free resolution of $\pi_*E$, one takes a direct sum of suspended copies of this resolution indexed by a basis of $\pi_*E$. Since $\pi_*F$ is an $\F_p$-module, applying the functor $- \otimes \pi_*F$ to this resolution one obtains the following.
  \[ \cdots \to 0 \to \pi_*E \otimes \pi_*F \xrightarrow{0} \pi_*E \otimes \pi_*F\]
  Therefore, the $E^2$ page of the Tor spectral sequence  calculating $\pi_*( E \whz F)$ is given by 
  \[\Lambda_{\F_p}(z) \otimes \pi_*E \otimes \pi_*F\]
  where $\text{deg}(z) = (1,0)$. There are no non-trivial differentials since the $E^2$ page is concentrated in homological degrees $0$ and $1$. This proves the first statement of the lemma.
  
  Assuming $\pi_*E$ and $\pi_*F$ are concentrated in degrees that are multiples of $d$, $\pi_*E \otimes \pi_*F$ is also concentrated in degrees that are multiples of $d\geq 2$. Therefore, for each total degree, there is at most one  bidegree of $E^2$ with a non-trivial entry. In particular, we have no additive extension problems. 
  The naturality follows by the functoriality of the Tor spectral sequence and the fact that there are no extension problems.
 \end{proof}
\begin{remark}\label{rmk homotopy and maps}
Let $K$ be an $\hz$-module and let $\ho(\hz \modu)$ denote the homotopy category of $\hz$-modules. Since  $\pis(\Sigma^i\hz)$ is free over $\pis \hz \cong \z$, we obtain the following using the Ext spectral sequence in \eqref{thm ext spectral sequence}.
\[\ho(\hz \modu)(\Sigma^i \hz, K) \cong \Hom_{\Z}(\Z, \pi_iK) \cong \pi_i K\]
In other words, for a given class $k \in \pi_iK$,  there is a unique map $\Sigma^i\hz \to K$ in $\ho(\hz \modu)$ that carries $1 \in \pi_i(\Sigma^i \hz) \cong \z$ to $k \in \pi_iK$. In this situation, we say the map $\Sigma^i \hz \to K$ represents $k \in \pi_iK$.
 One obtains a similar correspondence for $\hfp$-modules.  Given an $\hfp$-module $K$ and a class $k \in \pi_i K$, there is a unique map $\Sigma^i \hfp \to K$ in the homotopy category of $\hfp$-modules that carries $1 \in \pi_i (\Sigma^i \hfp) \cong \fp$ to $k \in \pi_i K$.

\end{remark}
\begin{lemma}\label{lemma homology}
Let $d$ be an integer with  $d>2$. Furthermore, let $X$ be an $H\Z$-algebra whose homotopy ring is an $\F_p$-algebra concentrated in degrees that are multiples of $d$.  
 There is an isomorphism of rings \[\pi_*(H\F_p \wedge_{H\Z}X) \cong \Lambda_{\F_p}(\tau_0) \otimes \pi_*X\] where $\lvert \tau_0 \rvert = 1$. 
 
 Let $E$ and $F$ be $H\Z$-algebras satisfying the hypothesis above  and let $\psi \co E \to F$ be a map of $H\Z$-algebras. The induced map  $\pi_*(H\F_p \wedge_{H\Z} \psi)$ is the map given by $id \otimes \pis \psi$ where $id$ denotes the identity map of $\lambdafp(\tau_0)$.
\end{lemma}
\begin{proof}
There is also a zig-zag of Quillen equivalences between the model categories of  $\hfp$-algebras and $\fp$-DGAs \cite{shipley2007hz}. Let $Z$ denote the $H\F_p$-algebra corresponding to the  formal $\fp$-DGA with homology $\pi_*X$; in particular, we have $\pi_* Z \cong \pi_* X$ as rings. Note that  we have the following isomorphism of rings
 \begin{equation}\label{eq homology of z is the correct one}
     \pi_*(H\F_p \wedge_{H\Z} Z) \cong \pi_*((H\F_p \wedge_{H\Z} H\F_p) \wedge_{H\F_p} Z) \cong \lambdatau \otimes \pi_*Z
 \end{equation}
 where the first isomorphism follows by the fact that $Z$ is an $H\F_p$-algebra and the second isomorphism follows by the fact that all the higher Tor terms are trivial in the relevant Tor spectral sequence as $\F_p$ is a field. We start by proving that $X$ is  isomorphic to $Z$ as a monoid in the homotopy category of $\hz$-modules  $\ho(\hz\modu)$.
 
 Since every chain complex is quasi-isomorphic to its homology,  the chain complexes corresponding to $X$ and $Z$ are quasi-isomorphic, i.e.\ $X$ and $Z$ are weakly equivalent as $H\Z$-modules. One can also see this by using the Ext spectral sequence in \eqref{thm ext spectral sequence} for maps of $\hz$-modules  from $X$ to $Z$ and the fact that $\z$ has global dimension $1$. In this spectral sequence, $E_2^{0,0} = \Hom_{\z}(\pi_{-*} X, \pi_{-*} Z)$ survives non-trivially to the $E_\infty$-page. Therefore, one can choose an equivalence $\phi \co X \to Z$ of $H\Z$-modules that induces a  ring isomorphism $\phi_* \co  \pi_* X \to \pi_* Z$. 
 
 Now we show that $\phi$ is a monoid isomorphism in $\ho(\hz \modu)$. For this, we need to show two things. We need to show that this map carries the unit to the unit and that it preserves the multiplication.

 We start by showing that $\phi$ preserves the unit. Let $u_X \co H\Z \to X$ and $u_Z \co \hz \to Z$ in $\ho(\hz \modu)$ denote the unit maps of  $X$ and $Z$ respectively.  Recall from Remark \ref{rmk homotopy and maps} that in $\ho(\hz \modu)$, a map out of $\hz$  is uniquely determined by the image of $1 \in \pi_0 \hz \cong \z$ in homotopy. Therefore, in order to show that $\phi \circ u_X = u_Z$ in $\ho(\hz \modu)$, it is sufficient to show that \[\pis(\phi \circ u_X)(1) =  \pis(u_Z)(1).\] Since $u_X$ and $u_Z$ come from maps of $\hz$-algebras, they  preserve the unit at the level of  homotopy groups. Furthermore, we chose $\phi$ in a way that $\phi_*$ is a ring homomorphism. In particular, $\phi_*$ also sends the unit to the unit. This establishes the identity above and shows that $\phi$ preserves the unit.  
 
 To show that $\phi$ preserves the multiplication, we need to to show that the following diagram commutes in $\ho(\hz \modu)$.
  \begin{equation} \label{diag map of monoids}
    \begin{tikzcd}
     X \wedge_{H\Z} X \arrow[r, "\phi \wedge_{H\Z} \phi"] \arrow[d,"m_X"] &  Z \wedge_{H\Z} Z \arrow[d,"m_Z"]
    \\
     X \ar[r,"\phi"] &  Z
    \end{tikzcd}
\end{equation} 
 All the smash products are derived as usual and $m_X$ and $m_Z$ denote the multiplication maps of $X$ and $Z$ respectively. To show that this diagram commutes, we need to identify the homotopy classes of maps from $X \wdgz X$ to $Z$ in $\hz$-modules. 
 
 There is an adjunction between the homotopy categories of $\hz$-modules and $\hfp$-modules where the left adjoint 
 \[\hfp \wdgz - \co \ho(\hz \modu) \to \ho(\hfp \modu)\]
 is the extension of scalars functor and the right adjoint is given by the forgetful functor. Note that the extension of scalars functor is strong monoidal due to the  equality 
  \begin{equation} \label{eq strongmonoidalsmash}
     H\F_p \wedge_{H\Z}( E \wedge_{H\Z} F) \cong (H\F_p \wedge_{H\Z} E) \wedge_{H\F_p} (H\F_p \wedge_{H\Z} F)
 \end{equation}
 that holds for every pair of $\hz$-modules $E$ and $F$.
 
 Since $Z$ is an $\hfp$-module, we obtain the following.  
 \begin{equation}\label{eq derived extension of scalars adjunction} 
 \begin{split}
 \text{Ho}(H\Z\modu) (X\wedge_{H\Z} X, Z) &\cong \text{Ho}(H\F_p\modu)(H\F_p \wedge_{H\Z}(X \wedge_{H\Z} X),Z)\\
 & \cong \text{Ho}(H\F_p\modu)((H\F_p \wedge_{H\Z} X) \wedge_{H\F_p} (H\F_p \wedge_{H\Z} X),Z)
 \end{split}
 \end{equation}
 Here, the first isomorphism is due to the adjunction and the second isomorphism follows by \eqref{eq strongmonoidalsmash}. 
By Lemma \ref{lemma module homology}, there is the following isomorphism of graded abelian groups. 
 \begin{equation*} \label{eq modulehomologyofX}
 \pi_*(H\F_p \wedge_{H\Z} X)\cong \lambdatau \otimes \pi_*X
 \end{equation*}

 Since $\F_p$ is a field, the Tor spectral sequence shows that smashing over $H\F_p$ results in a tensor product in homotopy. We obtain  the following isomorphism of graded abelian groups.
 \[\pi_* ((H\F_p \wedge_{H\Z} X) \wedge_{H\F_p} (H\F_p \wedge_{H\Z} X)) \cong \lambdatau \otimes \pi_* X \otimes \lambdatau \otimes \pi_*X\]
 Furthermore, the set of homotopy classes of maps between  $H\F_p$-modules are simply given by the set of maps of homotopy groups. This follows by the Ext spectral sequence calculating the homotopy classes of maps of $H\fp$-modules; this spectral sequence has trivial terms in non-zero cohomological degrees because $\F_p$ is a field.  Using this fact and the identification in \eqref{eq derived extension of scalars adjunction}, we deduce the following.
 \begin{equation*} \label{eq homsetforXtoZ}
     \text{Ho}(H\Z\modu) (X\wedge_{H\Z} X, Z) \cong \text{Hom}_{\F_p}(\lambdatau \otimes \pi_* X \otimes \lambdatau \otimes \pi_*X, \pi_*Z)
 \end{equation*}

Let $f$ and $g$ denote the maps on the right hand side corresponding to the maps $m_Z \circ (\phi \wdgz \phi)$ and $\phi \circ m_X$ of Diagram \eqref{diag map of monoids} respectively. In order to prove that Diagram \eqref{diag map of monoids} commutes, i.e.\ in order to deduce that $\phi$ is an isomorphism of monoids, it is sufficient to  show that $f = g$. 

In $\lambdatau \otimes \pi_* X \otimes \lambdatau \otimes \pi_*X$, the elements of the form $\tau_0 \otimes x_1 \otimes \tau_0 \otimes x_2$,   $\tau_0 \otimes x_1 \otimes 1 \otimes x_2$ and  $1 \otimes x_1 \otimes \tau_0 \otimes x_2$ for $x_1, x_2 \in \pi_*X$ lie in degrees, $2+dm$, $1+dm$ and $1+dm$ for some $m$ respectively. Since $\pi_*Z$ is concentrated in degrees that are multiples of $d$ and because $d \geq 3$, all the elements mentioned above are necessarily mapped to zero by both $f$ and $g$. In particular, $f$ and $g$ agree on elements with a $\tau_0$ factor. 

Therefore, it is sufficient to show that $f$ and $g$ agree  for elements of the form $1 \otimes x_1 \otimes 1 \otimes x_2$ for every $x_1,x_2 \in \pi_*X$.

The adjoint of the map $m_Z \circ (\phi \wdgz \phi)$, which induces the map $f$ in homotopy, is given by the following composite.
\begin{equation} \label{eq composinghorthenvert}
\begin{split}
    (H\F_p \wedge_{H\Z} X) \wedge_{H\F_p} (H\F_p \wedge_{H\Z} X) &\to (H\F_p \wedge_{H\Z} Z) \wedge_{H\F_p} (H\F_p \wedge_{H\Z} Z)\\
    &\to \hfp \wdgz Z \to Z
    \end{split}
\end{equation}
Here, the first map is the canonical map induced by $\phi$, the second map is the multiplication map of $H\F_p \wedge_{H\Z} Z$ and the third map is the $\hfp$-algebra  structure map of $Z$. Using the functoriality of the Tor sequence, we obtain that the induced ring map at the level of homotopy groups by the map  of $\hz$-algebras:
 \begin{equation}\label{eq inclusion of rings}
     X \cong H\Z \wedge_{H\Z} X \to H\F_p \wedge_{H\Z} X,
 \end{equation}
 is the canonical inclusion $\pi_* X \to \lambdatau \otimes \pi_*X$. One obtains a similar result by replacing $X$ with $Z$ in the map above. Because the following diagram commutes, 
  \begin{equation} \label{diag homology of phi}
    \begin{tikzcd}
     H\Z \wedge_{H\Z} X \arrow[r, "H\Z \wedge_{H\Z} \phi"] \arrow[d] & H\Z \wedge_{H\Z} Z   \arrow[d]
    \\
     H\F_p \wedge_{H\Z} X \ar[r,"H\F_p \wedge_{H\Z} \phi"] &  H\F_p \wedge_{H\Z}  Z
    \end{tikzcd}
\end{equation} 
we deduce that the map \[\pi_*(H\F_p \wedge_{H\Z} \phi) \co \lambdatau \otimes \pi_*X \to \lambdatau \otimes \pi_*Z\] carries $1 \otimes x$ to $1 \otimes \phi_*(x)$. Therefore, the first map in \eqref{eq composinghorthenvert} carries $1 \otimes x_1 \otimes 1 \otimes x_2$ to $1 \otimes \phi_*(x_1) \ot 1 \ot \phi_*(x_2)$. Due to \eqref{eq homology of z is the correct one}, the second map carries the latter element to $1 \otimes \phi_*(x_1) \phi_*(x_2) = 1\otimes \phi_*(x_1x_2)$ and the third map  carries this element to $\phi_*(x_1x_2)$. We deduce that $f(1 \otimes x_1 \otimes 1 \ot x_2) = \phi_*(x_1 x_2)$.

Therefore, we need to show that $g$ also carries $1 \otimes x_1 \otimes 1 \otimes x_2$ to $\phi_*(x_1x_2)$. The adjoint of $\phi \circ m_X$, which induces $g$ in homotopy, is given by the composite map 
\begin{equation*}
    \begin{split}
        (\hfp \wdgz X) \wdgfp (\hfp \wdgz X) \to \hfp \wdgz X \xrightarrow{\hfp \wdgz \phi} \hfp \wdgz Z \to Z
    \end{split}
\end{equation*}
where the first map is the multiplication map of $\hfp \wdgz X$ and the last map is the $\hfp$-algebra structure map of $Z$.  Using \eqref{eq inclusion of rings} as before, we deduce that the ring structure on $1 \otimes \pis X \subset \pis(\hfp \wdgz X)$ is given by the ring structure on $\pis X$. In particular, we deduce that the first map in the composite above carries $1 \otimes x_1 \otimes 1 \otimes x_2$ to $1 \otimes x_1 x_2$ at the level of homotopy groups. Due to \eqref{diag homology of phi}, this element is  carried to $1 \otimes \phi_*(x_1x_2)$ by the second map $\hfp \wdgz \phi$. Due to \eqref{eq homology of z is the correct one}, the last map in the composite carries $1 \otimes \phi_*(x_1 x_2)$ to $\phi_*(x_1x_2)$. We deduce that the induced map at the level of homotopy groups by the  composite above, and therefore $g$, carries $1 \ot x_1 \ot 1 \ot x_2$ to $\phi_*(x_1 x_2)$. This shows that $f$ and $g$ agree and that  Diagram \eqref{diag map of monoids} commutes. Therefore, $\phi$ is an isomorphism of monoids between $X$ and $Z$ in $\ho(\hz \modu)$.

 Since $\hfp \wdgz - $ is a strong monoidal functor, we deduce that $\hfp \wdgz X$ and $\hfp \wdgz Z$ are isomorphic as monoids in $\ho(\hfp \modu)$. Since the homotopy rings of $\hfp$-algebras are determined by their isomorphism classes as  monoids in $\text{Ho}(\hfp \modu)$, we deduce that $\pis(\hfp \wdgz X)$ is isomorphic to $\pis(\hfp \wdgz Z)$ as rings. Using \eqref{eq homology of z is the correct one}, we obtain the desired  ring isomorphism 
 \[\pis (\hfp \wdgz X) \cong \lambdafp(\tau_0) \ot \pis X.\]
 
The naturality of this isomorphism follows by the naturality result in Lemma \ref{lemma module homology}. 
\end{proof}
The following lemma provides a counter example to the  lemma above for $d=2$; this is obtained by using $X=Y$ in the first item  of the lemma below. 

We use the following lemma to prove the $p=2$ case of Theorem \ref{thm nontrivial}. Note that in this section, we do not assume the existence of $X$ or $Z$ satisfying the hypothesis of the items in the lemma below for $m>2$. However, this does not cause a problem since the lemma would be vacuously  true if such objects did not exist.

\begin{lemma} \label{lemma homologyfor p eq 2}
Let $Y$ denote the $H\Z$-algebra corresponding to the non-formal DGA given in Example \ref{ex duggershipley} for $p=2$. Also, let $X$ and $Z$ be $H\Z$-algebras with $2$ Postnikov section equivalent to $Y$. For   $m>1$ and $\lv x \rv = 2$, we have the following.
\begin{enumerate}
    \item If the homotopy ring of $X$ is  $\F_2[x]/(x^m)$, then there is an isomorphism of rings \[\pi_*(H\F_2 \wedge_{H\Z} X) \cong \F_2[\xi_1]/(\xi_1^{2m})\]  where $\lv \xi_1 \rv = 1$.
    \item If the homotopy ring of $X$ is   $\F_2[x]$, then there is an isomorphism of rings \[\pi_*(H\F_2 \wedge_{H\Z} X) \cong \F_2[\xi_1]\]  where $\lv \xi_1 \rv = 1$.
    \item Assume that there is an isomorphism   $\pis Z \cong \F_2[x]/(x^m) \oplus \F_2w$ of rings where $\lv w \rv = 2m$ and $\oplus$ denotes the square zero extension.  There is  a map of rings \[\F_2[\xi_1]/(\xi_1^{2m}) \otimes \F_2[w] \to \pi_*(H\F_2 \wedge_{H\Z} Z)  \]
    whose restriction to degrees less than $2m+1$ is an isomorphism, where $\lv \xi_1\rv=1$. 
    
    Furthermore, the ring map $\pi_*(H\F_2 \wedge_{H\Z} Z) \to \pi_*(H\F_2 \wedge_{H\Z}  Z[(m-1)2])$ induced by the Postnikov section map $Z \to Z[(m-1)2]$ agrees with the ring map  \[ \F_2[\xi_1]/(\xi_1^{2m}) \otimes \F_2[w] \to \F_2[\xi_1]/(\xi_1^{2m})\] that carries $\xi_1$ to $\xi_1$ and $w$ to $0$ on degrees below $2m+1$. 
    \end{enumerate}
    
\end{lemma}
\begin{proof}

We start by proving the first part of the lemma for $X \simeq Y$, i.e.\ we prove the case $m=2$ of the first part. 

Let $E$ be the $H\Z$-algebra $H\Z \wedge H\F_2$. This is the $H\Z$-algebra obtained from $H\F_2$ through the functor $H\Z \wedge - \co \Sp\text{-algebras} \to H\Z \text{-algebras}$. We have 
\begin{equation}\label{eq dsa for p is 2}
\pis(H\F_2 \wedge_{H\Z} E) \cong \pis(H\F_2 \wedge_{H\Z} (H\Z \wedge H\F_2)) \cong \pis(H\F_2 \wedge H\F_2) \cong \F_2[\xi_i \lvert i \geq 1]
\end{equation}
 where $\lv \xi_i \rv = 2^{i}-1$. Note that this is the dual Steenrod algebra. Furthermore, we have
\[\pi_*E = \pi_* (H\Z \wedge H\F_2) \cong \F_2[\xi_1^2] \otimes \F_2[\xi_i \lvert i \geq 2] \]
where $\lv \xi_i \rv = 2^{i}-1$ for $i \geq 2$ as before and $\lv \xi_1^2 \rv = 2$. 

Let $\psi \co E \to E[2]$ be the $2$ Postnikov section of $E$. In homotopy, this induces the map 
\[\pis(\psi) \co \F_2[\xi_1^2] \otimes \F_2[\xi_i \lvert i \geq 2] \to \lambdaft(x)\]
that satisfies $\pis(\psi)(\xi_1^2) = x$.

 By Lemma \ref{lemma module homology}, the map 
 \[\hft \wdgz \psi \co \hft \wdgz E \to \hft \wdgz E[2]\]
 induces a map of Tor spectral sequences given by the map 
\begin{equation}\label{eq map of tor ss p is 2}
     id \otimes \pis(\psi) \co  \Lambda_{\F_2}(t) \otimes \F_2[\xi_1^2] \otimes \F_2[\xi_i \lvert i \geq 2] \to \Lambda_{\F_2}(t) \otimes \Lambda_{\F_2}(x)
\end{equation}
on the second page where $id$ denotes the identity map of $\lambdaft(t)$ and $\text{deg}(t) = (1,0)$. Furthermore, both spectral sequences degenerate on the second page due to degree reasons.  The spectral sequence on the right hand side provides an isomorphism
\[\pi_*(H\F_2 \wedge_{H\Z} E[2]) \cong \Lambda_{\F_2}(z) \otimes \Lambda_{\F_2}(x) \]
of graded abelian groups where $\lv z \rv = 1$. Note that the spectral sequence on the left computes $\F_2[\xi_i \lvert i \geq 1]$, see \eqref{eq dsa for p is 2}.

The element $t$ on the left hand side of \eqref{eq map of tor ss p is 2} is the only class of total degree $1$. Therefore, it represents $\xi_1 \in \F_2[\xi_i \lvert i \geq 1]$. Similarly, $\xi_1^2$ on the left hand side represents $\xi_1^2 \in \F_2[\xi_i \lvert i \geq 1]$. Since the map in \eqref{eq map of tor ss p is 2} carries $t$ to $t$ and $\xi_1^2$ to $x$, we obtain that the ring map 
\[\pis(\hft \wdgz \psi) \co \pis(\hft \wdgz E)\cong \F_2[\xi_i \lvert i \geq 1]  \to \pis(\hft \wdgz E[2])\]
carries $\xi_1$ to $z$ and $\xi_1^2$ to $x$ as there are no extension problems. We obtain that $z^2 = x$ in $\pi_*(H\F_2 \wedge_{H\Z} E[2])$. Furthermore, by Lemma \ref{lem claim 1} below, $zx \neq 0$ in $\pis(\hft \wdgz E[2])$. Therefore, $zx = z^3$ is the non trivial element denoted by $z \otimes x$ above. This recovers the ring structure on $\pi_*(H\F_2 \wedge_{H\Z} E[2])$, we obtain an isomorphism of rings 
\begin{equation}\label{eq homology of e2 for p is 2}
    \pi_*(H\F_2 \wedge_{H\Z} E[2]) \cong \F_2[\xi_1]/(\xi_1^4).
\end{equation}

  Now, we show that $E[2]$ is weakly equivalent as an $H\Z$-algebra to $Y$. As described in Example \ref{ex duggershipley}, there are only two $H\Z$-algebras with homotopy ring $\Lambda_{\F_2}(x)$; one of these is the $H\Z$-algebra corresponding to the formal $\f_2$-DGA with  homology $\lambdaft(x)$ and the other one is $Y$. If $E[2]$ is the $H\Z$-algebra corresponding to the formal $\f_2$-DGA, then it is an $H\F_2$-algebra. In that case, we have the following isomorphisms of rings 
\begin{equation*} 
\pi_*(H\F_2 \wedge_{H\Z} E[2]) \cong \pi_*((H\F_2 \wedge_{H\Z} H\F_2) \wedge_{H\F_2} E[2]) \cong \Lambda_{\F_2}(z) \otimes \Lambda_{\F_2}(x).
\end{equation*}
However, this contradicts \eqref{eq homology of e2 for p is 2}. Therefore,  $E[2]$ is not the $\hz$-algebra corresponding to the formal $\f_2$-DGA with homology $\lambdaft(x)$; we deduce that $E[2] \simeq Y$ as $\hz$-algebras. This, together with \eqref{eq homology of e2 for p is 2}, provides the desired isomorphisms of rings
\begin{equation*}
    \pi_*(H\F_2 \wedge_{H\Z} Y) \cong \pi_*(H\F_2 \wedge_{H\Z} E[2]) \cong \F_2[\xi_1]/(\xi_1^4).
\end{equation*}
This finishes the proof of the first part of the lemma for $m=2$.

We prove the first part by doing induction over $m$ in $\pi_*X = \F_2[x]/(x^m)$. The proof of the $m=2$ case is given above. Assume that the first item of the lemma is true for some $m \geq 2$ and we are given an $H\Z$-algebra $X$ with $\pi_*X = \F_2[x]/(x^{m+1})$ and 2 Postnikov section equivalent to $Y$. Note that by Lemma \ref{lemma module homology}, we have an isomorphism of graded abelian groups
\begin{equation*}
    \pi_*(H\F_2 \wedge_{H\Z} X) \cong \Lambda_{\F_2}(z) \otimes \pi_*X
\end{equation*}
where $\lvert z \rvert = 1$. Our goal is to understand the ring structure on $\lambdaft(z) \ot \pis X$. Note that the $\hz$-algebra map 
\begin{equation}\label{eq 123456}
X \cong H\Z \wedge_{H\Z} X \to H\F_2 \wedge_{H\Z} X
\end{equation}
induces a ring map $\pi_*X \to \Lambda_{\F_2}(z) \otimes \pi_*X$. By the functoriality of the Tor spectral sequence, it is clear that this is the canonical inclusion. This shows that the ring structure on $\pi_*X \subset \Lambda_{\F_2}(z) \otimes \pi_*X$ is given by that of $\pis X$.

Let $\psi$ denote the Postnikov  section map $X \to X[(m-1)2]$. By the induction hypothesis, the statement in the lemma is true for $X[(m-1)2]$. Furthermore by Lemma \ref{lemma module homology}, the following ring map
\begin{equation} \label{eq 19876}
    \pi_*(H\F_2 \wedge_{H\Z} X) \to \pi_*(H\F_2 \wedge_{H\Z} X[(m-1)2]) \cong \F_2[\xi_1]/(\xi_1^{2m})
\end{equation}
is an isomorphism below degree $2m$. This shows that $ \pi_*(H\F_2 \wedge_{H\Z} X) $ has the desired ring structure in degrees less than $2m$. In particular, we have $z^2 =x$. The only non-trivial classes of $\pi_*(H\F_2 \wedge_{H\Z} X)$ in degrees larger than $2m-1$ are $x^m$ and $z \otimes x^m$.  Therefore, it is sufficient to show that $z^{2m} = x^m$ and $z^{2m+1} = z \otimes x^m$. Since the ring structure on $\pis X \subset \lambdaft(z) \ot \pis X$ is given by that of $\pis X$, $z^{2m} = (z^2)^m = x^m$ as desired. Now we show that $z^{2m+1} = z \otimes x^m$.  Since $z \otimes x^m$ is the only non-trivial class in its degree, and since $z^{2m+1} = z z^{2m}= zx^m$, it is sufficient to show that $z x^m$ is non-trivial. This follows by Lemma \ref{lem claim 1} below. This finishes the proof of the first part of the lemma. 

Now we prove the second part of the lemma; let  $\pis X= \ft[x]$. By Lemma \ref{lemma module homology}, the map of rings 
\[\pis(\hft \wdgz X) \to \pis(\hft \wdgz X[m2])\]
induced by the Postnikov section $X \to X[m2]$ is an isomorphism below degree $2m+2$. For every $m>1$, we already established the ring structure on the right hand side as $X[m2]$ satisfies the hypothesis of the lemma and  $\pis(X[m2]) \cong \ft[x]/x^{m+1}$. As we can choose $m$ to be arbitrarily  large, this establishes the ring structure on $\pis(\hft \wdgz X)$ as desired. This finishes the proof of the second part of the lemma.

Now we prove the third part using the first part of the lemma. By Lemma \ref{lemma module homology}, there is an isomorphism of graded abelian groups 
\[\pi_*(H\F_2 \whz Z) \cong \Lambda_{\F_2}(z) \otimes \pi_*Z\]
where $\lv z \rv = 1$ and $\pis Z = \F_2[x]/(x^m) \oplus \ft w$ with $\lv w \rv = 2m$. As in Equation \eqref{eq 123456}, the map $Z \to H\F_2 \whz Z$ shows that the ring structure on $\pi_*Z \subset \Lambda_{\F_2}(z) \otimes \pi_*Z$ is given by the ring structure on $\pi_*Z$. Furthermore,  we consider the Postnikov section map $Z \to Z[(m-1)2]$. Note that $Z[(m-1)2]$ satisfies the hypothesis of the first part of the lemma. Arguing as in \eqref{eq 19876}, we deduce that the ring structure on $\pi_*(H\F_2 \whz Z)$ below degree $2m$ is given by $\F_2[\xi_1]/(\xi_1^{2m})$ where $x \in \pis Z \subset \pis(\hft \wdgz Z)$ corresponds to  $\xi_1^2$. Since $x^m = 0$ in $\pis Z$, we deduce that $\xi_1^{2m} = 0$ in $\pis(\hft \wdgz Z)$. This shows that there is a map  \[\F_2[\xi_1]/(\xi_1^{2m}) \otimes \F_2[w] \to \pi_*(H\F_2 \whz Z) \]
of rings that carries $\xi_1$ to $z$ and $w$ to $w$. Furthermore, this map
induces an isomorphism  \[\F_2[\xi_1]/(\xi_1^{2m}) \otimes \F_2[w] \cong \pi_*(H\F_2 \whz Z) \]
of rings  after restricting to degrees less than $2m+1$. 

For the last statement of the third part of the lemma, note that $Z[(m-1)2]$ satisfies the hypothesis of the first part. It follows by Lemma \ref{lemma module homology} that, after restricting to degrees less than $2m+1$, the ring map induced by the  Postnikov section map $Z \to Z[(m-1)2]$ is given by the same restriction of the map of rings
\[\pi_*(H\F_2 \whz Z) \cong \F_2[\xi_1]/(\xi_1^{2m}) \otimes \F_2[w] \to \pi_*(H\F_2 \whz Z[(m-1)2]) \cong \F_2[\xi_1]/(\xi_1^{2m})\]
carrying $\xi_1$ to $\xi_1$ and $w$ to $0$. This finishes the proof of the third part of the lemma. 
 \end{proof}
 What is left to prove is the following lemma.
\begin{lemma}\label{lem claim 1}
 Let $X$ be an $H\Z$-algebra with homotopy ring $\F_2[x]/(x^{m+1})$ for some $m \geq 1$ where $\lv x \rv = 2$. Furthermore,  let 
\[\pi_*(H\F_2 \whz X) \cong \Lambda_{\F_2}(z) \otimes \pi_*X \]
be an identification of graded abelian groups provided by the Tor spectral sequence, see Lemma \ref{lemma module homology}. Under this identification, the ring structure on $\pis (\hft \wdgz X)$ satisfies $zx^m \neq 0$. 
\end{lemma}
\begin{proof}

The multiplication map of $H\F_2 \whz X$ is given by the composite 
\begin{equation*}\label{eq multiplicationforp eq 2 homology}
    (H\F_2 \whz X) \wedge_{H\F_2} (H\F_2 \whz X) \cong H\F_2 \whz (X \whz X) \to H\F_2 \whz X
\end{equation*}
where the map on the right is induced by the multiplication map 
\begin{equation*}
    m_X \co X \wdgz X \to X
\end{equation*}
of $X$.

In order to understand the product of  $z$ with $x^m$, we choose a morphism \[z \co \Sigma H\F_2 \to H\F_2 \whz X\] of $\hft$-modules representing the unique non-trivial class $z$ in $\pi_1(H\F_2 \whz X)$ and a map  
\[x^m \co \Sigma^{2m} \Z  \to X\]
of $\hz$-modules representing the unique non-trivial class $x^m$ in $\pi_{2m}X$, see Remark \ref{rmk homotopy and maps}. Using the induced map of Tor spectral sequences, one sees that 
\[id_{\hft}\wdgz x^m \co \hft \wdgz \Sigma^{2m} \hz \to \hft \wdgz X\]  represents the unique non-trivial degree $2m$ element $1 \ot x^m$ in $\pis(\hft \wdgz X)$. Note that $id_{\hft}$ denotes the identity map of $\hft$ and we have $\hft \wdgz \Sigma^{2m} \hz \cong \Sigma^{2m}\hft$.  The product \[zx^m \co \Sigma^{2m+1} H\F_2 \to H\F_2 \whz X\]
is given by composing vertically and then horizontally and then vertically twice in the following commuting diagram.  
\begin{equation} \label{eq 123}
    \begin{tikzcd}[row sep=normal, column sep = large]
    \Sigma^{2m+1} H\F_2  \arrow[d,"\cong"] & \\
      \Sigma H\F_2 \wedge_{H\F_2}(H\F_2 \whz \Sigma^{2m}H\Z) \arrow[r, "z \wedge_{H\F_2}(id_{H\F_2} \whz x^m)"] \arrow[d,"\cong"]  &(H\F_2 \whz X) \wedge_{H\F_2} (H\F_2 \whz X)   \arrow[d,"\cong"]
    \\
     \Sigma H\F_2 \wedge_{H\Z} \Sigma^{2m} H\Z  \ar[r,"z \whz x^m"] & (H\F_2 \whz X) \whz X   \cong H\F_2 \whz (X \whz X) \ar[d, "id_{H\F_2} \whz m_X"]\\
     &  H\F_2 \whz X
    \end{tikzcd}
\end{equation} 
Here, the vertical isomorphisms are given by the cancellation of $H\F_2$ by $\wedge_{H\F_2}$. When we say the class represented by $z \wdgz x^m$ in homotopy, we mean the image of $1 \in \pi_{2m+1}(\Sigma H\F_2 \whz \Sigma^{2m}H\Z)\cong \ft$ under the map $\pi_{2m+1}(z \wdgz x^m)$. This is consistent with the notation of Remark \ref{rmk homotopy and maps} since $z \wdgz x^m$ is a map of $\hft$-modules . Note that $z \wdgz x^m$ is a map of $\hft$-modules because $z$ is a map of $\hft$-modules. Similarly, when we say the class represented by $ (id_{\hft} \wdgz m_X) \circ (z \wdgz x^m)$ in homotopy, we mean the image of $1 \in \pi_{2m+1}(\Sigma H\F_2 \whz \Sigma^{2m}H\Z)$ under the map $\pi_{2m+1}((id_{\hft} \wdgz m_X) \circ (z \wdgz x^m))$. 

Using the Tor spectral sequence for $\hft$-modules, one observes that the smash product of $\hft$-modules denoted by $\wdgft$ results in a tensor product at the level of homotopy groups. Using this, we deduce that  $z \wedge_{H\F_2}(id_{H\F_2} \whz x^m)$ hits a non-trivial class in homotopy. Due to the commuting square above, we deduce that  $z \wdgz x^m$ represents a non-trivial class in homotopy.

In order to prove the lemma, i.e.\ in order to show that $zx^m$ represents a non-trivial class in homotopy groups, it is sufficient to show that the map 
\begin{equation*}\label{eq claim map to show nontrivial}
    (id_{\hft} \wdgz m_X) \circ (z \wdgz x^m) \co \Sigma H\F_2 \whz \Sigma^{2m}H\Z \to H\F_2 \whz X
\end{equation*} 
represents a non-trivial class in $\pi_{2m+1}(\hft \wdgz X)$.

Again due to Lemma \ref{lemma module homology}, we have 
\begin{equation}\label{eq homotopy of x wdgz x}
\pis(X\wdgz X) \cong \lambdaft(v) \ot \pis X \ot \pis X
\end{equation}
where $\lv v \rv = 1$. Let 
\[\alpha \co \Sigma^{2m+1}\hz \to X \wdgz X\]
denote a map of $\hz$-modules representing the class $v \ot 1 \ot x^{m}$. We consider the following diagram 

\begin{equation}\label{diag last diagram}
    \begin{tikzcd}
    \Sigma^{2m+1}\hz \ar[d,"\alpha"]  & &\\
      \hz \wdgz X \wdgz X \ar[d,"f"]  & &\\
    \hft \wdgz X \wdgz X \ar[r,"p"] & \hft \wdgz \hft \wdgz X \ar[r,"\ell"]& \hft \wdgft \hft \wdgz X\\
    \Sigma \hft \wdgz \Sigma^{2m}\hz \ar[u,"z \wdgz x^m"] & & &\\
    \end{tikzcd}
\end{equation}
where $f$ is induced by the map $\hz \to \hft$, $p$ is induced by the Postnikov section map $X \to \hft$ and $\ell$ is the canonical map. 

Note that the composite map $\ell\circ p\circ f $ is isomorphic to the canonical map 
\[X \wdgz X \to \hft \wdgz X\]
induced by the Postnikov section map $X \to \hft$.
Due to Lemma \ref{lemma module homology}, this map carries the class $v \otimes 1 \otimes x^m$ to the non-trivial class $z \otimes x^m$. In particular, we deduce that $\ell \circ p$ carries the homotopy class corresponding to $f \circ \alpha$ to a non-trivial element. 

On the other hand, the homotopy class corresponding to $z \wdgz x^m$ gets carried to a trivial class by $\ell \circ p$ as the following composite map of $\hft$-modules is null-homotopic, see Remark \ref{rmk homotopy and maps}.
\[\Sigma \hft \xrightarrow{z} \hft \wdgz X \to \hft \wdgz \hft \to \hft \wdgft \hft\cong \hft \]
We deduce that $f \circ \alpha$ and $z \wdgz x^m$ represent distinct elements in $\pis(\hft \wdgz X \wdgz X)$.

Due to \eqref{eq homotopy of x wdgz x}, the Tor spectral sequence computing $\hft \wdgz (X \wdgz X)$ is given by the following: 
\begin{equation*}
    F^2 \cong \lambdaft(t) \ot \lambdaft(v) \ot \pis X \ot \pis X \Longrightarrow \pis(\hft \wdgz (X \wdgz X))
\end{equation*}
where $\deg(t)= (1,0)$ and $\deg(v) = (0,1)$. Note that all the differentials of our Tor spectral sequences are trivial on the second page and after as the global dimension of $\z$ is $1$. We need to understand the class in $F^2$ that represents the homotopy class $z \wdgz x^m$.

Considering the map of Tor spectral sequences induced by $f$, one observes that the homotopy class corresponding to $f \circ \alpha$ is represented by $1 \ot v \ot 1 \ot x^m$ on $F^2$. Since it represents an element distinct from $f \circ \alpha$, and since there are no extension problems, we deduce that $z \wdgz x^m$ is represented by a class  that is different than $1 \ot v \ot 1 \ot x^m$. 


The map 
\[\zeta \co H\F_2 \whz( X \whz X) \to H\F_2 \whz( X \whz X[2m-2])\]
induced by the  Postnikov section map $X\to X[2m-2]$ induces the canonical map \[\zeta_2 \co F^2 \cong \Lambda_{\F_2}(t) \otimes \Lambda_{\F_2}(v) \otimes \pi_* X \otimes \pi_*X \to \Lambda_{\F_2}(t) \otimes \Lambda_{\F_2}(v) \otimes \pi_* X \otimes \pi_*X[md-2]\]
at the level of Tor spectral sequences. 

The composite $\zeta \circ (z \wdgz x^m)$ is null-homotopic because the composite map  
\[\Sigma^{2m} H\Z \xrightarrow{x^m} X \to X[2m-2]\] is null-homotopic as $\pi_{2m}(X[2m-2]) = 0$, see Remark \ref{rmk homotopy and maps}. Therefore, the map of spectral sequences $\zeta_2$ should carry a class representing $z \wdgz x^m$ to the trivial element. 

 The only non-trivial elements of $F^2$ that have total degree $\lv z \wdgz x^m \rv = 2m+1$ and gets carried to $0$ are $1 \otimes v \otimes 1 \ot x^m$ and $t \otimes 1 \ot 1 \ot x^m$. Since we already showed that $1 \ot v \ot 1 \ot x^m$ do not represent $z \wdgz x^m$, and since we know that $z \wdg_{\hz}x^m$ corresponds to a non-trivial class in homotopy, we deduce that the homotopy class corresponding to $z \wdgz x^m$ is represented by 
 \begin{equation}\label{eq claim representing class}
     t \otimes 1 \otimes 1 \ot x^m + c_1(1 \otimes v \ot 1 \ot x^m)
 \end{equation}
in $F^2$ for some $c_1 \in \ft$. Note that the $c_1 \neq 0$ case is actually redundant as the first summand has a higher homological degree than the second. 

Therefore, to finish the proof of the lemma, it is sufficient to show that the class represented by the element above is carried to a non-trivial class in $H\F_2 \whz X$ by the map $id_{\hft} \wdgz m_X$. 

Note that for every $x_1,x_2 \in \pis X$, the map
\[\pi_*m_X \co \pi_*(X \whz X) \cong \Lambda_{\F_2}(v) \otimes \pi_* X \otimes \pi_*X \to  \pi_* X\]
carries $v \ot x_1 \ot x_2$ to $0$ because $\pis X$ is concentrated in even degrees. For a general $\hz$-algebra $X$, the product of classes $x_1,x_2 \in \pis(X)$ is defined to be the image of the class $1 \otimes x_1 \otimes x_2$ through the map above. Therefore, the map above should carry $1 \ot x_1 \ot x_2$ to $x_1x_2$.

We use Lemma \ref{lemma module homology} to calculate the map of spectral sequences induced by the map \[id_{H\F_2} \whz m_X \co H\F_2 \whz(X \whz X) \to H\F_2 \whz X.\] On the second page, we have 
\[id \ot \pis m_X \co F^2 \cong \Lambda_{\F_2}(t) \otimes \Lambda_{\F_2}(v) \otimes \pi_* X \otimes \pi_*X \to \Lambda_{\F_2}(t)  \otimes \pi_* X\]
where $id$ denotes the identity map of $\lambdaft(t)$ and  $\text{deg}(t) = (1,0)$  on both sides.  The image of the element in $F^2$ representing $z \wdgz x^m$ is given by 
\begin{equation}\label{eq claim class goes to nontrivial element}
    id \ot \pis m_X (t \otimes 1 \otimes 1 \ot x^m + c_1(1 \otimes v \ot 1 \ot x^m)) = t \ot x^m.
\end{equation}
This represents a non-trivial class in $\pis(\hft \wdgz X)$. Indeed, this finishes the proof of the lemma. Wrapping up, $z \wdgz x^m$ represents the class $z \ot 1 \ot 1 \ot x^m$ in 
\[\pis ((\hft \wdgz X)\wdgft(\hft \wdgz X)) \cong \lambdaft(z) \ot \pis X \ot \lambdaft(z) \ot \pis X  \]
 due to Diagram \eqref{eq 123}. To prove the lemma, we show that this class maps to a non-trivial element by the multiplication map of $\hft \wdgz X$. Equation \eqref{eq claim representing class} provides the class in $F^2$ representing $z \wdgz x^m$ and Equation \eqref{eq claim class goes to nontrivial element} states that this class goes to a non-trivial element by the multiplication map of $\hft \wdgz X$ as desired.
 \end{proof}


\begin{thebibliography}{10}

\bibitem{camarena2019simpleuniversal}
Omar Antol\'{\i}n-Camarena and Tobias Barthel.
\newblock A simple universal property of {T}hom ring spectra.
\newblock {\em J. Topol.}, 12(1):56--78, 2019.

\bibitem{basterra1999andre}
M.~Basterra.
\newblock Andr\'e-{Q}uillen cohomology of commutative {$S$}-algebras.
\newblock {\em J. Pure Appl. Algebra}, 144(2):111--143, 1999.

\bibitem{basterra2013BP}
Maria Basterra and Michael~A. Mandell.
\newblock The multiplication on {BP}.
\newblock {\em J. Topol.}, 6(2):285--310, 2013.

\bibitem{bokstedt1985topological}
Marcel B{\"o}kstedt.
\newblock {\em The topological Hochschild homology of $\Z$ and $\Z/p$}.
\newblock Preprint, 1987.

\bibitem{chuang2018derivedlocalisations}
C.~Braun, J.~Chuang, and A.~Lazarev.
\newblock Derived localisation of algebras and modules.
\newblock {\em Adv. Math.}, 328:555--622, 2018.

\bibitem{dugger2006spectral}
Daniel Dugger.
\newblock Spectral enrichments of model categories.
\newblock {\em Homology, Homotopy and Applications}, 8(1):1--30, 2006.

\bibitem{dugger2006postnikov}
Daniel Dugger and Brooke Shipley.
\newblock Postnikov extensions of ring spectra.
\newblock {\em Algebr. Geom. Topol.}, 6:1785--1829, 2006.

\bibitem{dugger2007additiveendomorphism}
Daniel Dugger and Brooke Shipley.
\newblock Enriched model categories and an application to additive endomorphism
  spectra.
\newblock {\em Theory Appl. Categ.}, 18:No. 15, 400--439, 2007.

\bibitem{dugger2007topological}
Daniel Dugger and Brooke Shipley.
\newblock Topological equivalences for differential graded algebras.
\newblock {\em Advances in Mathematics}, 212(1):37--61, 2007.

\bibitem{dwyer2013dg}
WG~Dwyer, John~PC Greenlees, and Srikanth~B Iyengar.
\newblock D{G} algebras with exterior homology.
\newblock {\em Bulletin of the London Mathematical Society}, 45(6):1235--1245,
  2013.

\bibitem{elmendorf2007rings}
A.~D. Elmendorf, I.~Kriz, M.~A. Mandell, and J.~P. May.
\newblock {\em Rings, modules, and algebras in stable homotopy theory},
  volume~47 of {\em Mathematical Surveys and Monographs}.
\newblock American Mathematical Society, Providence, RI, 1997.
\newblock With an appendix by M. Cole.

\bibitem{francis2013thetangentcomplex}
John Francis.
\newblock The tangent complex and {H}ochschild cohomology of {$ E_n$}-rings.
\newblock {\em Compos. Math.}, 149(3):430--480, 2013.

\bibitem{hinich2015rectificationofalgebrasandmodules}
Vladimir Hinich.
\newblock Rectification of algebras and modules.
\newblock {\em Doc. Math.}, 20:879--926, 2015.

\bibitem{Hovey1999book}
Mark Hovey.
\newblock {\em Model categories}.
\newblock Mathematical Surveys and Monographs 63. American Mathematical
  Society, Providence, RI, 1999.

\bibitem{Hovey-Shipley-Smith}
Mark Hovey, Brooke Shipley, and Jeff Smith.
\newblock Symmetric spectra.
\newblock {\em J. Amer. Math. Soc.}, 13(1):149--208, 2000.

\bibitem{hunter1996thhss}
Thomas~J. Hunter.
\newblock On the homology spectral sequence for topological {H}ochschild
  homology.
\newblock {\em Trans. Amer. Math. Soc.}, 348(10):3941--3953, 1996.

\bibitem{johnson2014lifting}
Niles Johnson and Justin Noel.
\newblock Lifting homotopy {T-algebra} maps to strict maps.
\newblock {\em Advances in Mathematics}, 264:593--645, 2014.

\bibitem{larsen1992larsen}
M.~Larsen and A.~Lindenstrauss.
\newblock Cyclic homology of {D}edekind domains.
\newblock {\em $K$-Theory}, 6(4):301--334, 1992.

\bibitem{lazarev2004cohomology}
Andrey Lazarev.
\newblock Cohomology theories for highly structured ring spectra.
\newblock {\em Structured ring spectra}, 315:201--231, 2004.

\bibitem{lurie2012higher}
Jacob Lurie.
\newblock Higher algebra.
\newblock https://www.math.ias.edu/~lurie/papers/HA.pdf, 2017.
\newblock Electronic book.

\bibitem{Mandell01Model}
M.~A. Mandell, J.~P. May, S.~Schwede, and B.~Shipley.
\newblock Model categories of diagram spectra.
\newblock {\em Proc. London Math. Soc. (3)}, 82(2):441--512, 2001.

\bibitem{mathew2015nilpotenceconfecture}
Akhil Mathew, Niko Naumann, and Justin Noel.
\newblock On a nilpotence conjecture of {J}. {P}. {M}ay.
\newblock {\em J. Topol.}, 8(4):917--932, 2015.

\bibitem{mcclure1993onthethhbu}
J.~E. McClure and R.~E. Staffeldt.
\newblock On the topological {H}ochschild homology of {$b{\rm u}$}. {I}.
\newblock {\em Amer. J. Math.}, 115(1):1--45, 1993.

\bibitem{milnor1958steenrod}
John Milnor.
\newblock The {S}teenrod algebra and its dual.
\newblock {\em Annals of Mathematics}, 67(1):150--171, 1958.

\bibitem{pavlov2019symmetricopSsp}
Dmitri Pavlov and Jakob Scholbach.
\newblock Symmetric operads in abstract symmetric spectra.
\newblock {\em J. Inst. Math. Jussieu}, 18(4):707--758, 2019.

\bibitem{peroux2020coalgebras}
Maximilien P{\'e}roux.
\newblock Coalgebras in the {D}wyer-{K}an localization of a model category.
\newblock {\em arXiv preprint arXiv:2006.09407}, 2020.

\bibitem{richter2014algebraic}
Birgit Richter and Brooke Shipley.
\newblock An algebraic model for commutative {$H\mathbb{Z}$}-algebras.
\newblock {\em Algebr. Geom. Topol.}, 17(4):2013--2038, 2017.

\bibitem{serre1979localfields}
Jean-Pierre Serre.
\newblock {\em Local fields}, volume~67 of {\em Graduate Texts in Mathematics}.
\newblock Springer-Verlag, New York-Berlin, 1979.
\newblock Translated from the French by Marvin Jay Greenberg.

\bibitem{shipley2007hz}
Brooke Shipley.
\newblock {$H\mathbb{Z}$}-algebra spectra are differential graded algebras.
\newblock {\em American journal of mathematics}, 129(2):351--379, 2007.

\bibitem{stanley1997dissertation}
Donald~Westgarth Stanley.
\newblock {\em Closed model categories and monoidal categories}.
\newblock ProQuest LLC, Ann Arbor, MI, 1997.
\newblock Thesis (Ph.D.)--University of Toronto (Canada).

\end{thebibliography}
\end{document}